\documentclass[hidelinks,onefignum,onetabnum]{siamart220329}

\usepackage{lipsum}
\usepackage{amsfonts}
\usepackage{graphicx}
\usepackage{epstopdf}
\usepackage{algorithmic}
\ifpdf
  \DeclareGraphicsExtensions{.eps,.pdf,.png,.jpg}
\else
  \DeclareGraphicsExtensions{.eps}
\fi

\newsiamremark{remark}{Remark}
\newsiamremark{hypothesis}{Hypothesis}
\crefname{hypothesis}{Hypothesis}{Hypotheses}
\newsiamthm{claim}{Claim}

\headers{Homogeneous Control Systems on Cones}{A. Polyakov  and M. Krstic}

\title{Homogeneous Control Systems on Cones and Nonovershooting Finite-time Stabilizers}

\author{Andrey Polyakov\thanks{University of Lille, Inria, CNRS,  France
  (\email{andrey.polyakov@inria.fr}).}
\and Miroslav Krstic \thanks{University of California, San Diego, USA,
  (\email{krstic@ucsd.edu}).}}

\usepackage{amsopn}

\usepackage{alltt}
\usepackage{color}
\usepackage{amsmath,amssymb}
\usepackage{graphicx}          % Include this line if your
\usepackage{bbm}

\definecolor{green2}{rgb}{0,0.62,0.17}
\definecolor{orange}{rgb}{1,0.5,0.0}
\definecolor{b}{rgb}{0,0,1}
\definecolor{r}{rgb}{1,0,0}

\newtheorem{assumption}{Assumption}
\newcommand{\R}{\mathbb{R}}

\newcommand{\D}{\mathcal{D}}
\newcommand{\sign}{\mathrm{sign}}
\newcommand{\interior}{\mathfrak{int}\,}
\newcommand{\zero}{\mathbf{0}}

\newcommand{\dn}{\mathbf{d}}

\begin{document}
	\maketitle

	\begin{abstract}
		A nonovershooting finite-time control design for linear multi-input system is proposed by upgrading a linear (asymptotic) nonovershooting stabilizer to a  homogeneous one. Robustness of the safety and stability  properties is analyzed using 
		the concept of Input-to-State Stability (ISS) on invariant sets and Input-to-State Safety (ISSf).  Theoretical results are  illustrated on numerical examples. 
	\end{abstract}

	\section{Introduction}
	A  tracking problem under certain state/output constrains  can be solved using the so-called nonovershooting control \cite{KrsticBement2006:TAC} which, cast in the framework of control barrier functions \cite{WielandAllgower2007:NOLCOS}, \cite{Ames_etal2017:TAC}, \cite{Jankovic2018:Aut},
	can be also employed for a ''safety filter'' design, to override  a potentially unsafe nominal controller  \cite{Abel_etal2022:ACC,PolyakovKrstic2023:TAC}.
	Linear and nonlinear nonovershooting controllers have been designed for both linear \cite{PhillipsSeborg1988:IJC}, \cite{ElKhoury_etal1993:Aut}, \cite{DarbhaBhattacharrya2003:TAC} and nonlinear systems \cite{KrsticBement2006:TAC}, \cite{LindemannDimarogonas2019:IEEECSL}, \cite{Garg_etal2022:IEEE_CSL}.
	The safety filter synthesis based on control barrier functions (CBF) has been extensively used in  control applications such as automotive systems \cite{Ames_etal2015:CDC}, \cite{Rahman_etal2021:ACC} and multi-agent robotics \cite{Wang_etal2017:TR}, \cite{SantilloJankovic2021:ACC}.   The barrier functions  characterize positively invariant sets of safe control systems. 
	
	The positive invariance of closed sets  for systems having unique solutions was first characterized in \cite{Nagumno1942:PPMSJ}.  Control systems design (analysis) on invariant sets  usually deals with compacts including a set-point (stable equilibrium) in  its interior (as, for example, in  \cite{Blanchini1999:Aut}).  The ellipsoidal and polyhedral invariant sets are the most popular in this context \cite{BlanchiniMiani2016:Book}, \cite{Poznyak_etal2014:Book}. 
	For a nonovershooting stabilizer design, the desired set-point belongs to the boundary of  the safe set.   For linear systems,  the corresponding invariant sets are linear positive cones  and the nonovershooting property can be established by means of a transformation of the original system to a positive linear system \cite{FarinaRinaldi2000:Book}, \cite{LeenheerAeyels2001:SCL}, \cite{Rantzer2015:EJC} as, for example, in \cite{PolyakovKrstic2023:TAC}.
	A nonlinear nonovershooting stabilizer design may be required even for linear control plant if, for example, some time constraints have to be fulfilled.  Usually, the corresponding invariant set is not a linear cone and the related analysis is more complicated.  This paper deals with a class of homogeneous systems, which admits a nonovershooting analysis  similar to linear one. 
	
	Homogeneous system is  a  nonlinear system \cite{Zubov1958:IVM}, \cite{Kawski1991:ACDS}, \cite{Rosier1992:SCL}, \cite{Grune2000:SIAM_JCO}, which, on the one hand, have many properties typical for linear systems such as equivalence of local and global results \cite{BhatBernstein2005:MCSS} or equivalence of asymptotic stability and Input-to-State Stability (ISS) with respect to homogeneously (resp., linearly)  involved perturbations \cite{Ryan1995:SCL}, \cite{Hong2001:Aut}.  
	On the other hand, they may demonstrate faster convergence \cite{BhatBernstein2005:MCSS}, better robustness \cite{Andrieu_etal2008:SIAM_JCO} and smaller overshoots \cite[Chapter 1]{Polyakov2020:Book}, \cite{PolyakovKrstic2023:TAC}. A  
	convergence rate of any stable homogeneous system is characterized by its homogeneity degree \cite{Nakamura_etal2002:SICE}.
	Homogeneity is a dilation symmetry widely studied in the group theory \cite{FischerRuzhansky2016:Book} and useful for control systems design \cite{Polyakov2020:Book}. Similarly to linear systems, invariant sets of homogeneous systems may be of a conic type (in a generalized sense).  Therefore, analysis and design of nonovershooting homogeneous control systems should be developed on some homogeneous cones. This paper develops some tools for stability and robustness analysis of homogeneous systems on cones and provides a simple  scheme of a nonovershooting  finite-time stabilizer design for linear time-invariant multi-input plants.

	The main contributions of the paper are as follows:
%	\begin{itemize}
	   % \item 
	    the homogeneous Lyapunov function theorem \cite{Rosier1992:SCL} is refined for homogeneous systems on cones;
	    %\item 
	    the criterion of (robust) positive invariance \cite{Blanchini1999:Aut}, \cite{Aubin1991:Book} is adapted to homogeneous cones;
	    %\item 
	    the homogeneous ISS theorem \cite{Ryan1995:SCL}, \cite{Andrieu_etal2008:SIAM_JCO} is extended to systems on cones;
	    %\item 
	    an algebraic condition of ISSf of homogeneous systems is obtained;
	    %\item 
	    an algorithm of non-overshooting finite-time stabilizer design is proposed for linear multi-input systems with linear conic constraints; 
	    %\item 
	    ISS and ISSf of the obtained homogeneous non-overshooting control is proven.
	%\end{itemize}
 The paper is organized as follows. First, the problem statement and notions of interest are discussed. Next, some preliminaries about homogeneous systems are presented. After that, the issues of stability and robustness analysis of homogeneous systems on cones are studied and the nonovershooting homogeneous control is designed for linear plant. Finally, numerical examples and concluding remarks are given.

	\textit{Notation}.
	%\begin{itemize}
		%\item 
		$\R$ is the field of reals, $\R_+=\{x\in \R: x\geq 0\}$ and  $\R^{n}$ is the Euclidean space of column real vectors;
		%\item 
		$\|\cdot\|$ denotes a norm in $\R^n$ and $|\cdot|$ is the Euclidean norm in $\R^n$; we also use  $\|\cdot\|_{\R^k}$ in order to highlight the dimension of the Euclidean space; 
		%\item
		 $e_i=(0,\ldots,0,1,0,\ldots,0)^{\top}$ is the unit vector of the Euclidean basis in $\R^n$;
		%\item 
		$\partial \Omega$ denotes a boundary of $\Omega\subset \R^n$ and $\interior \Omega$ is the interior of $\Omega$;
		%\item 
		$C(\Omega,\R^m)$ denotes a set of continuous mappings $\Omega\subset \R^n \mapsto \R^m$, where $\Omega$ is a connected set with non-empty interior;
		%\item 
		$C^1(\Omega,\R^m)$ is a subset of mappings  $C(\Omega,\R^m)$, which are continuously differentiable on the interior of $\Omega$ such that all partial derivatives have a continuous prolongation  to $\partial \Omega\cap \Omega$; we write shortly $C(\Omega)$ and $C^1(\Omega)$ if a context is clear or a dimension of the codomain is not important;
		%\item 
		$L^{\infty}(\R,\R^k)$ is a set of uniformly essentially bounded functions and  $\|q\|_{L^{\infty}(t_0,t_1)}=\mathrm{ess}\sup_{t\in(t_0,t_1)} |q(t)|;$ 
		%\item
		$\mathcal{K}$ denotes the class of continuous strictly increasing functions $\sigma:[0,+\infty)\mapsto [0,+\infty)$ such that $\sigma(0)=0$; 
		a  function $\sigma\in \mathcal{K}$ is of the class $\mathcal{K}_{\infty}$
		if $\sigma(s)\to+\infty$ as $s\to +\infty$;
		%\item 
		$\mathcal{KL}$ denotes the class of continuous  functions $\beta:[0,+\infty)\times [0,+\infty)\mapsto [0,+\infty)$ such that 
		$\beta(\cdot,t)$ is of class $\mathcal{K}$ for any $t\geq 0$ and the function $\beta(r,\cdot):[0,+\infty)\mapsto [0,+\infty]$ is decreasing such that $\beta(r,t)\to 0$ as $t\to +\infty$, for any $r\in [0,+\infty)$;
		%\item
		 for a symmetric matrix $P=P^{\top}\in \R^{n\times n}$, the order relation $P\succ0$ (resp., $P\prec 0$) means that the matrix $P$ is positive (resp., negative) definite.
%	\end{itemize}
	%%%%%%%%%%%%%%%%%%%%%%%
	%%%%%%%%%%%%%%%%%%%%%%%
	%%%%%%%%%%%%%%%%%%%%%%%
	\section{Problem Statement}
	Let us consider the system
	\begin{equation}\label{eq:p_hom_system}
		\dot x=f(x,q), \quad t>t_0,\quad x(t_0)= x_0,%\quad {\color{blue} q(t)\in \mathcal{Q}},
	\end{equation}
	where $x(t)\in \R^n$ is the system state,  $q\in L^{\infty}(\R,\R^k)$ is the exogenous input, %{\color{blue} $\mathcal{Q}\subseteq\R^k$}
	and $f\in C(\R^{n}\times \R^k,\R^n)$.
	We consider the system \eqref{eq:p_hom_system} on  a closed set  
	\begin{equation}\label{eq:inv_set}
		\Omega=\{x\in \R^n\;:\; \phi_i(x)\geq 0,\;i=1,\ldots,p \},
	\end{equation}	
	with non-empty interior,
	where $\phi_i\in C(\R^n, \R)\cap C^1(\R^n\backslash\{\zero\}, \R)$, $i=1,\ldots,p$, and study an asymptotic (or a finite-time)  stability of the set 
	\begin{equation}\label{eq:Omega_0}
		\Omega_{0}=\{x\in \R^n: \phi_i(x)=0,\;i=1,2,\ldots,p\}
	\end{equation}  
	as well as  Input-to-State Stability (ISS) of the system \eqref{eq:p_hom_system} on $\Omega$.
	In practice, the set $\Omega$ may define a safe set  of the system  \eqref{eq:p_hom_system}. The functions $\phi_i$ are called the barrier functions in this case  \cite{Ames_etal2017:TAC}, \cite{Jankovic2018:Aut}.
	For example, if  the system \eqref{eq:p_hom_system} is globally asymptotically stable and  the set $\Omega$ with $\phi_1(x)=x_1,p=1$ is  positively invariant\footnote{The set $\Omega$ is positively invariant for the system \eqref{eq:p_hom_system} if the following implication $x(t^*)\!\in\! \overline{\Omega} \Rightarrow x(t)\!\in\! \Omega,\forall t\!>\!t^*$ takes a place  for any solution of the system \eqref{eq:p_hom_system}.} for the system \eqref{eq:p_hom_system}, then, obviously, all trajectories of the system initialized in $\Omega$ converge to $\Omega_0$ without overshoot in the first coordinate.

	Below we assume that the set $\Omega$ has a topological characterization of  a positive (in a generalized sense)  cone. Recall that the (conventional) positive cone in a vector space  is a set $\Xi$ such that  $\lambda x\in \Xi$,  $\forall \lambda>0$, $\forall x\in \Xi$. The positive cone $\Xi$  is  linear if $x+y\in \Xi, \forall x,y\in \Xi$. 
	Notice that the multiplication of a vector $x$ by a positive scalar $\lambda>0$ is the standard  dilation in the vector space.  The set $\Omega$ defines a standard positive cone in $\R^n$ if  the functions $\phi_i$ are standard homogeneous $\phi_i(\lambda x)=\lambda^{\nu_i} \phi_i(x),\forall \lambda>0,\forall x\in \R^n$, where $\nu_i>0$ is the so-called homogeneity degree. To design  a generalized positive cone in $\R^n$, a generalized dilation \cite{Zubov1958:IVM}, \cite{Khomenuk1961:IVM}, \cite{Rosier1992:SCL} can be utilized.  Topological characterization of generalized dilations in Frech\'et and Euclidean spaces can be found in \cite{Husch1970:Math_Ann} and in \cite{Kawski1991:ACDS}, respectively. In this paper, we deal only with a cone $\Omega$ induced by  means of the so-called linear dilation in $\R^n$  \cite{Polyakov2018:RNC}, \cite{FischerRuzhansky2016:Book}. The corresponding cone is usually nonlinear in $\R^n$, but it may become linear in a vector space homeomorphic to $\R^n$ (see Section \ref{sec:pre_hom}). The system \eqref{eq:p_hom_system} is assumed to be  generalized homogeneous (i.e., symmetric with respect to a generalized dilation) as well. 
	
	\begin{definition}
		The system \eqref{eq:p_hom_system}  is said to be
		uniformly asymptotically stable  on $\Omega$ if 
		there exist $\beta_i\in \mathcal{KL}$  such that
		\begin{equation}\label{eq:GAS}
			0\leq \phi_i(x(t)) \leq \beta_i(|x_0|,t-t_0),\quad 
		\end{equation}
		for all $t\geq t_0$, $x_0\in \Omega$ and $i=1,\ldots,p$.	
	\end{definition}
	This definition introduces  the asymptotic stability of the set $\Omega_0$ belonging to the boundary of  the invariant set $\Omega$. Such a requirement is typical for the design of nonovershooting stabilizers and safety filters \cite{KrsticBement2006:TAC}, \cite{WielandAllgower2007:NOLCOS}, \cite{Abel_etal2022:ACC}, \cite{PolyakovKrstic2023:TAC}. \textit{Notice that if $\Omega$ is a homogeneous cone and the set $\Omega_0$ is bounded then $\Omega_0=\{\zero\}$} (see Section \ref{sec:pre_hom} for more details). We study the robustness of the stability property of the system \eqref{eq:p_hom_system} on the cone $\Omega$ in the sense of the following definition inspired by \cite{Sontag1989:TAC},  \cite{Magni_etal2006:TAC}.
	%  .we introduce the notions of  the input-to-state safety and stability of the system \eqref{eq:p_hom_system}  on a cone $\Omega$.
	\begin{definition}
		The system \eqref{eq:p_hom_system} is said to be 
		ISS on $\Omega$ if 
		there exist $\beta_i\in \mathcal{KL}$  and $\gamma_i\in \mathcal{K}$ such that
		\begin{equation}\label{eq:ISS_Omega}
			0\leq \phi_i(x(t)) \leq \beta_i(|x_0|,t-t_0)+\gamma_i(\|q\|_{L^{\infty}_{(t_0,t)}}), \quad
		\end{equation}
		for all $t\!\geq\! t_0$, $x_0\!\in\! \Omega$, $q\!\in\! L^{\infty}(\R,\R^k)$ and $i\!=\!1,\ldots,p$.
	\end{definition}
 
	Our first aim is to obtain a characterization of  ISS on  $\Omega$ for a class of (generalized) homogeneous systems. 
	The condition \eqref{eq:ISS_Omega}  looks similar to input-to-output stability (IOS) (see, e.g.,  \cite{Sontag1997:ECC}). In the general case, IOS is neither necessary nor sufficient for ISS on $\Omega$. 
	However, the ISS on $\Omega$ may be equivalent to the regional ISS \cite{Magni_etal2006:TAC} provided that $\Omega$ is a compact and $x\to \zero \Leftrightarrow \phi_i(x)\to \zero,\forall i=1,\ldots,p$. The left  inequality in  \eqref{eq:ISS_Omega} simply means the set $\Omega$ is positively invariant for the system \eqref{eq:p_hom_system} independently of $q\in L^{\infty}$. 
	 This condition may be very restrictive in many practical cases. The concept of the Input-to-State Safety  \cite{KolathayaAmes2019:CSL} relaxes the invariance of $\Omega$ as follows.
		
		\begin{definition}\label{def:ISSf}
			The system \eqref{eq:p_hom_system} is said to be 
			Input-to-State Safe (ISSf) on $\Omega$ if 
			there exist $\gamma_i\in \mathcal{K}$ such that
			\begin{equation}\label{eq:ISSF}
				-\gamma_i\left(\|q\|_{L^{\infty}_{(t_0,t)}}\right)\leq \phi_i(x(t))
			\end{equation}
			for all $t\!\geq\! t_0$, $x_0\!\in\! \Omega$, $q\!\in\! L^{\infty}(\R,\R^k)$ and $i\!=\!1,\ldots,p$.
		\end{definition}
		
  The inequality \eqref{eq:ISSF} simply means that the overshoot of the perturbed system must depend continuously on the magnitude of the perturbation. For robustness analysis of nonovershooting stabilizers, ISS on $\Omega$ and ISSf on $\Omega$ are merged as follows.
  
		\begin{definition}
			The system \eqref{eq:p_hom_system} is said to be 
			Input-to-State Safe and Stable (ISSfS) on $\Omega$ if 
			there exist $\beta_i\in \mathcal{KL}$  and $\gamma_i\in \mathcal{K}$ such that
			\begin{equation}\label{eq:ISSS}
				-\gamma_i(\|q\|_{L^{\infty}_{(t_0,t)}})\leq \phi_i(x(t)) \leq \beta_i(|x_0|,t-t_0)+\gamma_i(\|q\|_{L^{\infty}_{(t_0,t)}}), \quad
			\end{equation}
			for all $t\!\geq\! t_0$, $x_0\!\in\! \Omega$, $q\!\in\! L^{\infty}(\R,\R^k)$ and $i\!=\!1,\ldots,p$.
			%, where   $\phi\!=\!\left(\begin{array}{ccc}\phi_1,\ldots,\phi_{p}\end{array}\right)^{\top}$.
		\end{definition}
  
		Our second aim  is to characterize the robustness of both safety and stability properties of the system \eqref{eq:p_hom_system} on the cone $\Omega$ in the sense of the latter definition.
		
		The final goal of the paper is to design
		a  nonovershooting finite-time stabilizer for a nonlinear generalized homogeneous system using  its unperturbed linear model,
\begin{equation*}
\dot x=f(x,\zero)=Ax+Bu(x),
\end{equation*}
 where $x$ is as before, $A\in \R^{n\times n}$, $B\in \R^{n\times m}$ and  $u:\R^{n}\mapsto \R^{m}$ is a nonovershooting finite-time stabilizer to be designed for a certain class of generalized homogeneous cones $\Omega$ defined below.  	The robustness (with respect to additive perturbations and measurement noises)  of stability and safety properties of the closed-loop system is going to be studied in the sense of  the above definitions.

	%%%%%%%%%%%%%%%%%%%%%%%%%%%%%%%%%%%%%%%%%%%%%%%%%%%%%%%%%%%%%%%
	%%%%%%%%%%%%%%%%%%%%%%%%%%%%%%%%%%%%%%%%%%%%%%%%%%%%%%%%%%%%%%%
	%%%%%%%%%%%%%%%%%%%%%%%%%%%%%%%%%%%%%%%%%%%%%%%%%%%%%%%%%%%%%%%
	\section{Generalized Homogeneity in $\R^n$}\label{sec:pre_hom}
	\subsection{Linear dilations}
	
	Let us recall  that \textit{a family of  operators} $\dn(s):\R^n\mapsto \R^n$ with $s\in \R$ is  a \textit{group} if
	%\begin{itemize}
	%\item 
	$\dn(0)x\!=\!x$, $\dn(s)\!\circ\! \dn(t) x\!=\!\dn(s\!+\!t)x$, $\forall x\!\in\!\R^n, \forall s,t\!\in\!\R$.
	%\end{itemize}	
	A \textit{group} $\dn$ is 
	%\begin{itemize}		
		%\item
		a) \textit{continuous} if the mapping $s\mapsto \dn(s)x$ is continuous,  $\forall x\!\in\! \R^n$;
		%\item
		  b) \textit{linear} if $\dn(s)$ is a linear mapping (i.e., $\dn(s)\in \R^{n\times n}$), $\forall s\in \R$;
		%\item 
		c)  a \textit{dilation} in $\R^n$ if $\liminf\limits_{s\to +\infty}\|\dn(s)x\|=+\infty$ and $\limsup\limits_{s\to -\infty}\|\dn(s)x\|=0$,  $\forall x\neq \zero$.
%	\end{itemize}
	Any linear continuous group in $\R^n$ admits the representation  \cite{Pazy1983:Book}
	\begin{equation}\label{eq:Gd}
		\dn(s)=e^{sG_{\dn}}=\sum_{j=1}^{\infty}\tfrac{s^jG_{\dn}^j}{j!}, \quad s\in \R,
	\end{equation}
	where $G_{\dn}\in \R^{n\times n}$ is a generator of $\dn$. A continuous linear group \eqref{eq:Gd} is a dilation in $\R^n$ if and only if $G_{\dn}$ is an anti-Hurwitz matrix \cite{Polyakov2018:RNC}. In this paper we deal only with continuous linear dilations. 
	A \textit{dilation} $\dn$ in $\R^n$ is
	%\begin{itemize}
		%\item
		a) \textit{monotone} if the function $s\mapsto \|\dn(s)x\|$ is strictly increasing,  $\forall x\neq \zero$;
		%\item 
	b) 	\textit{strictly monotone} if  $\exists \beta\!>\!0$ such that $\|\dn(s)x\|\!\leq\! e^{\beta s}\|x\|$, $\forall s\!\leq\! 0$, $\forall x\in \R^n$.
	%\end{itemize}
	The following result is the straightforward consequence of the existence of the quadratic Lyapunov function  for asymptotically stable LTI systems.
	\begin{corollary}\label{cor:monotonicity}
		A linear continuous dilation in $\R^n$ is strictly monotone with respect to the weighted Euclidean norm $\|x\|=\sqrt{x^{\top} Px}, 0\prec P\in \R^{n\times n}$ if and only if 
		\begin{equation}\label{eq:mon_d_P}
			PG_{\dn}+G_{\dn}^{\top}P\succ 0, \quad P\succ 0.
		\end{equation}
	\end{corollary}
 
	Any dilation in $\R^n$ defines an alternative topology (balls, spheres, cones, etc) in $\R^n$. Below we study the systems on generalized homogeneous cones.
	\begin{definition}
		Let $\dn$ be a dilation in $\R^n$. A nonempty set $\D\subset \R^n$ is said to be $\dn$-homogeneous  cone if 
		$\dn(s)\D\subset \D$ for all $s\in \R$.
	\end{definition}
 
	A standard homogeneous cone (i.e., $\dn_1(s)=e^s I_n$) and the generalized homogeneous cone 
	are  illustrated on Fig. \ref{fig:cones}.
	\begin{figure}
		\centering
		\includegraphics[width=53mm]{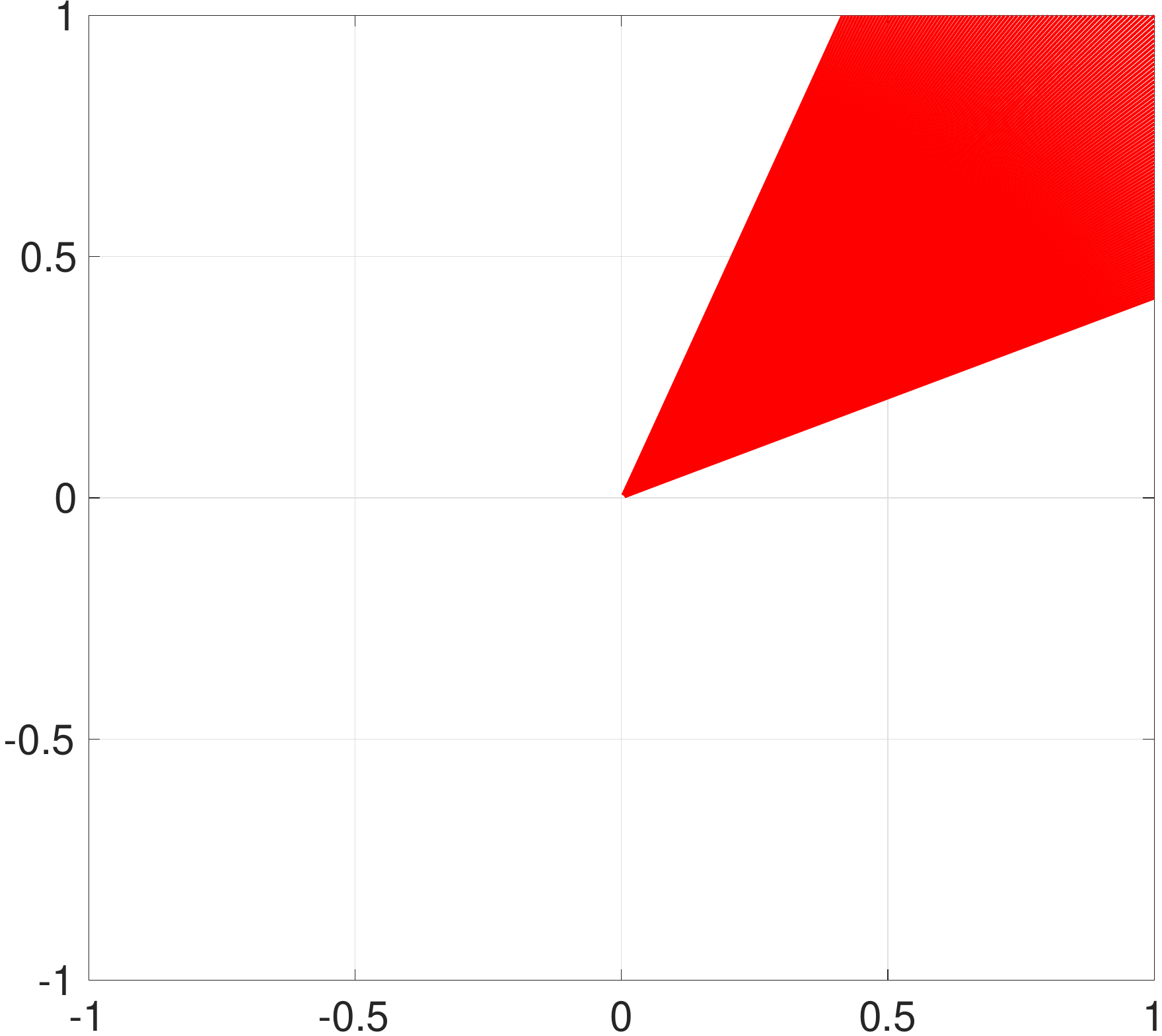}   
		\includegraphics[width=52mm]{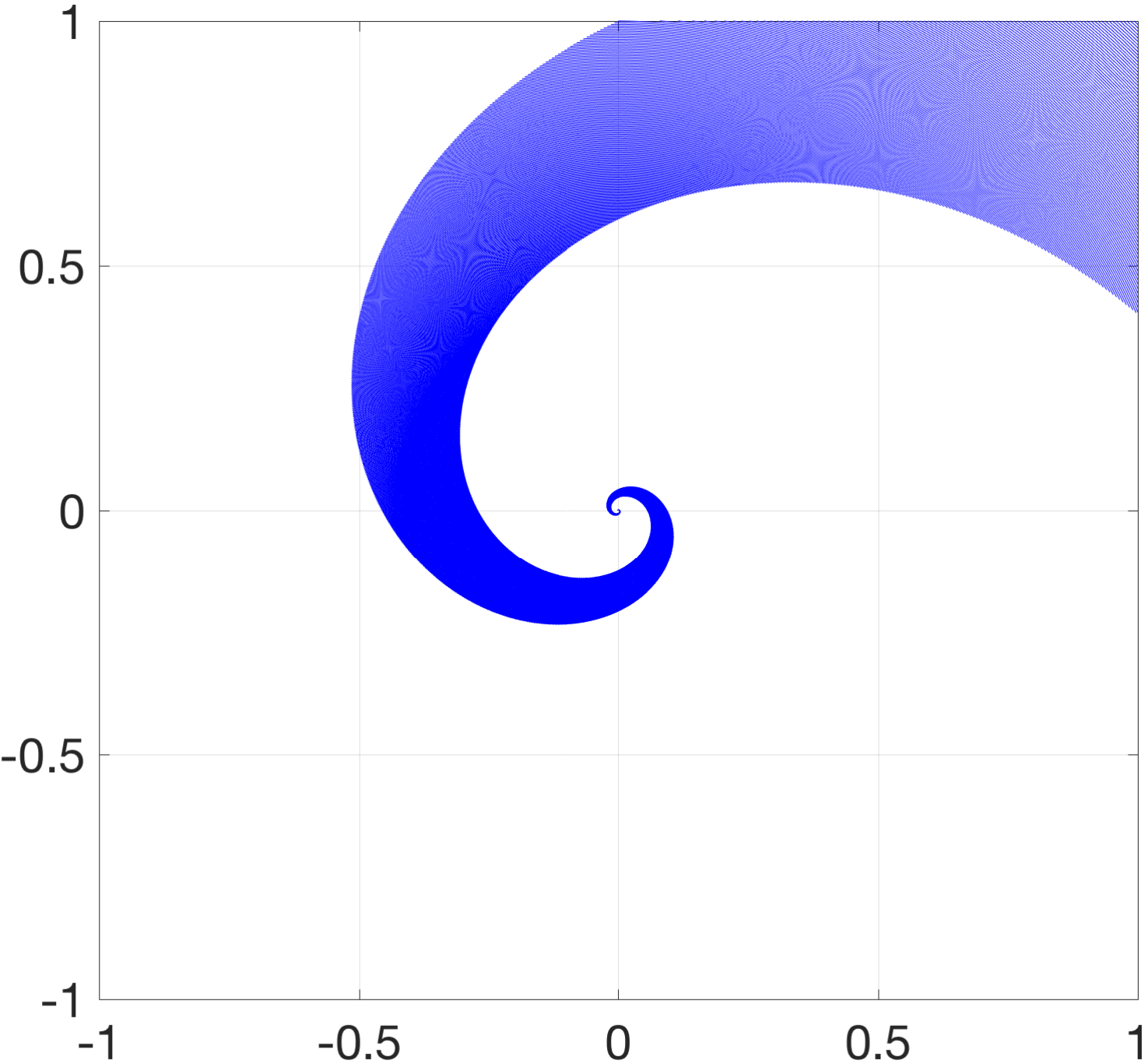}  
		\caption{Standard and generalized homogeneous cones}
		\label{fig:cones}
	\end{figure}
	Notice that if $\D\subset \R^n$ is a non-empty closed $\dn$-homogeneous cone  then $\zero \in \D$. Indeed, 
	if $x\in \D, x\neq \zero$ then $\dn(s)x\in \D,\forall s\in \R$ and $\dn(s)x\to \zero$ as $s\to -\infty$. Since $\D$
	is closed then the limit point  $\zero$  belongs to $\D$. 	Below we deal only with a linear continuous dilation $\dn$ in $\R^n$.

	\subsection{Canonical homogeneous norm}
	Any linear  continuous and monotone  dilation in a normed vector space introduces also an alternative norm topology defined by the so-called canonical homogeneous norm \cite{Polyakov2018:RNC}.
	\begin{definition}[{\small Canonical homogeneous norm}]%\hfill\newline
		\label{def:hom_norm_Rn}
		Let a linear dilation $\dn$ in $\R^n$ be  continuous and monotone with respect to a norm $\|\cdot\|$.
		A function $\|\cdot\|_{\dn} : \R^n \mapsto [0,+\infty)$ defined as  follows: $\|\zero\|_{\dn}=0$ and 
		\begin{equation}\label{eq:hom_norm_Rn}
			\|x\|_{\dn}=e^{s_x}, \;  \text{where} \; s_x\in \R: \|\dn(-s_x)x\|=1, \quad x\neq \zero
			\vspace{-1mm}
		\end{equation}
		is said to be a canonical $\dn$-homogeneous norm \index{canonical homogeneous norm} in  $\R^n$ 
	\end{definition}	 
 
	For standard dilation $\dn_1(s)=e^{s}I_n$ we, obviously, have $\|x\|_{\dn_1}=\|x\|$. In other cases, $\|x\|_{\dn}$ with $x\neq \zero$ is implicitly defined by a nonlinear algebraic equation, which always have a unique solution due to monotonicity of the dilation. In some particular cases \cite{PolyakovKrstic2023:TAC}, this implicit equation has explicit solution even for non-standard dilations. 
	%Since   $\|\dn(-\ln \|x\|_{\dn})x\|=1$ for $x\neq 0$  then the operator 
	%	\begin{equation}
	%		\pi_{\dn}(x):=
	%		\left\{
	%		\begin{array}{ccc}
	%			\dn(-\ln \|x\|_{\dn})x & \text{ if } & x\neq \zero,\\
	%			\zero & \text{ if } & x=\zero,
	%		\end{array}
	%		\right.
%		\end{equation}
%		is  a projector of a nonzero vector $x$ on the unit  sphere $\|x\|=1$.  For the standard dilation, such a projector is defined as $\pi_{\dn_1}(x)=\dn_1(-\ln \|x\|_{\dn_1})x=\frac{x}{\|x\|}$ for $x\neq \zero$.  We have $\pi_{\dn}(x)=\sign(x)$ in the scalar case $n=1$.
	\begin{lemma}\cite{Polyakov2018:RNC}\label{lem:hom_norm}
		If  a linear continuous dilation $\dn$ in $\R^n$  is  monotone with respect to a norm $\|\cdot\|$
		then
		\begin{itemize}
			\item[1)]   $\|\cdot\|_{\dn} : \R^n\mapsto \R_+$ is single-valued and continuous on $\R^n$;
			\item[2)]
			there exist $\sigma_1,\sigma_2 
			\in \mathcal{K}_{\infty}$
			such that 
			\begin{equation}\label{eq:rel_norm_and_hom_norm_Rn}
				\sigma_1(\|x\|_{\dn})\leq \|x\|\leq \sigma_2(\|x\|_{\dn}), \quad \quad 
				\forall x\in \R^n;
			\end{equation}
			\item[3)]   $\|\cdot\|$ is locally Lipschitz continuous on $\R^{n}\backslash\{\zero\}$ provided that  the linear dilation $\dn$ is strictly monotone
			;%, moreover,
			%\begin{equation}\label{eq:hom_norm_alpha}
		%		\left| \|x_1\|_{\dn}^{\alpha}-\|x_2\|_{\dn}^{\alpha}\right|\leq  \|x_1-x_2\|, \quad \forall x_1,x_2\in B,
	%		\end{equation}
	%		\begin{equation}\label{eq:hom_norm_beta}
	%			\left| \|x_1\|_{\dn}^{\beta}-\|x_2\|_{\dn}^\beta\right|\leq \|x_1-x_2\|, \quad \forall x_1,x_2\in \R^n\backslash B,
		%	\end{equation}
		%	where $\alpha =\|G_{\dn}\|$, $\beta>0$ is introduced in the definition of strict monotonicity (see above) and $B=\{x\!\in\!\R^n: \|x\|\!\leq\! 1\}$ is the  unit ball in $\R^n$.
			\item[4)] $\|\cdot\|_{\dn}$ is continuously differentiable on $\R^n\backslash\{\zero\}$ provided that $\|\cdot\|$ is continuously differentiable on $\R^n\backslash\{\zero\}$ and $\dn$ is strictly monotone.
			%:
			%\begin{equation}\label{eq:hom_norm_der_general}
			%	\frac{\partial \|x\|_{\dn}}{\partial x}=\left.e^{s}\tfrac{\frac{\partial \|\dn(-s)x\|}{\partial x}}{\left.\frac{\partial \|z\|}{\partial z}\right|_{z=\dn(-s)x} G_{\dn} \dn(-s)x}\right|_{s=\ln\|x\|_{\dn}} \;\;\text{ for } \;\; x\neq \zero.
			%\end{equation}
		\end{itemize}
	\end{lemma}
 
	For the $\dn$-homogeneous norm $\|x\|_{\dn}$ induced by the weighted Euclidean norm $\|x\|=\sqrt{x^{\top}Px}$ we have \cite{Polyakov2018:RNC}
	 %the formula \eqref{eq:hom_norm_der_general} can be rewritten as follows
	\begin{equation}\label{eq:hom_norm_der}
	    \tfrac{\partial \|x\|_{\dn}}{\partial x}=\|x\|_{\dn}\tfrac{x^{\top}\dn^{\top}(-\ln \|x\|_{\dn})P\dn(-\ln \|x\|_{\dn})}{x^{\top}\dn^{\top}(-\ln \|x\|_{\dn})PG_{\dn}\dn(-\ln \|x\|_{\dn})x}.
	\end{equation}
	If $\dn$ is a linear continuous strictly monotone dilation then the mapping $\Psi : \R^n \mapsto  \R^n$,
	\begin{equation}\label{eq:coord_trans_Rn}
		\Psi(x)=\|x\|_{\dn} \dn(-\ln \|x\|_{\dn}) x, \quad x\in \R^n\backslash\{\zero\}
	\end{equation}
	is a homeomorphism in $\R^n$, its inverse is given by 
	$
	\Psi^{-1}(z)=\|z\|^{-1}\dn(\ln \|z\|)z, \quad z\in  \R^n\backslash\{\zero\}
	$,
	with the prolongation $\Psi(\zero)=\Psi^{-1}(\zero)=\zero$ by continuity. If $\|\cdot\|$ is differentiable on $\R^{n}\backslash\{\zero\}$ then $\Psi$ is diffeomorphism on $\R^n\backslash\{\zero\}$.
	
	\begin{theorem}\cite{Polyakov2020:Book}\label{thm:tildeRn}
		Let a linear continuous dilation $\dn$ in $\R^n$ be monotone with respect to  a norm $\|\cdot\|$. Let an addition operation $\tilde + : \R^n\times \R^n \mapsto \R^n$ and a 
		multiplication by a scalar $\tilde \cdot: \R\times \R^n\mapsto \R^n$ be defined as 
		follows 
		\begin{itemize}
			\item $x\tilde+y:=\Psi^{-1}(\Psi(x)+\Psi(y))$, where $x,y\in \R^n$,
			\item $\lambda \tilde \cdot x:=\sign(\lambda)\dn(\ln |\lambda|)x$, where $\lambda\in \R$, $x\in \R^n$,
		\end{itemize}
		where $\Psi: \R^n\mapsto \R^n$ is given by \eqref{eq:coord_trans_Rn}.
		Then  the set $\R^n$ together with the operations $\tilde +$ and $\tilde \cdot$
		is a linear vector space $\R^n_{\dn}$ with the norm  $\|\cdot\|_{\dn}$.
	\end{theorem}
 
	The latter theorem justifies the name norm for  $\|\cdot\|_{\dn}$. Notice that $\dn$-homogeneous cones are usually nonlinear in $\R^n$ (see Fig. \ref{fig:cones}), but they may be linear  in $\R^n_{\dn}$.
	\begin{corollary}\label{cor:hom_lin_cone}
		For any $h_i\in \R^n, i=1,\ldots,p$	the set 
		\begin{equation}\label{eq:Omega_lin}
			\Omega_{\rm lin}=\{x\in \R^n: h_i^{\top}{\Psi(x)}\geq 0, i=1,\ldots,p\}
		\end{equation}
		is a linear positive cone in the space $\R^n_{\dn}$.
	\end{corollary}
 
	\begin{proof}
		Indeed, since  for $x,y\in \Omega_{\rm lin}$ we have $h_i^{\top}\Psi(x)\geq 0$, $h_i^{\top}\Psi(y)\geq 0$ for all $i=1,\ldots,p$
		then using $x\tilde+y=\Psi^{-1}(\Psi(x)+\Psi(y))$
		we derive 
		$
		h_i^{\top} \Psi(x\tilde+ y)=h_i^{\top}\Psi(x)+h_i^{\top}\Psi(y)\geq  0 \quad \Rightarrow \quad  x\tilde +y\in\Omega_{\rm lin}.
		$
	\end{proof}

	\subsection{Homogeneous functions and vector fields}
	
	Below we study systems that are symmetric on homogeneous cones with respect to a linear dilation $\dn$. The  dilation symmetry introduced by the following definition is known as a generalized  homogeneity \cite{Zubov1958:IVM}, \cite{Kawski1991:ACDS}, \cite{Rosier1992:SCL}, \cite{BhatBernstein2005:MCSS}, \cite{Polyakov2020:Book}.
	
	\begin{definition}\cite{Kawski1991:ACDS}\label{def:hom_fun}
		A function $h: \R^n\mapsto \R$ is $\dn$-homogeneous of degree $\nu\!\in\! \R$ if 
		\[
		h(\dn(s)x)=e^{\nu s}h(x), \quad \forall x\in \R^n, \quad \forall s\in \R.
		\]
	\end{definition}

	Notice that $h(x)=\|x\|_{\dn}^{\nu}h(\dn(-\ln \|x\|_{\dn})x), \forall x\neq \zero$ for any $\dn$-homogeneous function $h$ of degree $\nu$,
	where $\|\cdot\|_{\dn}$ is a canonical homogeneous norm induced by a norm $\|\cdot \|$ in $\R^n$, so $\|\dn(-\ln \|x\|_{\dn})x\|=1,\forall x\neq \zero$.
	The latter implies that if a $\dn$-homogeneous function $h$ is continuous at zero then either $h\equiv {\rm const}$ or $h(\zero)=0$ and $\nu>0$. So, non-constant \textit{continuous} $\dn$-homogeneous  functions $\phi_i$  may define a $\dn$-homogeneous cone  $\Omega$ given by \eqref{eq:inv_set} \textit{if and only if} they have positive homogeneity degree. 
	\begin{assumption}\label{as:1}
		Let  $\phi_i\in   C(\R^n,\R)\cap C^1(\R^n\backslash\{\zero\})$ in \eqref{eq:inv_set} be  $\dn$-homogeneous of degree $\nu_i> 0$, $i=1,2,\ldots,p$ 
	\end{assumption}
 
	 The set $\Omega_i=\{x\in \R^n: \phi_i(x)\geq 0\}$ with $\phi_i$ satisfying  this assumption is a $\dn$-homogeneous cone with a smooth boundary everywhere except, probably, the origin. 
		The set $\Omega=\bigcap\limits_{i=1}^p \Omega_i$ given by \eqref{eq:inv_set} is also a $\dn$-homogeneous cone.
		\begin{assumption}\label{as:1bis}
			The set $\Omega_0$ given by \eqref{eq:Omega_0} is bounded. 
		\end{assumption}
  
		Since $\Omega_0$ is a $\dn$-homogeneous cone, then  \textit{the boundedness of $\Omega_0$ means that}   
	\begin{equation}\label{eq:omega0=0}
		\Omega_0\!=\!\{\zero\}.
	\end{equation}
	Indeed, otherwise, the set $\Omega_0$ is unbounded since  $\dn(s)x\in \Omega_0, \forall x\in \Omega_0, \forall s\in \R$,
	but $\|\dn(s)x\|\to+\infty$ as $s\to +\infty$ for any $x\neq \zero$.  Below we study the systems on $\dn$-homogeneous cones $\Omega$ with both bounded and unbounded sets	 $\Omega_0$.
	
	\begin{definition}\cite{Kawski1991:ACDS}\label{def:hom_vf}
		A vector field $f:\R^n \mapsto \R^n$ is  $\dn$-homogeneous of 
		degree $\mu\in \R$  if 
		\begin{equation}\label{eq:homogeneous_operator_Rn}
			g(\dn(s)x)=e^{\mu s}\dn(s)g(x), \quad  \quad  \forall s\in\R, \quad \forall x\in \R^n.
		\end{equation}
		%	where $T\R^n$ is the tangent space for $\R^n$ (so, $T\R^n=\R^n$).
	\end{definition}
 
%	Repeating the above consideration for any $\dn$-homogeneous vector field  we derive 	$$g(x)=e^{\ln \|x\|_{\dn}(\mu I_n+G_{\dn})}g(\dn(-\ln \|x\|_{\dn})x),\forall x\neq \zero$$  and conclude that if a $\dn$-homogeneous vector field is continuous on $\R^n$ and $g(\zero)=\zero$  then either $g\equiv \zero$  or  $\mu I_n+G_{\dn}$ is anti-Hurwitz. By default below we consider only non-constant $\dn$-homogeneous vector fields.
	
	Formally, to avoid a collision in Definitions \ref{def:hom_fun} and \ref{def:hom_vf} for $n=1$, a vector field $g$  should be defined as $g: \Xi\mapsto T\Xi$, where $T\Xi$ is the tangent space for $\Xi$.  Since the tangent space of $\R^n$ is associated with $\R^n$, we  simply write $g:\R^n\mapsto \R^n$.
	
	The  homogeneity of a mapping is inherited by other mathematical objects induced by this mapping.
	In particular, solutions of $\dn$-homogeneous system\footnote{A system is homogeneous  if its is governed by a $\dn$-homogeneous vector field}
	\begin{equation}\label{eq:hom_system}
		\dot x=g(x),\quad  t>0, \quad x(0)=x_0\in \R^n
	\end{equation}
	are symmetric with respect to the dilation $\dn$ in the following sense \cite{Zubov1958:IVM}, \cite{Kawski1991:ACDS}, \cite{BhatBernstein2005:MCSS}
	\begin{equation}
		x(t,\dn(s)x_0)=\dn(s)x(e^{\mu s}t,x_0),
	\end{equation}  
	where $x(\cdot,z)$ denotes a solution of  \eqref{eq:hom_system}  with $x(0)=z\in \R^n$ and 
	$\mu\in \R$ is the homogeneity degree of $g$. The mentioned symmetry  
	of solutions implies many useful properties of homogeneous system such as equivalence of local and global results (e.g, about existence, uniqueness and stability  solutions) or finite-time stability of any asymptotically stable $\dn$-homogeneous  system \eqref{eq:hom_system} with  negative degree  \cite{BhatBernstein2005:MCSS}.
	\begin{corollary}\label{cor:forward_mu_neg}
	 If $g\in C(\R^n)$ is a  $\dn$-homogeneous vector field of degree $\mu\leq 0$  then the system \eqref{eq:hom_system} is forward complete\footnote{A system is  forward complete on $\Omega$ if, for any initial state $x_0\in \Omega$, any solution of the system is defined globally in time.}.
	\end{corollary}
 
	\begin{proof}
		By Corollary \ref{cor:monotonicity}, any linear continuous dilation is strictly monotone with respect to the weighted Euclidean norm $\|x\|=\sqrt{x^{\top}Px}$ with P satisfying \eqref{eq:mon_d_P}. Let $\|\cdot\|_{\dn}$ be the canonical homogeneous norm induced by $\|\cdot \|$. Then using \eqref{eq:hom_norm_der} and $\dn$-homogeneity of $g$ we derive
		\[
		\tfrac{d}{dt}\|x\|_{\dn}=\|x\|_{\dn} \tfrac{x^{\top}\dn^{\top}(-\ln \|x\|_{\dn})P\dn(-\ln \|x\|_{\dn})\dot x}{x^{\top}\dn^{\top}(-\ln \|x\|_{\dn})PG_{\dn}\dn(-\ln \|x\|_{\dn})x}=\|x\|^{1+\mu}_{\dn}\tilde g(\dn(-\ln \|x\|_{\dn})x),
		\]
		where $\tilde g(z)=\tfrac{z^{\top}Pg(z)}{z^{\top}PG_{\dn}z}$, $z\in\R^n$. Taking into account $\|\dn(-\ln \|x\|_{\dn})x\|=1$ we conclude $\exists \gamma>0 : \|\tilde g(\dn(-\ln \|x\|_{\dn})x)\|\leq \gamma=\sup_{\|z\|=1}\|\tilde g(z)\|<+\infty$ for all $x\neq \zero$ and 
		$
		\tfrac{d}{dt}\|x\|_{\dn}\leq \gamma \|x\|^{1+\mu}_{\dn},\quad \forall x\neq \zero.
		$
		For $\mu\leq 0$ the obtained differential inequality guarantees that the system \eqref{eq:hom_system} is forward complete.
	\end{proof}

	The well-known  Euler's homogeneous function theorem proves that the derivative of any smooth standard homogeneous function $h(e^s x)=e^{\nu s}h(x)$ is homogeneous and $\frac{\partial h(x)}{\partial x}x=\nu f(x),\forall x\in \R^n$. An analog of the Euler's theorem for $\dn$-homogeneous vector fields is given below.
	\begin{theorem}\cite{Polyakov2020:Book}\label{thm:Euler}
 If $h\!\in\! C^1(\R^n\backslash\{\zero\},\R)$ is $\dn$-homogeneous function of degree $\nu\in \R$ then 
		\begin{equation}
			\tfrac{\partial h(x)}{\partial x}G_{\dn}x=\nu h(x),  \quad \forall x\neq \zero,
		\end{equation} 
		\begin{equation}
			\left.\tfrac{\partial h(z)}{\partial z}\right|_{z=\dn(s)x}\dn(s)=e^{\nu s}\tfrac{\partial h(x)}{\partial x}, \quad \forall x\neq \zero, \quad \forall s\in \R.
		\end{equation}   	
		If $g\in C^1(\R^n\backslash\{\zero\},\R^n)$ is a $\dn$-homogeneous vector field of degree $\mu\in \R$ then 
		\begin{equation}
			\tfrac{\partial g(x)}{\partial x}G_{\dn}x=(\mu I_n+G_{\dn})g(x),  \quad \forall x\neq \zero,
		\end{equation} 
		\begin{equation}
			\left.\tfrac{\partial g(z)}{\partial z}\right|_{z=\dn(s)x}\dn(s)=e^{\mu s}\dn(s)\tfrac{\partial g(x)}{\partial x}, \quad \forall x\neq \zero, \quad \forall s\in \R.
		\end{equation}   	
	\end{theorem}

		The perturbed  homogeneous system \eqref{eq:p_hom_system} under consideration is characterized  by the following  assumption inspired by \cite{Ryan1995:SCL}, \cite{Hong2001:Aut}, \cite{Andrieu_etal2008:SIAM_JCO}.
	\begin{assumption}\label{as:2}		
		Let a linear continuous dilation  $\dn$ in $\R^n$ be monotone with respect to a norm $\|\cdot\|_{\R^n}$, and a linear continuous dilation $\tilde \dn$ in $\R^k$ be monotone with respect to a norm $\|\cdot\|_{\R^k}$.
		Let  the vector field $\hat f \in C(\R^{n+k})$ given by	
		\[
		\hat f(x,q)=
		\left(
		\begin{array}{c}
			f(x,q)\\
			\zero	
		\end{array}
		\right),\quad x\in \R^n, \quad q\in \R^k,
		\]
		be $\hat \dn$-homogeneous of degree $\mu> -1$, where
		$
		\hat \dn(s) =\left(\begin{smallmatrix}\dn(s) & \zero \\\zero & \tilde \dn(s)\end{smallmatrix}\right), s\in \R.
		$
	\end{assumption}

		A certain dilation symmetry can be discovered for solutions of the perturbed homogeneous system as well (see, e.g., \cite{Polyakov2021:Aut}).
		
		\begin{lemma}\label{lem:hom_sol_pert}
			Let $x_{q}(t,x_0)$, $t\geq t_0$ denote a solution of the system \eqref{eq:p_hom_system}.
			Let $x_{q_s}(t,\dn(s)x_0)$, $t\geq t_0$ denote a solution of the system \eqref{eq:p_hom_system} 
			for $x_0\in \R^n$ replaced with $\dn(s)x_0\in\R^n,s\in \R$ and $q\in L^{\infty}(\R,\R^n)$  replaced with $q_s\in L^{\infty}(\R,\R^k)$ given by 
			\begin{equation}\label{eq:q_s}
			q_s(t)= \tilde \dn(s)q(t_0+e^{\mu s} (t-t_0)), \quad \forall t\in \R.
			\end{equation}
			Then, under Assumption \ref{as:2}, the  identity 
			\begin{equation}\label{eq:hom_sol_pert}
				x_{q_s}(t,\dn(s)x_0)\!=\!\dn(s)x_{q}(t_0+e^{\mu s}(t-t_0),x_0), \quad t\geq t_0,
			\end{equation}
			holds as long as solutions exist.
		\end{lemma}
  
		\begin{proof}
			Indeed, denoting $\tau=t_0+e^{\mu s}(t-t_0)$ we obtain
			\[
			\frac{d}{dt}\dn(s)x_{q}(t_0+e^{\mu s}(t-t_0),x_0)=e^{\mu s}\dn(s) \frac{d}{d\tau} x_q(\tau,x_0)=e^{\mu s}\dn(s)f(x_q(\tau),q(\tau)).
			\]
			Using the $\hat \dn$-homogeneity of $\hat f$ we derive
			\[
			\frac{d}{dt}\dn(s)x_{q}(t_0+e^{\mu s}(t-t_0),x_0)\!=\!f(\dn(s)x_{q}(t_0+e^{\mu s}(t-t_0),x_0), \tilde \dn(s)q(t_0+e^{\mu s}(t-t_0)))
			\]
			Since $\frac{d}{d t} x_{q_s}(t,\dn(s)x_0)=f(x_{q_s}(t,\dn(s)x_0),q_s(t))$, then  the identity 
			\eqref{eq:hom_sol_pert} holds.
		\end{proof}
	
	If solutions of the perturbed system \eqref{eq:p_hom_system} are not unique then the identity \eqref{eq:hom_sol_pert} is understood in the set-theoretic sense: \textit{the set of solutions is invariant up to certain scaling of initial condition, time and the perturbation function $q$}.

	%%%%%%%%%%%%%%%%%%%%%%%%%%%%%%%%%%%%%%%%%%%%%%
	%%%%%%%%%%%%%%%%%%%%%%%%%%%%%%%%%%%%%%%%%%%%%%
	%%%%%%%%%%%%%%%%%%%%%%%%%%%%%%%%%%%%%%%%%%%%%%
	\section{Homogeneous Systems on Cones}\label{sec:hom_ISS_on_omega}

	\subsection{Positively Invariant Homogeneous Cones}
	The criterion of the positive invariance of a \textit{closed} set $\Omega\subset \R^n$  is well known in the literature (see, e.g., \cite{Nagumno1942:PPMSJ}, \cite{Aubin1991:Book}, \cite{Blanchini1999:Aut}): \textit{if  the forward complete system 
		\eqref{eq:hom_system} has a unique solution for any $x_0\in \Omega$ then the closed set $\Omega$ is positively invariant for the system \eqref{eq:hom_system} if and only if 
		\begin{equation}\label{eq:inv_general}
			g(x)\in \mathcal{C}_{\Omega}(x), \quad \forall x\in \Omega,
		\end{equation}
		where $\mathcal{C}_{\Omega}(x)\subset \R^n$ is the (Bouligand)  tangent cone to $\Omega$ at $x\in \R^n$ defined as}
	\begin{equation}
		\mathcal{C}_{\Omega}(x)=\left\{z\in \R^n: \liminf_{h\to 0^+}\tfrac{\inf_{y\in \Omega}\|x+hz-y\|}{h}=0 \right\}.
	\end{equation}
	Obviously, $\mathcal{C}_{\Omega}(x)$ is a positive cone ($z\in \mathcal{C}_{\Omega}(x)\Rightarrow \lambda z\in \mathcal{C}_{\Omega}(x),\forall \lambda>0$).
	Notice that $\mathcal{C}_{\Omega}(x)=\{\emptyset\}$ for  $x\notin\Omega$ and $\mathcal{C}_{\Omega}(x)=\R^n$ for $x\in \interior \Omega$, so the condition \eqref{eq:inv_general} can be relaxed to $x\in \partial \Omega$.
	The homogeneity of the set $\Omega$ implies the homogeneity of the tangent cone $\mathcal{C}_{\Omega}(x)$.
	\begin{lemma}\label{lem:hom_tang_cone}
		If  $\Omega\subset \R^n$ is a closed $\dn$-homogeneous cone then 
		\begin{equation}
			\mathcal{C}_{\Omega}(\dn(s)x)=\dn(s)\mathcal{C}_{\Omega}(x),\quad \forall x\in \Omega, \quad \forall s\in \R.
		\end{equation}
	\end{lemma}
 
	\begin{proof}
		On the one hand, we have  
		\[
		\mathcal{C}_{\Omega}(\dn(s)x)=\left\{z\in \R^n: \liminf_{h\to 0^+}\tfrac{\inf_{y\in \Omega}\|\dn(s)x+hz-y\|}{h}=0 \right\}
		\]
	%	\[
	%	=\left\{z\in \R^n: \liminf_{h\to 0^+}\tfrac{\inf_{y\in \Omega}\|\dn(s)x+hz-y\|}{h}=0 \right\}
	%	\]
		\[
		=\left\{\dn(s)\tilde z\in \R^n: \liminf_{h\to 0^+}\tfrac{\inf_{\dn(s)\tilde y\in \Omega}\|\dn(s)(x+h\tilde z-\tilde y)\|}{h}=0 \right\}
		\]
		\[
		=\dn(s)\left\{\tilde z\in \R^n: \liminf_{h\to 0^+}\tfrac{\inf_{\tilde y\in \Omega}\|\dn(s)(x+h\tilde z-\tilde y)\|}{h}=0 \right\}
		\]
		where the $\dn$-homogeneity of $\Omega$ (i.e., $\dn(s)\Omega\subset\Omega,\forall s\in \R$) is utilized on the last step. On the other hand, the set $\mathcal{C}_{\Omega}(x)$ is independent of the selection of a norm \cite{Nagumno1942:PPMSJ}, \cite{Blanchini1999:Aut}.
		Since $\dn(s)$ is invertible then $\|\cdot \|_{*}=\|\dn(s)\cdot\|$ is a norm in $\R^n$ and 
		\[
		\mathcal{C}_{\Omega}(x)=\left\{\tilde z\in \R^n: \liminf_{h\to 0^+}\tfrac{\inf_{\tilde y\in \Omega}\|\dn(s)(x+h\tilde z-\tilde y)\|}{h}=0 \right\}.
		\]
		The proof is complete.	\end{proof}
	
	In the view of the latter lemma, the positive invariance of a homogeneous cone for a  homogeneous system can be analyzed considering the tangent cone and the vector field on the unit sphere.
	\begin{corollary}\label{cor:inv_cone}
		Let	a linear continuous dilation $\dn$ be monotone with respect to a norm $\|\cdot\|$ in $\R^n$ and a vector field  $g\in  C(\R^n)$ be  $\dn$-homogeneous of degree   $\mu\in \R$  such that $g(\zero)=\zero$, solutions of the system \eqref{eq:hom_system} are forward unique\footnote{A system is forward unique if any its solution is unique in the forward time.} and complete on $\Omega$.  	
		A closed $\dn$-homogeneous  cone $\Omega$  is positively invariant for  \eqref{eq:hom_system}
		if and only if 
		\begin{equation}\label{eq:hom_inv}
			g(x)\in   \mathcal{C}_{\Omega}(x), \quad \forall x\in\partial \Omega \cap S, 
		\end{equation}
		where $S=\{x\in \R^n: \|x\|=1\}$ is the unit sphere.
	\end{corollary}
 
	\begin{proof}
		For any $x\neq \zero$, due to $\dn$-homogeneity of $g$ and $\mathcal{C}_{\Omega}$ (see  Lemma \ref{lem:hom_tang_cone}), we have 
		\[
		g(x)= \|x\|^{-\mu}_{\dn}\dn(\ln \|x\|_{\dn})g(\dn(-\ln\|x\|_{\dn})x)
		%\]
		\text{ and } 
		%\[
		\mathcal{C}_{\Omega}(x)= \dn(\ln \|x\|_{\dn})\mathcal{C}_{\Omega}(\dn(-\ln \|x\|_{\dn})x).
		\]
		Hence, the condition \eqref{eq:hom_inv} implies 
	$	%\[
		g(x)\in   \mathcal{C}_{\Omega}(x), \quad \forall x\in\partial \Omega \backslash\{\zero\}, 
	$	%\]
		which is necessary and sufficient for the set $\Omega\backslash\{\zero\}$ to be positively invariant 
		(at least, as long as $x(t)\neq \zero$). On the one hand, since the $\dn$-homogeneous cone is closed then  $\zero\in \Omega$. On the other hand, since  $g(\zero)=\zero$ then,
		taking into account the forward uniqueness of the zero  solution, we complete the proof. 
	\end{proof}
	
	The positive invariance of  a ``linear'' $\dn$-homogeneous cone $\Omega_{\rm lin}$ admits a more simple algebraic characterization. 
	\begin{corollary}\label{cor:inv_linear_cone}
		Let, under conditions of Corollary \ref{cor:inv_cone}, 
		the norm in $\R^n$ be defined as $\|x\|=\sqrt{z^{\top}Pz}$ with $P\in \R^{n\times n}$ satisfying \eqref{eq:mon_d_P}. The $\dn$-homogeneous cone $\Omega=\Omega_{\rm lin}$ given by \eqref{eq:Omega_lin} is positively invariant for the system \eqref{eq:hom_system} if and only if
		\begin{equation}\label{eq:inv_lin_cone}
			h_i^{\top} g(z)\geq \tfrac{z^{\top}Pg(z)}{z^{\top}PG_{\dn}z} h_i^{\top}G_{\dn}z , \quad \forall z\in \Xi_i, \quad \forall i=1,\ldots,p,
		\end{equation}
		where  $\Xi_i:=\{z\in \R^n: \|z\|=1,  h_i^{\top}z=0, h^{\top}_jz\geq 0, \forall j\neq i\}$.
	\end{corollary}
 
	\begin{proof}
		In the new coordinates
		$%\[
		z=\Psi(x)%:=\|x\|_{\dn}\dn(-\ln \|x\|_{\dn})x
		$, %\]
		the system \eqref{eq:hom_system} becomes \cite{Polyakov2018:RNC}
		\begin{equation}\label{eq:z_system_g}
			\dot z=\|z\|^{1+\mu}\tilde g \left(\tfrac{z}{\|z\|}\right), 
		\end{equation}
		where $
		\tilde g(\tilde x)=\tfrac{\tilde x^{\top}Pg(\tilde x)}{\tilde x^{\top}PG_{\dn}\tilde x}(I_n-G_{\dn})\tilde x+g\left(\tilde x\right),\tilde x\in \R^{n}\backslash\{\zero\}.	$ 	
		By Corollary \ref{cor:hom_lin_cone}, the $\dn$-homogeneous cone $\Omega_{\rm lin}$ in the new coordinates becomes the conventional linear cone:
		\begin{equation}\label{eq:tilde_Omega_lin}
		\tilde \Omega_{\rm lin}=\{z\in \R^n : h_i^{\top}z\geq 0, i=1,\ldots,n\}.
		\end{equation}
		Notice that $\Xi_i=\{z\in \partial \Omega_{\rm lin}\cap S : h_i^{\top}z=0\}$.
		The system \eqref{eq:z_system_g} and the linear cone $\tilde \Omega_{\rm lin}$ are standard homogeneous, so 
		using the criterion \eqref{eq:hom_inv} and 
		taking into account $\partial \tilde \Omega_{\rm lin}\cap S=\partial \Omega_{\rm lin}\cap S=\Xi_1\cup\ldots\cup \Xi_p$ we derive 
		the following  necessary and sufficient condition of the positive invariance of the  cone  $\tilde \Omega_{\rm lin}$:
	$ %	\[
		\tilde g(x)\in \mathcal{C}_{\Omega}(x),  \forall x\in \Xi_i,  \forall i=1,\ldots,p.
	$ %	\]
		The latter is  equivalent to
	$ %	\[
		h_i^{\top}\tilde g(x)\geq 0, \forall x\in \Xi_i,  \forall i=1,\ldots,p.
	$ %	\]
		Taking into account $h_i^{\top}x=0,\forall x\in \Xi_i$ we derive \eqref{eq:inv_lin_cone}.
	\end{proof}

	The requirement of the uniqueness of solutions of \eqref{eq:hom_system} is fundamental \cite{Blanchini1999:Aut}
	for the characterization of the positive invariance
	of the closed set $\Omega$. It  can be guaranteed, for example, asking $g\in C^1(\Omega)$ (or Lipschitz continuity on $\Omega$). However, in this case, the homogeneous vector field may have only non-negative homogeneity degree. Below we assume  $g\in C^1(\Omega\backslash\{\zero\})$ in order to include the finite-time stable homogeneous systems into considerations. Since the regularity of the vector field at $\zero$ is excluded we would need an additional assumption about  the zero solution.

	\begin{corollary}\label{cor:inv_sufficient}
		Let a linear continuous dilation $\dn$ in $\R^n$ be strictly monotone with respect to the norm $\|x\|=\sqrt{x^{\top}P x}$. Let a non-constant vector field $g\in C(\R^ n)\cap C^1(\Omega\backslash \{\zero\})$ be $\dn$-homogeneous of degree $\mu\leq 0$ such that the matrix $\mu I_n+G_{\dn}$ is invertible and the system \eqref{eq:hom_system} has the unique zero solution.  Let the  $\dn$-homogeneous  cone $\Omega=\Omega_{\rm lin}$ be given by the formula \eqref{eq:Omega_lin} with $p=n$ and the  matrix 
		\begin{equation}\label{eq:H}
			H=\left(
			\begin{smallmatrix}
				h_1^{\top}\\
				\cdots\\
				h_n^{\top}\\	
			\end{smallmatrix}
			\right)\in \R^{n\times n}
		\end{equation}
		be invertible.	The cone $\Omega_{\rm lin}$ 
		is positively invariant for the system \eqref{eq:hom_system} if 
		\begin{equation}\label{eq:M}
			M(x)=H\!\left(\tfrac{x^{\top}Pg(x)}{x^{\top}PG_{\dn}x}(I_n-G_{\dn}) +(\mu I_n+G_{\dn})^{-1}\tfrac{\partial g(x)}{\partial x}G_{\dn}\right)H^{-1}
		\end{equation} 
		is a Metzler\footnote{A matrix $M\in\R^{n\times n}$ is Metzler if all its off-diagonal elements are nonnegative.} matrix for any $x\in \partial \Omega_{\rm lin}\cap S$.
	\end{corollary}
 
	\begin{proof}
		The system \eqref{eq:hom_system}  is forward complete in  view of Corollary \ref{cor:forward_mu_neg}. Since $g\in C^1(\R^n\backslash\{\zero\})$ then its solutions are unique on $\R^n\backslash\{\zero\}$. Uniqueness of the zero solution is guaranteed by the assumption of the corollary.
		
		Let us consider the equivalent system \eqref{eq:z_system_g}.
		Since  $\mu I_{n}+G_{\dn}$ is invertible then using Theorem \ref{thm:Euler} we rewrite \eqref{eq:z_system_g} as follows
		\begin{equation}\label{eq:z_system_g2}
			\dot z=\|z\|^{1+\mu}\left(\tilde \gamma\left(\tfrac{z}{\|z\|}\right)(G_{\dn}-I_n)+(\mu I_n+G_{\dn})^{-1}\left.\tfrac{\partial g(x)}{\partial x}\right|_{x=\frac{z}{\|z\|}}G_{\dn}\right)\tfrac{z}{\|z\|},
		\end{equation}
		where $\tilde \gamma(x)=\tfrac{x^{\top}Pg(x)}{x^{\top}PG_{\dn}x}, x\neq \zero$.
		Since the matrix $H$ is invertible then the invariance of $\tilde \Omega_{\rm lin}$ is equivalent to the positivity of the system
		\begin{equation}\label{eq:y_system_g}
			\dot y=\|z\|^{\mu}\left.M(x)\right|_{x=\frac{z}{\|z\|}} y, \quad y=Hz.
		\end{equation}
		The latter system is positive if $M(x)$ is a Metzler matrix for any $x\in S\cap \partial \tilde \Omega_{\rm lin}$.
		Taking into account $S\cap \partial \tilde \Omega_{\rm lin}=S\cap  \partial \Omega_{\rm lin}$
		we complete the proof.
	\end{proof}
	
	\subsection{Stability Analysis on Homogeneous Cones}
	Using  Lemma \ref{cor:inv_cone} we extend the Zubov-Rosier Theorem   \cite{Zubov1958:IVM}, \cite{Rosier1992:SCL}
to systems on homogeneous cones. 
	\begin{theorem}\label{thm:Zubov_Rosier_on_Omega}
		Let	a linear continuous dilation $\dn$ be monotone with respect to a norm $\|\cdot\|$ in $\R^n$, Assumptions \ref{as:1}  and \ref{as:1bis} be fulfilled and a vector field  $g\in C(\R^n)\cap C^1(\Omega\backslash\{\zero\})$ be  $\dn$-homogeneous of degree   $\mu\in \R$  such that  the system \eqref{eq:hom_system} has the unique zero solution. Then the system \eqref{eq:hom_system}   is uniformly asymptotically stable on $\Omega$ if and only if the  condition \eqref{eq:hom_inv} is fulfilled and there exists a $\dn$-homogeneous function $V:\R^n\mapsto \R_+$ 
		of degree 1  satisfying
		\begin{itemize}
			\item[1)] $V\!\in\! C(\Omega)\!\cap\! C^1(\Omega\backslash\{\zero\})$ and  there exist $0<k_1\leq k_2<\infty$ such that 
			\[
			k_1\|x\|_{\dn}%\left|\phi^{\frac{1}{\nu_i}}_i(x)\right|
			\leq V(x)\leq k_2\|x\|_{\dn}, \quad \forall  x\in \Omega;
			\]
			\item[2)] there exist a $\dn$-homogeneous function $W:\R^n\mapsto \R$ of degree $\mu+1$ such that $W\in C(\Omega\backslash\{\zero\})$, $W(x)>0,\forall x\in \Omega\backslash\{\zero\}$ and 
			\[
			\dot V(x)\leq -W(x), \quad \forall x\in \Omega\backslash\{\zero\}.
			\]
		\end{itemize}
	\end{theorem}
 
	\begin{proof}
		\textit{Sufficiency}. By Corollary \ref{cor:inv_cone},  any solution $x(t), t\geq 0$ initiated in $\Omega$ belongs to $\Omega$ 
		and 
		\[
		0\leq \phi_i(x(t))
		\]
		as long as $|x(t)|<+\infty$. Using  the Lyapunov function $V$, we conclude (in the usual way) that $\exists \tilde \beta\in \mathcal{KL}$ such that  $|x(t)|\leq \tilde \beta(|x_0|,t)$ for all $t\geq 0$ and for all $x(0)=x_0\in \Omega$.	
		Since $\phi_i$ is $\dn$-homogeneous of degree $\nu_i>0$ then 
		\begin{equation}\label{eq:phi2x}
			0\leq \phi(x(t))=\|x(t)\|_{\dn}^{\nu_i}\phi(\dn(-\ln \|x(t)\|)x(t))\leq c_i \|x(t)\|_{\dn}^{\nu_i}
		\end{equation}
		as long as $|x(t)|\neq \zero$, where $c_i=\sup_{\|z\|=1}\phi_i(z)<+\infty$ due to continuity of $\phi_i$.
		Using \eqref{eq:rel_norm_and_hom_norm_Rn} and the equivalence of the norms $|\cdot|$ and $\|\cdot\|$ we derive \eqref{eq:GAS}.

		\textit{Necessity.}  
		The inequality \eqref{eq:GAS}  guarantees that $\Omega$ is  positively invariant for the system \eqref{eq:hom_system}, so the inclusion  \eqref{eq:hom_inv} holds. Since $\Omega_0=\{\zero\}$ then 
		$
		x(t)\to \zero \quad \Leftrightarrow \quad x(t)\to \Omega_0 % \quad \Leftrightarrow \quad \phi_i(x(t))\to 0, \forall i=1,2,...,p
		$
		and the inequality \eqref{eq:GAS}  implies that the origin is the asymptotically stable equilibrium of the system \eqref{eq:hom_system} on $\Omega$. Let $a\in C^{\infty}(\R)$ such that $a(\rho)=0$ for $\rho\in [0,0.5])$, $a'(\rho)>0$ for $\rho\in (0.5,1)$ 
		and $a(\rho)={\rm const}$ for $\rho\geq 1$. Inspired by \cite{Massera1949:AM}, let us consider  the function
		\[
		V_{0}(x_0)=\int^{\infty}_0a(|x(\tau,x_0)|)d \tau,
		\]
		where $x(\cdot,x_0)$ denotes the unique solution of the system \eqref{eq:hom_system}. Due to asymptotic stability of the system on $\Omega$, the function $V$ is well defined and locally  bounded on $\Omega$. Since $g\in C^1(\Omega\backslash\{\zero\})$ then  solutions depend smoothly on initial conditions
		\cite{CoddingtonLevinson1955:Book}, so $V_0\in C(\Omega)\cap C^1(\Omega\backslash\{\zero\})$. 
		Using the semi-group property of solutions we derive 
		\[
		V_{0}(x(t+h,x_0))-V_{0}(x(t,x_0))\!=\!\int^{\infty}_0\!\!\!\!\!\!a(|x(\tau+h, x(t,x_0))|)d \tau-\!\int^{\infty}_0\!\!\!\!\!\!a(|x(\tau,x(t,x_0))|)d \tau
		\]
		\[
		=-\int^h_0a(|x(\tau,x(t,x_0))|)d \tau, \quad \forall h>0
		\]
		and 
		$$
		\tfrac{dV_{0}(x(t,x_0))}{dt}=\left.\tfrac{\partial V(z)}{\partial z} g(z)\right|_{z=x(t,x_0)}=-a(|x(t,x_0)|)\quad  \text{ if } \quad x(t,x_0)\in \Omega\backslash\{\zero\}.
		$$
		Moreover,  there exists $c_1>0$ such that $V_0(x_0)\leq c_1$ for $|x_0|\leq 1, x_0\in \Omega$ and 
		there exist $c_2>c_1$ and $r>1$ such that  $x_0\in \Omega, |x_0|>r \Rightarrow V_{0}(x_0)\ge c_2$.
		
		Let $b\in C^{\infty}(\R)$ such that $b(\rho)=0$ for $\rho\in [0,c_1])$, $b'(\rho)>0$ for $\rho\in (c_1,c_2)$ 
		and $b(\rho)={\rm const}$ for $\rho\geq c_2$. 
		Since $\Omega$ is the $\dn$-homogeneous cone then the function \cite{Rosier1992:SCL}
		\[
		V(x)=\int^{+\infty}_0 e^{-s} b(V_0(\dn(s)x))ds
		\]
		is well-defined for all $x\in \Omega$. By construction, $V(\dn(\tau)x)=e^{\tau}V(x),\forall x\in \Omega$ and $V\in C(\Omega)\cap C^1(\Omega\backslash\{\zero\})$.
		Using Lemma \ref{lem:hom_norm} we conclude  $x\in \Omega,  \|\dn(s)x\|_{\dn}\leq r_1 \Rightarrow |\dn(s)x|\leq 1$ for some $r_1>0$ and $ x\in \Omega, \|\dn(s)x\|_{\dn}\geq r_2 \Rightarrow |\dn(s)x|\geq r>1$ for some $r_2\geq \max\{1,r_1\}$.
		Hence, the following representation
		\[
		V(x)=\int^{\ln\frac{r_2}{\|x\|_{\dn}}}_{\ln\frac{r_1}{\|x\|_{\dn}}} \!\!\!e^{-s} b(V_0(\dn(s)x))ds+
		\int^{+\infty}_{\ln\frac{r_2}{\|x\|_{\dn}}} \!\!\!e^{-s} b(c_2)ds 
		\]
		holds and $ \tfrac{b(c_2)}{r_2}\|x\|_{\dn}\leq V(x)\leq \tfrac{b(c_2)}{r_1}\|x\|_{\dn}$.
		Using $\dn$-homogeneity  of the vector field $g$ we derive
		\[
		\frac{\partial V}{\partial x}g(x)=\int^{+\infty}_0 e^{-s} b'(V_0(\dn(s)x))\left.\tfrac{\partial V_0(z)}{\partial z}\right|_{z=\dn(s)x}\dn(s) g(x)ds=
		\]
		\[
		\int^{+\infty}_0 \!\!\!\!e^{-(1+\mu)s} b'(V_0(\dn(s)x))\!\left.\tfrac{\partial V_0(z)}{\partial z} g(z)\right|_{z=\dn(s)x}\!\!\!\!ds\!=
		\]
		\[
		-
		\int^{\ln\frac{r_2}{\|x\|_{\dn}}}_{\ln\frac{r_1}{\|x\|_{\dn}}} \!\!e^{-(1+\mu)s}  b'(V_0(\dn(s)x)) a(|\dn(s)x|) ds=-W(x).
		\]
		By construction, $W\in C(\Omega\backslash\{\zero\})$, $b'(V_0(\dn(s)x))>0$ for $\ln\frac{r_1}{\|x\|_{\dn}}\leq s<\ln\frac{r_2}{\|x\|_{\dn}}$ and $a(|\dn(s)x|)=a(1)>0$ for $\ln\frac{1}{\|x\|_{\dn}}\leq s<\ln\frac{r_2}{\|x\|_{\dn}}$. Hence,  $W(x)>0$ for any $x\in \Omega\backslash\{\zero\}$.  Obviously, $W(\dn(\tau)x)=e^{\tau(\mu+1)}W(x),\forall \tau\in \R, \forall x\in \Omega\backslash\{\zero\}$. Assigning $V(\zero)=W(\zero)=V(x)=W(x)=0$ for $x\in \R\backslash\Omega$ we complete the proof.
	\end{proof}
	\begin{remark}\label{rem:fts}
		Since $W$ is $\dn$-homogeneous of degree $1+\mu$ then $\exists \rho>0$ such that
		$W(x)=W(\dn(-\ln \|x\|_{\dn})x) V^{1+\mu}(x)\geq \rho V^{1+\mu}(x)$. Hence, $\dot V(x)\leq -\rho V^{1+\mu}(x)$ for $x\neq \zero$. If $\mu<0$ then  the system is finite-time stable \cite{BhatBernstein2005:MCSS}:
		%we derive $\frac{V^{-\mu}(x(t))-V^{-\mu}(x_0)}{-\mu}\leq -\rho t$ as long as $x(t)\!\neq\! \zero$.  Therefore, 
		$x(t)=\zero$ for all $t\geq \frac{V^{-\mu} (x_0)}{(-\mu)\rho}$.
	\end{remark}
	\begin{remark}
		  For the ``linear'' $\dn$-homogeneous cone $\Omega=\Omega_{\rm lin}$ given by \eqref{eq:Omega_lin}, the condition \eqref{eq:hom_inv}  can be replaced with the inequality \eqref{eq:inv_lin_cone} due to Corollary \ref{cor:hom_lin_cone}.
	\end{remark}

		\subsection{Robustly Positively Invariant Homogeneous Cones}
	
	The positive invariance of the system \eqref{eq:p_hom_system} with perturbation  $q\in L^{\infty}(\R,\R^k)$  can be studied similarly provided that the solutions of the system are  forward unique and complete. In this case, the  criterion of  the robust positive invariance\footnote{ The set $\Omega$ is said to be robustly positively invariant  for system \eqref{eq:p_hom_system}  if $x(t_0)\in \Omega \Rightarrow x(t)\in \Omega$ for any $q\in L^{\infty}(\R,\R^k) : q(t)\in \mathcal{Q}\subset \R^k$ and for all $t\geq t_0$.} becomes \cite{Aubin1991:Book}, \cite{Blanchini1999:Aut}:
	\begin{equation}\label{eq:p_inv}
		f(x,q)\in C_{\Omega}(x), \quad \forall x\in \partial \Omega, \quad \forall q\in \R^n.
	\end{equation}

		For perturbed homogeneous system \eqref{eq:p_hom_system}, the  uniqueness of solutions at zero may be a rather conservative assumption, which can be slightly relaxed as follows.
		
		\begin{assumption}\label{as:3}
			For any $q\in L^{\infty}(\R,\R^k)$ and for any $x_0\in \partial \Omega\backslash \{\zero\}$, a solution the system \eqref{eq:p_hom_system} is forward unique (at least, locally in time) but a solution initiated as $x(t_0)=\zero$  may leave the origin only  by entering the set $\Omega$.   
		\end{assumption} 
  
		 The simplest example of the standard homogeneous satisfying  this assumption is $\dot x=|q||x|^{1/3}$ with $\Omega=\{x\in \R: x\geq 0\}$. For $q\in \R\backslash\{0\}$ and $x(0)=0$,
		this equation has the zero solution, but this solution is non-unique, e.g., $x(t)=(2|q|t/3)^{3/2}$ is a solution as well.
	Since $|q||x|^{1/3}\geq 0$, then all solutions initiated at $0$ remain in $\Omega$. 
		
		The robust positive invariance of the $\dn$-homogeneous cone $\Omega$ for the perturbed  homogeneous system \eqref{eq:p_hom_system} can also be analyzed considering  the vector field $f$ and the tangent cone on the intersection of the unit sphere with the boundary of $\Omega$.
	
	\begin{corollary}\label{cor:p_hom_inv}
		Under Assumptions \ref{as:1},  \ref{as:2} and \ref{as:3}, the $\dn$-homogeneous cone $\Omega$ is robustly positively invariant for the forward complete system \eqref{eq:p_hom_system} if and only if   
		\begin{equation}\label{eq:p_hom_inv}
			f(x,q)\in \mathcal{C}_{\Omega}(x), \quad   \forall x\in \partial \Omega\cap S,   \quad \forall q\!\in\! \R^k.
		\end{equation}
	\end{corollary}
 
	\begin{proof} 
		Using $\dn$-homogeneity of $\mathcal{C}_{\Omega}$ (see Lemma \ref{lem:hom_tang_cone}) and $f$ (see Assumption \ref{as:2}), from \eqref{eq:p_hom_inv}  we derive $f(x,q)\in C_{\Omega}(x), \forall x\in \partial \Omega\backslash\{\zero\},  \forall q\in  \R^k$. 
		Taking into account Assumption \ref{as:3} we conclude that  the set $\Omega$ is positively invariant.
	\end{proof}

	Below we study the ISS  on $\Omega$  under  Assumption \ref{as:3}, while the ISSf/ISSfS analysis requires the more conservative assumption.
	
		\begin{assumption}\label{as:4}			
			For any $q\in L^{\infty}(\R,\R^k)$ and any $x_0\in \R^n\backslash \interior \Omega$, a solution of the system \eqref{eq:p_hom_system} is forward unique (at least, locally in time).
		\end{assumption}

	\subsection{ISS Analysis on Homogeneous Cones}

	The Zubov-Rosier theorem  is the main tool for ISS analysis of homogeneous  systems on the Euclidean space $\R^n$ \cite{Hong2001:Aut}, \cite{Andrieu_etal2008:SIAM_JCO}, \cite{Bernuau_etal2013:SCL}. 
	The ISS on $\Omega$ is characterized by the following theorem being a generalization of the result 
	\cite{Ryan1995:SCL} to $\dn$-homogeneous systems on cones.

	\begin{theorem}\label{thm:Ryan_Hong_on_Omega}
		Let $f(\cdot,\zero)\in C^1(\R^n\backslash\{\zero\})$ and Assumptions  \ref{as:1}, \ref{as:1bis}, \ref{as:2}, \ref{as:3} be fulfilled. The system \eqref{eq:p_hom_system} is ISS on $\Omega$
		if   the unperturbed ($q\!=\!\zero$) system  is uniformly asymptotically stable on $\Omega$ and the condition \eqref{eq:p_hom_inv} is fulfiled.
		Moreover, if the system \eqref{eq:p_hom_system} is asymptotically stable on $\R^n$ then Assumption \ref{as:1bis}  can be omitted.
	\end{theorem}
 
	\begin{proof} 
	 By Corollary \ref{cor:p_hom_inv}, the condition \eqref{eq:p_hom_inv} guarantees the robust positive invariance of the set $\Omega$ 
			for the system \eqref{eq:p_hom_system} provided the system is forward complete on $\Omega$. To prove the claim we just need to show that the uniform asymptotic stability of the unperturbed ($q=\zero$) system on $\Omega$  (resp., on $\R^n$) implies ISS on $\Omega$ (even if $\Omega_0$ is unbounded).
		The proof is inspired by \cite{Ryan1995:SCL} and \cite{Hong2001:Aut}.
		
		Since the unperturbed system is assumed to be uniformly asymptotically stable on $\Omega$ then, by Theorem \ref{thm:Zubov_Rosier_on_Omega},  there exists a $\dn$-homogeneous Lyapunov function $V$   on $\Omega$ (resp., on $\R^n$) of degree 1. Using $\dn$-homogeneity of $f$ we derive
		\[
		\begin{array}{l}
			\tfrac{\partial V(x)}{\partial x} f(x,q) \!=\!\tfrac{\partial V(x)}{\partial x} f(x,0) \!+\!  \tfrac{\partial V(x)}{\partial x} \left(f(x,q)\!-\!f(x,0)\right)\leq \\
			-W(x)+\|x\|^{1+\mu}_{\dn}\left.\tfrac{\partial V(z)}{\partial z} \left(f(z, \tilde \dn(-\ln \|x\|_{\dn})q)-f(z,\zero)\right)\right|_{z=\dn(-\ln \|x\|_{\dn})x}\leq \\
			-\|x\|_{\dn}^{1+\mu}\left.\left(c+\tfrac{\partial V(z)}{\partial z} \left(f(z, \tilde \dn(-\ln \|x\|_{\dn})q)-f(z,\zero)\right)\right)\right|_{z=\dn(-\ln \|x\|_{\dn})x}
		\end{array}\vspace{-2mm}
		\]
		where $0<c=\inf_{\|z\|_{\R^n}=1}W(z)$.
		Since $f\in C(\R^{n+k})$ then, by Heine-Cantor theorem, $f$ is uniformly continuous on any compact. 
		Hence, taking into account $\|z\|_{\R^n}=1$  we conclude that $f(z, \tilde q)\to f(z,\zero)$ as $\|\tilde q\|_{\R^k}\to 0$ uniformly on $z$ from the unit sphere, and there exists $\delta>0$ such that 
	$ %	\[
		\|\tilde q\|_{\R^k}\leq \delta \quad  \Rightarrow \quad \tfrac{\partial V(z)}{\partial z} \left(f(z, \tilde q)-f(z,\zero)\right)\leq \tfrac{c}{2}
	$ %	\]
		for all $z\in\{z\in \R^{n}:\|z\|_{\R^n}=1\}$. Therefore, we derive 
		\[
		\|\tilde \dn(-\ln \|x\|_{\dn})q\|_{\R^k}\leq \delta \quad \Rightarrow  \quad \dot V(x)\leq -\tfrac{c\|x\|_{\dn}^{1+\mu}}{2}\leq- \tfrac{c}{2k_2^{1+\mu}}V^{1+\mu}.
		\]
		The inequality $\|\tilde \dn(-\ln \|x\|_{\dn})q\|_{\R^k}\leq \delta$ is equivalent to 
		$\|q/\delta\|_{\tilde \dn}\leq \|x\|_{\dn}$, so $V$ is an ISS Lyapunov  function \cite{SontagWang1996:SCL} for the system \eqref{eq:p_hom_system} on $\Omega$ (resp., on $\R^n$).% Using \eqref{eq:phi2x} we complete the proof.
	\end{proof}
	
	\textit{For $\mu<0$ the above theorem provides a characterization of the finite-time ISS  \cite{Hong_etal2010:SIAM_JCO}, since the negative homogeneity degree corresponds to the case of finite-time stability of homogeneous  system} (see, Remark \ref{rem:fts}). 
	
	Repeating the proof of Corollary \ref{cor:inv_linear_cone},  Theorem \ref{thm:Ryan_Hong_on_Omega}  can be adapted to the ``linear" $\dn$-homogeneous cone.
	
	\begin{corollary}\label{cor:Ryan_Hong_on_Omega_lin}
	Let the dilation $\dn$ in $\R^n$ be strictly monotone with respect to the norm $\|x\|=\sqrt{x^{\top}Px}$ with $P\succ 0$. 
		For $\Omega=\Omega_{\rm lin}$ given by \eqref{eq:Omega_lin}, Theorem 
		\ref{thm:Ryan_Hong_on_Omega} remains valid if the condition \eqref{eq:p_hom_inv} is replaced with 
		\begin{equation}\label{eq:p_hom_inv_lin}
			h_i^{\top} f(x,q)\geq \tfrac{x^{\top}Pf(x,q)}{x^{\top}PG_{\dn}x} h_i^{\top}(G_{\dn}-I_n)x, \quad \forall x\in \Xi_i, \quad \forall q\in \R^k, \quad \forall i=1,\ldots,p,
		\end{equation}
		where $\Xi_i:=\{z\in \R^n: \|z\|=1,  h_i^{\top}z=0, h^{\top}_jz\geq 0, \forall j\neq i\}$.
	\end{corollary}
 
	\begin{proof}
	    	In the new coordinates
	$ %	\[
		z=\Psi(x)%:=\|x\|_{\dn}\dn(-\ln \|x\|_{\dn})x
	$ %	\]
			we derive 
		\[
		\dot z= (I_n-G_{\dn})\dn(-\ln \|x\|_{\dn})x\tfrac{\partial \|x\|_{\dn}}{\partial x}\dot x+\|x\|_{\dn}\dn(-\ln \|x\|_{\dn})\dot x=
		\]
		\[
		\|x\|_{\dn}\tfrac{(I_n-G_{\dn})\Psi(x)\Psi^{\top}(x)P\dn(-\ln \|x\|_{\dn})f(x,q)}{\Psi^{\top}(x)PG_{\dn}\Psi(x)}+\|x\|_{\dn}\dn(-\ln \|x\|_{\dn})f(x,q),
		\]
		where the identity \eqref{eq:hom_norm_der} is utilized on the last step.
		Since, by Assumption \ref{as:2}, we have $f(\dn(s)x,\tilde \dn(s)q)=e^{\mu s}\dn(s)f(x,q),\forall s\in \R$ then, taking into account the identity $\|x\|_{\dn}=\|z\|$, 
		the system \eqref{eq:p_hom_system} can be rewritten as follows 
		\begin{equation}\label{eq:z_system_f}
			\dot z=\|z\|^{1+\mu}\tilde f \left(\tfrac{z}{\|z\|},\tilde \dn(-\ln\|z\|)q\right), 
		\end{equation}
		where $
		\tilde f(\tilde x,\tilde g)=\tfrac{\tilde x^{\top}Pf(\tilde x, \tilde q)}{\tilde x^{\top}PG_{\dn}\tilde x}(I_n-G_{\dn})\tilde x+f\left(\tilde x, \tilde q\right),\quad \tilde x\in \R^{n}\backslash\{\zero\},\quad 	\tilde q\in \R^k.$ 	
The $\dn$-homogeneous cone $\Omega_{\rm lin}$ in the $z$-coordinates has the form \eqref{eq:tilde_Omega_lin} then repeating the proof of Corollary \ref{cor:inv_linear_cone}
we conclude \eqref{eq:p_hom_inv_lin} is equivalent to \eqref{eq:p_hom_inv}.
	\end{proof}

		\subsection{ISSf and ISSfS Analysis on Homogeneous Cones}
		 Let us consider the set $\Omega_r\subset \R^n$ given by
		\begin{equation}
			\Omega_r=\{x\in \R^n : \phi_i(x)+r^{\nu_i}_i\geq 0,i=1,2,\ldots,p\}
		\end{equation}
		where $\nu_i>0$ and $r=(r_1,\ldots,r_p)^{\top}\in \R^n_+$ . Obviously, under Assumption \ref{as:1}, it holds $\Omega\subset \Omega_{r}$  and $\partial \Omega_r\in \R^n\backslash \interior \Omega$ for any $r\in \R^p_+$.
		Moreover, $r\leq \tilde r\Rightarrow \Omega_{r}\subset \Omega_{\tilde r}$, where the inequality for the vectors $r,\tilde r\in \R^p_{+}$
		is understood in the component-wise sense.

		\begin{theorem}\label{thm:ISSf}
			Let the system \eqref{eq:p_hom_system} be forward complete for any $q\in L^{\infty}(\R,\R^k)$.
			Under Assumptions \ref{as:1}, \ref{as:2} and \ref{as:4}, the system \eqref{eq:p_hom_system} is ISSf if there exists $r\in \R^{p}_{+}:$ 
			\begin{equation}\label{eq:ISSf_criterion}
					f(x,q) \in C_{\Omega_r}(\xi), \quad \forall x\in  \partial \Omega_r\backslash S, \quad \forall q\in \R^k : \|q\|_{\R^k}\leq 1.
			\end{equation}
		%	Moreover, the functions  $\gamma_i\in \mathcal{K}$ in Definition \ref{def:ISSf} can always be selected as $\gamma_i(\sigma)=(r_i\sigma)^{\nu_i}$, $i=1,\ldots,p$.
		\end{theorem}
  
		\begin{proof}
		The condition \eqref{eq:ISSf_criterion} implies that the set $\Omega_r$ is positively invariant for the system \eqref{eq:p_hom_system} if $\|q\|_{L^{\infty}}\leq 1$.  If $0<\|q\|_{\dn,L^{\infty}_{(t_0,t_1)}}<+\infty$ then
	$\|q_s\|_{\dn,L^{\infty}_{(t_0,t'_1)}}\leq 1$ (or, equivalently, $\|q_s\|_{L^{\infty}_{(t_0,t'_1)}}\leq 1$), where $q_s$ is given by the formula \eqref{eq:q_s} with $s=-\ln \|q\|_{\dn,L^{\infty}_{(t_0,t_1)}}$ and $t'=t_0+e^{-\mu s}(t_1-t_0)$. Since the $\Omega_r$ positively invariant for the system  \eqref{eq:p_hom_system} with any perturbation $\|q\|_{L^{\infty}}\leq 1$ (in particular with $q_s$) then
	$
	\phi_i(x_{q_s}(t))+r_i^{\nu_i}\geq 0
	$ for $i=1,\ldots,p$ and $t\in (t_0,t_1')$. Using the formula \eqref{eq:hom_sol_pert}
	we derive
	\[
	\phi(\dn(s) x_q(t_0+e^{\mu s}(t-t_0)))+r_i^{\nu_i}\geq0
	\]
	for all $t\in (t_0,t'_1)$, where $x_q$ is a solutions of the system with the perturbation $q$.
	Hence, taking into account $s=-\ln \|q\|_{\dn,L^{\infty}_{(t_0,t_1)}}$ and $t'_1=t_0+e^{-\mu s}(t_1-t_0)$ we derive the inequality \eqref{eq:hom_ISSf}, which guarantees ISSf. The proof is complete.
		\end{proof}
		
		The following lemma presents a necessary condition of ISSf for homogeneous systems.
		\begin{lemma}
		If, under Assumptions \ref{as:1}, \ref{as:2} and \ref{as:4}, the system \eqref{eq:p_hom_system} is ISSf then 
		there exists $r\in \R^p_+$ such that 
			\begin{equation}\label{eq:hom_ISSf}
				\phi_i(x(t))+(r_i\|q\|_{\tilde \dn, L^{\infty}_{(t_0,t)}})^{\nu_i}\geq 0
			\end{equation}
			for all $t\!\geq\! t_0$, $x_0\!\in\! \Omega$, $q\!\in\! L^{\infty}(\R,\R^k)$, $i\!=\!1,\ldots,p$, where 
		$ %	\begin{equation}
				\|q\|_{\tilde \dn, L^{\infty}_{(t_0,t)}}\!=\!\mathrm{ess}\sup\limits_{\tau\in (t_0,t)}\|q(\tau)\|_{\tilde \dn}.
		$ %	\end{equation}
		\end{lemma}
  
		\begin{proof}
	Indeed,	on the one hand, if the system \eqref{eq:p_hom_system} is ISSf on $\Omega$ then, in  view of Lemma \ref{lem:hom_norm},
			there exists $\tilde \gamma_i\in \mathcal{K}$ such that 
		$ %	\begin{equation*}
				-\tilde \gamma_i(\|q\|_{\tilde \dn, L^{\infty}_{(t_0,t)}})\leq \phi_i(x_{q}(t,x_0))
		$ %	\end{equation*}
			for all $t\!\geq\! t_0$, $x_0\!\in\! \Omega$, $q\!\in\! L^{\infty}(\R,\R^k)$ and $i\!=\!1,\ldots,p$, where $x_q(\cdot,x_0)$ denotes a solution of the system with  $x(t_0)=x_0$ and the perturbation $q$.
			On the other hand, since  
			$
			-\tilde \gamma_i(\|q_s\|_{\tilde \dn, L^{\infty}_{(t_0,t)}})\leq \phi_i(x_{q_s}(t,\dn(s)x_0))
			$
			for all $t\geq t_0$ then using \eqref{eq:hom_sol_pert}, $\dn$-homogeneity of $\phi_i$ and $\tilde \dn$-homogeneity of $\|\cdot\|_{\tilde \dn}$ we derive
		$ %	\begin{equation*}
				-\tilde \gamma_i(e^s\|q\|_{\tilde \dn, L^{\infty}_{(t_0,t)}})\leq e^{\nu_i s}\phi_i(x_{q}(t,x_0)), \forall s\in \R,  \forall t\geq t_0,
		$ %	\end{equation*}
			where $q_s$ and $x_{q_s}$ are defined in Lemma \ref{lem:hom_sol_pert}. Hence,
			taking 
		$ %	\[
			\gamma_i(\sigma)=\inf_{s\in \R} e^{-\nu_i s}\tilde \gamma_i(e^s\sigma), \sigma\geq 0
		$ %	\]
			we derive 
		$ % 	\begin{equation*}
				- \gamma_i(\|q\|_{\tilde \dn, L^{\infty}_{(t_0,t)}})\leq \phi_i(x_{q}(t,x_0)),  \forall t\geq t_0.
		$ %	\end{equation*}
			By construction, $\gamma_i$ is a standard homogeneous function of the positive degree $\nu_i$ such that $0\leq \gamma_i(\sigma)\leq \tilde \gamma_i(\sigma), \forall \sigma\geq 0$. Hence, $\gamma_i$ is continuous at zero. By homogeneity, $\gamma_i(\sigma)=e^{-\nu_i s}\gamma(e^s \sigma),\forall s\in \R, \forall \sigma\geq 0$, so taking $s=-\ln \sigma$ we derive $\gamma(\sigma)=\sigma^{\nu_i}\gamma_i(1), \forall \sigma\neq 0$, i.e., 
		  the inequality 
			\eqref{eq:hom_ISSf} is fulfilled for $r_i^{\nu_i}=\gamma_i(1)$. 
		\end{proof}
		
		The sufficient ISSf condition \eqref{eq:ISSf_criterion} is rather close to the necessary condition \eqref{eq:hom_ISSf}.
		Indeed, if \eqref{eq:hom_ISSf} holds  $\forall x_0\in \Omega_r$ (but not only for $x_0\in \Omega$) then the set $\Omega_r$ is positively invariant for the system \eqref{eq:p_hom_system} with $\|q\|_{L^{\infty}}\leq 1$. This  is equivalent to \eqref{eq:ISSf_criterion}.

For the ``linear'' $\dn$-homogeneous cone, the ISSf can be established by checking a more simple  algebraic condition for the unperturbed system. 
\begin{corollary}\label{cor:ISSf_on_Omega_lin}Let  $\dn$ be strictly monotone with respect to the norm $\|x\|=\sqrt{x^{\top}Px}$ with $P\succ 0$. 
		For $\Omega=\Omega_{\rm lin}$ given by \eqref{eq:Omega_lin} with $h_i\neq \zero,i=1,\ldots,p$, Theorem 
		\ref{thm:ISSf} remains valid if Assumption \ref{as:4} is omitted but the condition \eqref{eq:ISSf_criterion} is replaced with 
		\begin{equation}\label{eq:ISSf_Omega_lin_0}
		\exists	r\in \R^p_+ \;:\;  h_i^{\top} \!f\left(\tfrac{z}{\|z\|},\zero\right)\!>\! \tfrac{h_i^{\top}(G_{\dn}-I_n)zz^{\top}Pf\left(\frac{z}{\|z\|},\zero\right)}{z^{\top}PG_{\dn}z}, \; \forall z\!\in\! \Xi_i(r), \; i=1,\ldots,p, 
		\end{equation}
		where $\Xi_i(r):=\{z\in \R^n:   h_i^{\top}z+r_i=0, h^{\top}_jz+r_j\geq 0, \forall j\neq i\}$ and $r=(r_1,\ldots,r_p)^{\top}\in \R^p_+$.
		Moreover, if $p=n$, the matrices $H$ and $\mu I_n +G_{\dn}$ are invertible and $f(\cdot,\zero)\in C^1(\R^n\backslash\{\zero\})$ then the inequality \eqref{eq:ISSf_Omega_lin_0} holds provided that the matrix 
		\[
 M(z)=\left.H\!\left(\tfrac{x^{\top}Pf(x,\zero)}{x^{\top}PG_{\dn}x}(I_n-G_{\dn}) +(\mu I_n+G_{\dn})^{-1}\tfrac{\partial f(x,\zero)}{\partial x}G_{\dn}\right)H^{-1}\right|_{x=\frac{z}{\|z\|}}
		\]
		is Metzler and
		\begin{equation}\label{eq:ISSf_Omega_lin_M(x)}
		\exists	r\in \R^n_+\; : \; -e_i^{\top}M(z)r>0, \quad \forall z\in \Xi_i(r), \quad i=1,\ldots,n.  
		\end{equation}
	\end{corollary}
 
	\begin{proof} 
	The repeating the proof of Corollary \ref{cor:Ryan_Hong_on_Omega_lin} we rewrite  \eqref{eq:ISSf_criterion}
	in the form 
	\begin{equation}\label{eq:ISSf_Omega_lin_q}
			h_i^{\top} \!f\left(\!\tfrac{z}{\|z\|},\tilde \dn(-\ln \|z\|)q\!\right)\!>\! \tfrac{h_i^{\top}\!(G_{\dn}-I_n)zz^{\top}\!Pf\left(\frac{z}{\|z\|},\tilde \dn(-\ln \|z\|)q\right)}{z^{\top}PG_{\dn}z} , \;  z\!\in\! \Xi_i(r), \; \|q\|_{\R^k}\!\leq\! 1.
		\end{equation}
	Taking to account the uniform continuity of $f$ on any compact we derive
	\[
	\sup_{\|y\|=1, \|q\|_{\R^k}\leq 1} \left\|f\left(y,\tilde \dn(-\ln \|z\|)q\right)-f\left(y,\zero\right)\right\|\to 0 \text{ as } \|z\|\to +\infty. 
	\]
	so the inequality \eqref{eq:ISSf_Omega_lin_0} implies  the inequality \eqref{eq:ISSf_Omega_lin_q} for all $z:\|z\|\geq r_0$, where $r_0>0$ is a sufficiently large number. Since $h_i$ is assumed to be non-zero then number $r_i$ can be selected large enough to guarantee that $ h_i^{\top}z+r_i=0\Rightarrow \|z\|\geq r_0$. The inequality \eqref{eq:ISSf_Omega_lin_q} guarantees that 
	for any (possibly non-unique) solution $z$ of the equivalent system \eqref{eq:z_system_f} 
	the following implication 
	$$
	h_i^{\top} z(t)\in \Xi_i(r) \Rightarrow h_i^{\top}\dot z(t)>0,
	$$
	holds. The latter means that the set $\tilde \Omega_r=\{z\in \R^n : h_i^{\top}z+r_i\geq 0,i=1,\ldots,p\}$ is \textit{strictly}\footnote{A positively invariant set $\Omega$ is strictly positively invariant if the implication $z(t_0)\in \partial \Omega \Rightarrow z(t)\in \interior\Omega, \forall t>t_0$ holds for any trajectory $z$ of the system.} positively invariant for the equivalent system \eqref{eq:z_system_f}. Hence, the set $\Omega_r$ is strictly positively invariant 
	for the original system. 
	Moreover, if $f(\cdot,\zero)\in C^1(\R^n\backslash\{\zero\})$ then, by Theorem \ref{thm:Euler},
	$f(z/\|z\|,\zero)=\left.(\mu I_n+G_{\dn})^{-1}\tfrac{\partial f(x,\zero)}{\partial x}G_{\dn}\right|_{x=\frac{z}{\|z\|}}$ and the inequality \eqref{eq:ISSf_Omega_lin_0} becomes
	\[
	\left.h_i^{\top}\!\left(\tfrac{x^{\top}Pf(x,\zero)}{x^{\top}PG_{\dn}x}(I_n-G_{\dn}) +(\mu I_n+G_{\dn})^{-1}\tfrac{\partial f(x,\zero)}{\partial x}G_{\dn}\right)\right|_{x=\frac{z}{\|z\|}} z>0, \forall z\in  \Xi_i(r).
	\]
	If $H$ is invertible then the latter inequality can be rewritten as
$ %	\[
	e_i^{\top} M(z)Hz>0, \forall z\in \Xi_i(r).
$ %	\]
	Since $M(z)$ is a Metzler matrix and the function $z\mapsto M(z)$ is continuous and uniformly bounded on $S$ then there exists $\lambda>0$ such that $ M(z)+\lambda I_n$ is a non-negative matrix for all $z\in S$. Since 
	\[
	M(z) Hz=(  M(z)+\lambda I_n) Hz-\lambda Hz=(  M(z)+\lambda I_n) (Hz+r)-
	(  M(z)+\lambda I_n) r-\lambda Hz
	\]
	then taking into account $( M(z)+\lambda I_n) (Hz+r)\in \R^n_+$ for $z\in \Omega_r$  and $e_i^{\top}Hz=-r_i$ for $z\in \Xi_i(r)$ we derive 
$ %	\[
	-e_i^{\top} M(z)r>0, \;\; \forall z\in \Xi_i(r) \quad \Rightarrow \quad 
	e_i^{\top}  M(z)Hz>0, \;\; \forall z\in \Xi_i(r).
$ %	\]
	Therefore, the inclusion \eqref{eq:ISSf_Omega_lin_M(x)} guarantees that the inequality \eqref{eq:ISSf_Omega_lin_0} is fulfilled.
	\end{proof}

The ISSfS 
on $\Omega$ can be characterized similarly to ISS on $\Omega$.
		
		\begin{theorem}\label{thm:ISSfS}
			Under Assumptions \ref{as:1}, \ref{as:2}, \ref{as:4},  the system  \eqref{eq:p_hom_system} is ISSfS on $\Omega$ if 
			the unperturbed ($q=\zero$) system is asymptotically stable on $\R^n$ and the condition \eqref{eq:ISSf_criterion} is fulfilled for some $r\in \R^{n}_+$.
		\end{theorem}
  
		\begin{proof}
			On the one hand, under Assumptions \ref{as:1}, \ref{as:2}, \ref{as:4}, the condition \eqref{eq:ISSf_criterion} implies ISSf of the system \eqref{eq:p_hom_system} on $\Omega$ provided that the system is forward complete. 
			Under Assumption \ref{as:2}, the asymptotic stability of the unperturbed ($q=\zero$) system \eqref{eq:p_hom_system} on $\R^n$
			implies its ISS on $\R^n$
			and the forward completeness. 
			Using the estimates \eqref{eq:phi2x} and \eqref{eq:rel_norm_and_hom_norm_Rn} we complete the proof. 
		\end{proof}
		
		Similarly to ISS analysis, in the case of ``linear'' $\dn$-homogeneous cone, the ISSfS analysis of homogeneous system \eqref{eq:p_hom_system} can be reduced to an investigation of the unperturbed ($q=\zero$) system. However, in addition to asymptotic stability,  the strict positive invariance of the set $\Omega_{r}$ for the unperturbed system is required
		to guarantee ISSfS (see, the inequality  \eqref{eq:ISSf_Omega_lin_0}).

	%%%%%%%%%%%%%%%%%%%%%%%
	%%%%%%%%%%%%%%%%%%%%%%%
	%%%%%%%%%%%%%%%%%%%%%%%
	
	\section{Homogeneous Nonovershooting Stabilizer}\label{sec:hom_nono_st}

	Let us consider the system
	\begin{equation}\label{eq:LTI_system}
		\dot x=Ax+Bu,\quad t>0,
	\end{equation}
	where the pair $A\in \R^{n\times n}$, $B\in \R^{n\times m}$ is assumed to be controllable, $x$ is the system state and $u$ is the control input.
	
	A nonovershooting linear control design is studied in the literature (see, e.g., \cite{PhillipsSeborg1988:IJC},
	\cite{ElKhoury_etal1993:Aut}, \cite{DarbhaBhattacharrya2003:TAC}) allowing the exponential stabilization of the system \eqref{eq:LTI_system}. For example,  if  a desired invariant set for the system is defined  by a linear cone 
	\[
	\tilde \Omega_{\rm lin}=\{ x\in \R^n: h_i^{\top}x\geq 0,i=1,\ldots,n \}
	\]
	with an invertible matrix $H$ given by \eqref{eq:H} then, inspired by \cite{FarinaRinaldi2000:Book}, \cite{LeenheerAeyels2001:SCL}, \cite{Rantzer2015:EJC} the following procedure can be utilized for the design of a nonovershooting linear controller
	\begin{equation}\label{eq:lin_control}
		u_{\rm lin}=Kx.
	\end{equation}
	On the one hand, the set $\tilde \Omega_{\rm lin}$ is positively invariant for the system \eqref{eq:LTI_system}, \eqref{eq:lin_control} if an only if the matrix $\tilde A=H(A+BK)H^{-1}$ is Metzler.  On the other hand,  the Metzler matrix is Hurwitz if an only if there exists a positive $\ell\in \R^n_+$ such that 
	$-\ell^{\top}(A+BK)\in \R^n_+
	$ (see, \cite{Rantzer2015:EJC}).
	Let the pair $(\ell^{\rm opt}, K^{\rm opt})$ be a solution  of the  optimization problem with the bilinear functional
	\begin{equation}\label{eq:bilinear_cost}
		J(p,K):=\max_j \ell^{\top}(A+BK)e_j\to \min
	\end{equation}
	subject to the  linear constraints
	\begin{equation}\label{eq:lin_constr}	
		\begin{array}{ccc}
			-\rho\leq e_i^{\top}H(A+BK)H^{-1}e_i, & 0< \ell_i\leq 1, & i=1,\ldots,n,\\ 
			0\leq e_i^{\top}H(A+BK)H^{-1})e_j, & j\neq i, & j=1,\ldots,n, \\
		\end{array} 
	\end{equation}
	where  $\ell=(\ell_1,\ldots,\ell_n)^{\top}\in \R^{n}_+$, $K\in \R^{m\times n}$ are optimization variables and $\rho>0$ is a tuning parameter required for boundedness of the admissible set.   
	If the optimal cost $J(\ell^{\rm opt},K^{\rm opt})$ is negative then the matrix $\tilde A$ is Metzler and Hurwitz, and 
	the system \eqref{eq:LTI_system} with the linear control \eqref{eq:lin_control} is uniformly asymptotically stable on the cone $\tilde \Omega_{\rm lin}$. In practice, the considered problem can be solved using some optimization software (e.g., YALMIP toolbox for MATLAB).
	
	A possible way to design a nonovershooting finite-time  stabilizer for the system \eqref{eq:LTI_system}  is by  upgrading (a transformation) an existing linear controller to a homogeneous one \cite{Wang_2020:RNC}. A possibility of such an upgrade for the integrator chain  is demonstrated in \cite{PolyakovKrstic2023:TAC}.  The following theorem extends this method to multi-input case.
	
	\begin{theorem}\label{thm:upgrade}
		Let a linear feedback law  \eqref{eq:lin_control} be such that the matrix $A+BK$ is Hurwitz and
		the  linear cone $\tilde \Omega_{\rm lin}$ be  positively invariant for  the closed-loop system \eqref{eq:LTI_system}, \eqref{eq:lin_control}. 
		Let the pair $\{A,B\}$ be controllable.	
		If
		the matrix $H$	given by \eqref{eq:H}
		is invertible and 
		the following linear algebraic system 
		\begin{equation}\label{eq:G0Y0}
			AG_0+BY_0=G_0A+A, \quad G_0B=\zero,  
		\end{equation}
		\begin{equation}\label{eq:Metzler}
			e^{\top}_i H\left(\tau(A+BK)-G_0\right)H^{-1} e_j\geq 0, \quad i\neq j, \quad i,j=1,\ldots, n,
		\end{equation}
		is feasible with respect to $G_0\in \R^{n\times n}$, $Y_0\in \R^{m\times n}$ and $\tau>0$
		then 
		\begin{itemize}
			\item $G_{\dn}=I_n+\mu G_0$ is anti-Hurwitz for any $\mu\in [-1,0)$;
			\item  the system of linear matrix inequalities 
			\begin{equation}\label{eq:LMI}
				P(A+BK)+(A+BK)^{\top}P\prec 0, \quad PG_{\dn}+G_{\dn}^{\top}P\succ 0, \quad P\succ 0
			\end{equation}
			is feasible with respect to $P=P^{\top}\in \R^{n\times n}$, at least, for $\mu\in [-1,0)$ sufficiently close to $0$;
			\item the feedback law
			\begin{equation}\label{eq:hom_control}
				u=K_0 x+\|x\|_{\dn}^{1+\mu}(K-K_0)\dn(-\ln \|x\|_{\dn})x, \quad K_0=Y_0(G_0-I_n)^{-1},
			\end{equation}
			is smooth  on $\R^n\backslash \{\zero\}$, continuous at $\zero$ for $\mu>-1$ and locally bounded on $\R^n$ for $\mu=-1$,
			where the canonical homogeneous norm $\|\cdot\|_{\dn}$ is induced by the norm $\|x\|=\sqrt{x^{\top}Px}$ with $P$ satisfying \eqref{eq:LMI};
			\item the closed-loop system \eqref{eq:LTI_system}, \eqref{eq:hom_control} is 
			$\dn$-homogeneous of degree $\mu\in [-1,0)$ and  finite-time stable on $\R^n$;
				\item the $\dn$-homogeneous cone 
				$
				\Omega_{\rm lin}\!=\!\{x\!\in\! \R^n\!:\! \|x\|_{\dn}h_i^{\top}\dn(-\ln \|x\|_{\dn})x\!\geq\! 0,i\!=\!1,...n\},
				$  
				is positively invariant for the closed-loop system for $\mu\in [-1,0)\cap [-\tau^{-1},0)$.
		\end{itemize}
		%	all conditions of Theorem \ref{thm:NO_hom_control} are fulfilled for $G_{\dn}=I+\mu G_0$, $K_0=Y_0(G_0-I_n)^{-1}$ and the homogeneity degree $\mu\in [-1,0)$ sufficiently close to $0$.
		%for $G_{\dn}=I+\mu G_0$ there exists $P\in \R^{n\times n}$ satisfying \eqref{eq:
	\end{theorem}
 
	\begin{proof}
		If the pair $\{A,B\}$ is controllable then  algebraic equation \eqref{eq:G0Y0} is always feasible with respect to $Y_0$ and $G_0$  \cite{Zimenko_etal2020:TAC} and the matrix $G_{\dn}=I_n+\mu G_0$ is anti-Hurwitz  \cite{Nekhoroshikh_etal2021:CDC} for $\mu\in [-1,1/n]$. Denoting 
		$A_0=A+BK_0$ we derive
		\[
		\eqref{eq:G0Y0} \quad \Rightarrow \quad AG_0-G_0A+BK_0(G_0-I_n)=A,\;\; G_0B=\zero \quad \Rightarrow
		\]
		\[
		(A+BK_0)G_0 =G_0A+\underbrace{G_0B}_{=\zero}K_0+A+BK_0 \quad \Rightarrow 
		\]
		\[
		\mu A_0G_0=\mu (I_n+G_0)A_0 \quad \Rightarrow \quad A_0G_{\dn}=(G_{\dn}+\mu I_n)A_0.
		\] 
		The latter means that the linear vector field $x\mapsto A_0x$ is $\dn$-homogeneous, i.e.,  $A_0\dn(s)=e^{\mu s}\dn(s)A_0, \forall s\in \R$. Moreover, the identity $G_0B=\zero$ implies that $\dn(s)B=e^{s}$.

		The  LMI \eqref{eq:LMI} is always feasible with respect to $P$, at least,
		for $\mu$ sufficiently close to zero. Indeed, 
		since the matrix $A+BK$ is Hurwitz then 
		the first (Lyapunov) inequality in \eqref{eq:LMI}  is feasible together with $P\succ 0$ (see, e.g., \cite{Boyd_etal1994:Book}).	Rewriting the second inequality for $G_{\dn}=I_n+\mu G_0$ as  follows
		$
		2P+\mu(PG_0+G_0^{\top}P)\succ 0,
		$
		we conclude its feasibility at least for $\mu$ close  zero.
		
		The smoothness  of $u$ on $\R^n\backslash\{\zero\}$ follows from the smoothness  of the canonical homogeneous norm $\|\cdot\|_{\dn}$ on $\R^n\backslash\{\zero\}$, which is induced by
		the norm $\|\cdot\|$ being smooth on $\R^n\backslash\{\zero\}$. Since $\|\dn(-\ln \|x\|_{\dn})x\|=1$  and $\|\cdot\|_{\dn}\in C(\R^n)$ then $u$ is continuous at $\zero$ for $\mu>-1$ and locally bounded (but discontinuous at zero) if $\mu=-1$. In the latter case, solutions of the system at $\zero$ are understood in the sense of Filippov \cite{Filippov1988:Book}.

		The  norm  $\|\cdot\|_{\dn}$ is a Lyapunov function of the closed-loop system  \cite{Zimenko_etal2020:TAC}:
		\[
		\frac{d}{dt}\|x\|_{\dn}=\|x\|^{\mu+1}_{\dn}\tfrac{x^{\top}\dn^{\top}(-\ln\|x\|_{\dn})P(A+BK)\dn(-\ln \|x\|_{\dn})x}{x^{\top}\dn^{\top}(-\ln\|x\|_{\dn})PG_{\dn}\dn(-\ln \|x\|_{\dn})x}\leq -\gamma_{\min} \|x\|_{\dn}^{1+\mu},\quad x\neq \zero,
		\]
		where the formula \eqref{eq:hom_norm_der} and the identities  $\dn(s)B=e^{s}$, $A_0\dn(s)=e^{\mu s}\dn(s)A_0$ are utilized to  derive the equation and  $0<\gamma_{\min}=\inf_{\|z\|=1}\tilde \gamma(z)$ for 
		$
		\tilde \gamma(z):=\tfrac{-z^{\top}\!P(A+BK)z}{z^{\top}\!PG_{\dn}z}>0, \quad \forall z\in \R^{n}\backslash\{\zero\}.$
		Since $\mu<0$ then the closed-loop system is globally  finite-time stable on $\R^n$.
		By Corollary \ref{cor:inv_linear_cone} , the set $\Omega_{\rm lin}$ is positively invariant if the matrix 
		\begin{equation}\label{eq:Mz}
			M(z)=H(A+BK+\mu\tilde \gamma(z)G_{0})H^{-1}
		\end{equation}
		is Metzler.  By assumption, the set $\tilde \Omega_{lin}$ is the positively invariant set of the linear system \eqref{eq:LTI_system}
		with the linear control \eqref{eq:lin_control}. The latter means that the matrix $H(A+BK)H^{-1}$ is necessarily Metzler.
		If the condition \eqref{eq:Metzler} is fulfilled then the matrix 
		$H \left(A+BK+\mu \gamma(z)G_{0}\right)H^{-1}$ is Metzler  provided that $\mu\in [-1,0)\cap [-\tau^{-1},0)$.
	\end{proof}

%	\begin{remark}
		The case $\mu>0$ can be treated by replacing the sign ``$-$" with ``$+$" in  \eqref{eq:Metzler}.
%	\end{remark}

	The identity $(A+BK_0)G_{\dn}=(G_{\dn}+\mu I_n)(A+BK_0)$ is necessary and sufficient \cite{Zimenko_etal2020:TAC} for the linear vector field  $x\mapsto (A+BK_0)$ to be $\dn$-homogeneous of degree $\mu$. Therefore, the linear term in the control law \eqref{eq:hom_control} "homogenize"  the linear plant with respect to the dilation $\dn$. The later is required for $\dn$-homogeneous stabilization of the LTI system. The second inequality in  \eqref{eq:LMI} guarantees monotonicity of the dilation $\dn$ with respect to the norm $\|x\|=\sqrt{x^{\top}Px}$, which is needed for the existence of the canonical homogeneous norm $\|\cdot\|_{\dn}$ (see,  the control law \eqref{eq:hom_control}). The cone $\Omega_{\rm lin}$ is linear in the space $\R^n_{\dn}$ homeomorphic to $\R^n$ and the closed-loop system in $\R^n_{\dn}$ is similar to linear as well (see the formulas \eqref{eq:z_system_g} and \eqref{eq:y_system_g}).
	So,  the invariance of $\Omega_{\rm lin}$ was deduced from positivity of the system in the new coordinates. To be positive, a linear system must have a Metzler matrix \cite{FarinaRinaldi2000:Book}. The condition \eqref{eq:Metzler} just follows this  criterion.   If the matrix  $H(A+BK)H^{-1}$ is Metzler then the condition \eqref{eq:Metzler} simply means that \textit{for any negative  off-diagonal element of the matrix $H(-G_0)H^{-1}$ the corresponding off-diagonal element of the  matrix $H(A+BK)H^{-1}$ has to be positive}. 
	
	In practice, frequently,  just some coordinates of the system state are constrained. For example, in \cite{PolyakovKrstic2023:TAC}  the safe set is given by
	\[
	\Sigma=\{x\in \R^n: \tilde h^{\top}_1x\geq 0\},
	\]
	while the restriction of $\Sigma$ to the invariant sets $\Omega_{\rm lin}$ or $\tilde \Omega_{\rm lim}$ was caused by 
	the design methodology and the structure of the plant model (the integrator chain). In this case,  it is important to know if 
	%\begin{itemize}
		%\item 
		a) some linear constraint  $h_1^{\top}x\geq 0$ can be preserved in the $\dn$-homogeneous invariant cone $\Omega_{\rm lin}$;
		%\item 
		b) the set $\Omega_{\rm lin}$ is larger or smaller than $\tilde \Omega_{\rm lin}$. 
%	\end{itemize}
	The following remarks answer the above questions. 
	\begin{remark}
		If $h_1\in \R^n$ is a real left eigenvector of the generator $G_{\dn}$ then
		\begin{equation}\label{eq:equiv_hom_lin_cone}
			\|x\|_{\dn}h_1^{\top}\dn(-\ln \|x\|_{\dn})x\geq 0 \quad \Leftrightarrow \quad h_1^{\top}x\geq 0.
		\end{equation}
	\end{remark} 
	\begin{proof}
		Indeed, since $\dn(s)=\sum_{i=0}^{\infty} \frac{s^iG_{\dn}^i}{i!}$ and $h_1^{\top}G_{\dn}=\lambda h_1^{\top}$, then
		$h_1^{\top}\dn(s)=e^{\lambda s}h_1^{\top}$ and 
		\[
		\|x\|_{\dn}h_1^{\top}\dn(-\ln \|x\|_{\dn})x=\|x\|^{1-\lambda}_{\dn}h_1^{\top}x. 
		\]
		Hence, \eqref{eq:equiv_hom_lin_cone} holds for $x\neq \zero$. For  $x=\zero$ the equivalence is obvious, since $\|\cdot\|_{\dn}\in C(\R^n)$, $\|\zero\|_{\dn}=0$ and $\|\dn(-\ln \|x\|_{\dn})x\|=1$ for all $x\neq \zero$.
	\end{proof}
	\begin{remark}
		Under conditions of Theorem \ref{thm:upgrade},  if the matrix $H(-G_0)H^{-1}$ is Metzler then the following inclusion 
		\begin{equation}\label{eq:local_emedding_Omega_lin}
			(\tilde \Omega_{\rm lin}\cap B)\subset (\Omega_{\rm lin} \cap B)		
		\end{equation} 
		holds, where $B=\{ z\in \R^n: \|z\|\leq 1\}$ is the unit ball. 
	\end{remark}
	\begin{proof}
		Since 
		$
		e^{s}H\dn(-s)H^{-1}\!=\!e^{s}He^{-s(I_n+\mu G_0)}H^{-1}=He^{-s\mu G_0}H^{-1}=
		e^{-s\mu HG_0H^{-1}}
		$
		then for the Metzler matrix $H(-G_0)H^{-1}$ and for $\mu s>0$ the matrix
		$ % \[
		e^{s}H\dn(-s)H^{-1}\geq 0,
		$ % \]
		where the sign "$\geq 0$" means that all elements of the matrix are non-negative.
		On the other hand, 
		$ %\[
		x\in \tilde \Omega_{\rm lin}  \Rightarrow Hx\geq 0,
		$ %\]
		and 
		$ %\[
		x\in \Omega_{\rm lin}  \Rightarrow (e^{s}H\dn(-s)H^{-1})Hx\geq 0  \text{ with }  s=\ln \|x\|_{\dn}.
	$ %	\]
		If $\|x\|\leq 1$ then $s<0$, $\mu s\geq 0$ (due to $\mu\leq 0$) and $e^{s}H\dn(-s)H^{-1}\geq 0$. Hence,
	$ %	\[
		x\in \tilde \Omega_{\rm lin}\cap B  \quad \Leftrightarrow \quad \|x\|\leq 1 \text{ and }  Hx\geq 0\; \Rightarrow \;
		(e^{s}H\dn(-s)H^{-1}) Hx\geq \zero.
	$ %	\]
	\end{proof}
	
	The latter remark  results in the less conservative overlapping of the original safe set $\Sigma$ 
	by the homogeneous cone $\Omega_{\rm lin}$, at least,
	close to the origin (in the unit ball).  This property is important for the safety filter design \cite{Abel_etal2022:ACC}, \cite{PolyakovKrstic2023:TAC}. To avoid possible conservatism away from the unit ball, the upgrade of 
	linear nonovershooting stabilizer may be realized locally. Namely, if, under condition of Theorem \ref{thm:upgrade}, the control $u$ is defined as follows
	\begin{equation}\label{eq:mix_control}
		u=\left\{
		\begin{array}{ccc}
			Kx & \text{ if } & \|x\|\geq 1\\
			K_0x+\|x\|_{\dn}^{\mu+1}(K-K_0)\dn(-\ln \|x\|_{\dn})x & \text{ if } & \|x\|<1,
		\end{array}
		\right.
	\end{equation}  
	then $u$ is Lipschitz  continuous away from the origin and 
	the system \eqref{eq:LTI_system} with the control \eqref{eq:mix_control} is finite-time stable on the set
	\begin{equation}
		\Omega=(\Omega_{\rm lin}\cap B)\cup (\tilde \Omega_{\rm lin}\backslash B).
	\end{equation}
	
	Finally, in many cases, the linear model  \eqref{eq:LTI_system} is an approximation of a nonlinear control system under assumption that some (probably uncertain time-varying) parameters are small enough to be neglected. Let a more precise model of a control system have the form \eqref{eq:p_hom_system} with $f(x,q)$ such that
	$
	f(x,\zero)=Ax+Bu(x),
	$
	where  $q$ defines uncertainties and disturbances of the original nonlinear system 
	and $u$ is the nonovershooting  feedback \eqref{eq:hom_control}.
	If the vectors  $h_i\in \R^n,i=1,\ldots,n$ are linearly independent then $\Omega_{\rm lin}$ satisfies Assumption \ref{as:1} and the robustness (ISS, ISSf and ISSfS) of the nonovershooting property for the original nonlinear system follows from the results presented in Section  \ref{sec:hom_ISS_on_omega}.

	\begin{corollary}\label{cor:ISSf_cont_system}
	 If $\mu>-1$ then, under conditions of Theorem \ref{thm:upgrade}, the closed-loop system 
	 \eqref{eq:LTI_system}, \eqref{eq:hom_control} is ISSf and ISSfS with respect to measurement noises and additive perturbations provided that the matrix $M$ given by 
	 \eqref{eq:Mz} satisfies the condition \eqref{eq:ISSf_Omega_lin_M(x)}. The condition \eqref{eq:ISSf_Omega_lin_M(x)} is always fulfilled for $\mu$ being sufficiently close to zero.
	\end{corollary}
 
	\begin{proof}
	Let us consider the system
$ %	\[
	\dot x=f(x,q):=Ax+Bu(x+q_1)+q_2,
	$ %\]
	where $A,B,x$ and $u$ are as in Theorem \ref{thm:upgrade}, $q_1,q_2\in L^{\infty}(\R,\R^n)$ and $q=(q_1^{\top},q_2^{\top})^{\top}\in L^{\infty}(\R,\R^{2n})$. 
	For $\mu>-1$ the vector field $f$ satisfies Assumption \ref{as:2} with the dilation $\tilde \dn(s)=\left(
	\begin{smallmatrix}
	\dn(s) & \zero\\
	\zero & e^{\mu s}\dn(s)
	\end{smallmatrix}\right)
	$
	in $\R^{2n}$. In the proof of Theorem \ref{thm:upgrade} it is shown that the matrix $M$ is Metzler.
	By assumption, the condition \eqref{eq:ISSf_Omega_lin_M(x)} holds then, by Corollary 
	\ref{cor:ISSf_on_Omega_lin}, the closed-loop system 
	 \eqref{eq:LTI_system}, \eqref{eq:hom_control} is ISSf on $\Omega_{\rm lin}$.
	 The global asymptotic stability of the unperturbed system implies ISSfS in the view
	 of Theorem \ref{thm:ISSfS}. Moreover, since $H(A+BK)H^{-1}$ is Metzler and Hurwitz
	 then there exists a strictly positive vector $r\in \R^n_+$ such that the vector  $H(A+BK)H^{-1}r$ is strictly negative (see, e.g., \cite{Rantzer2015:EJC} for more details). Hence, \eqref{eq:ISSf_Omega_lin_M(x)} is always fulfilled if $\mu$ is close to zero. 
	\end{proof}
	
	It is worth stressing that neither linear nor homogeneous non-overshooting stabilizer design uses any decomposition of the system to a canonical form (such as the integrator chain in the papers \cite{Abel_etal2022:ACC}, \cite{PolyakovKrstic2023:TAC}). So, the control design does not	require any information about relative degrees of the barrier functions $\phi_i$.

	%%%%%%%%%%%%%%%%%%%%%%%
	%%%%%%%%%%%%%%%%%%%%%%%
	%%%%%%%%%%%%%%%%%%%%%%%

	\section{Numerical Example}
	
	Let us consider the system \eqref{eq:LTI_system}, \eqref{eq:lin_control} with  the parameters:
	$ %\begin{equation}
		A\!=\!\left(\begin{smallmatrix}
			3 &0 &1\\
			0 &-1 &1\\
			-2& 0& 0
		\end{smallmatrix}
		\right)\!, B\!=\!\left(\begin{smallmatrix}
			1 & -1\\
			0 &1\\
			0 &1\\
		\end{smallmatrix}
		\right)\!.
	$ %\end{equation}
	First, let us design a linear feedback $u_{\rm lin}=Kx$ which  stabilizes asymptotically the system preserving the relations
	\begin{equation}\label{eq:restriction}
		-x_1(t)\leq x_3(t)\leq x_2(t), \quad \forall t\geq 0.
	\end{equation} 
	These constraints mean that  the linear  cone 
	\[
	\Sigma=\{x\in \R^3: \,h_1^{\top}x\geq 0, \,h_2^{\top}x\geq 0\},
	\quad 
	h_1=\left(\begin{smallmatrix}
		1 \\
		0 \\
		1
	\end{smallmatrix}
	\right),\quad  h_2=\left(\begin{smallmatrix}
		0\\
		1\\
		-1
	\end{smallmatrix}
	\right)
	\]
	must be positively invariant for the closed-loop system.
	Selecting the additional (virtual) constraint by assigning
$ %	\[
	h_3=\left(\begin{smallmatrix}
		0  \\ 0 \\  -1
	\end{smallmatrix}
	\right).
$ %	\]
	we derive the invertible matrix $H$ given by \eqref{eq:H}. 
	Solving the optimization problem \eqref{eq:bilinear_cost}, \eqref{eq:lin_constr} for $\rho=4$ we obtain the gain of the linear feedback 
	$
	K=
	\left(
	\begin{smallmatrix}
		-4.7536  &  0  & -4.9393\\
		1.7415  & 0 &  -3.7856\\
	\end{smallmatrix}
	\right),
	$
	which stabilizes the system asymptotically on the linear cone $\tilde \Omega_{\rm lin}\subset\Sigma$. 
	
	To stabilize the system in a finite time to zero  we upgrade the linear feedback to a homogeneous one using Theorem \ref{thm:upgrade}. Solving the equation \eqref{eq:G0Y0} we derive 
$ %	\[
	G_{0}\!=\!\left(\begin{smallmatrix}
		0  & -0.5 &  0.5\\
		0 &   -0.5 &  0.5\\
		0 &  0.5 &  -0.5
	\end{smallmatrix}
	\right) \text{ and } K_0\!=Y_0(G_0-I_3)^{-1}=\!\left(\begin{smallmatrix}   
		-1  & 0 &  -1\\
		1   & 0.5 &  -0.5
	\end{smallmatrix}
	\right)\!.
$ %	\]
	Since the matrices 
	\begin{equation}\label{eq:temp2}
	H(A+BK)H^{-1}\!=\!\left(\begin{smallmatrix}
		-3.7536  &  0   & 0.1857\\
		2 &  -1  &  2\\
		0.2585  &  0  & -3.5271
	\end{smallmatrix}
	\right)\!\!,\;\;
	H(-G_0)H^{-1}\!=\!\left(\begin{smallmatrix}
		0 &  0 &  0\\
		0   & 1  &  0\\
		0 & 0.5 &   0\end{smallmatrix}
	\right) 
	\end{equation} 
	are Metzler then the condition \eqref{eq:Metzler} is fulfilled for any $\tau>0$. This means that any  homogeneity degree  $\mu\in [-1,0)$  can be selected for the upgrade.
	Solving the LMI \eqref{eq:LMI} for $\mu=-0.75$ we derive 
	$
	P=\left(\begin{smallmatrix}
		0.8707   & 0.2572 &  -0.1918\\
		0.2572   & 1.0229   &-0.3984\\
		-0.1918 &  -0.3984  &  0.9301
	\end{smallmatrix}
	\right)
	$
	and define the  homogeneous controller  \eqref{eq:hom_control} with $G_{\dn}=I_3+\mu G_0$, which stabilizes the system in a finite time on the $\dn$-homogeneous cone $\Omega_{\rm lin}$. Since $h_1$ and $h_2$ are left eigenvectors of $G_{\dn}$ then $\Omega_{\rm lin} \subset\Sigma$.
	
	The simulation results for the nonovershooting linear and homogeneous control (with $\mu=-0.75$) are given in Figure \ref{fig:non_pert}, where the behavior of the barrier functions $\phi_1(x)=h_1^{\top}x$
	and $\phi_2(x)=h_2^{\top}x$ is depicted as well.  
	The linear control stabilizes the system exponentially and  the homogeneous control stabilizes it in a finite time ($\approx 3$ for $x(0)=(0.5,1,0)^{\top}$). 
	The safety constraint $x(t)\in \Sigma, \forall t\geq 0$ is fulfilled for both stabilizers, while the virtual constraint $x_3(t)\leq 0,\forall t\geq 0$ is violated by the homogeneous stabilizer to guarantee finite-time convergence. This highlights the less conservative overlapping of the set $\Sigma$ by the $\dn$-homogeneous cone $\Omega_{\rm lin}$ (at least, close to the origin) comparing with the linear cone. 
	
	\begin{figure}[t]
		\centering
		\includegraphics[width=53mm]{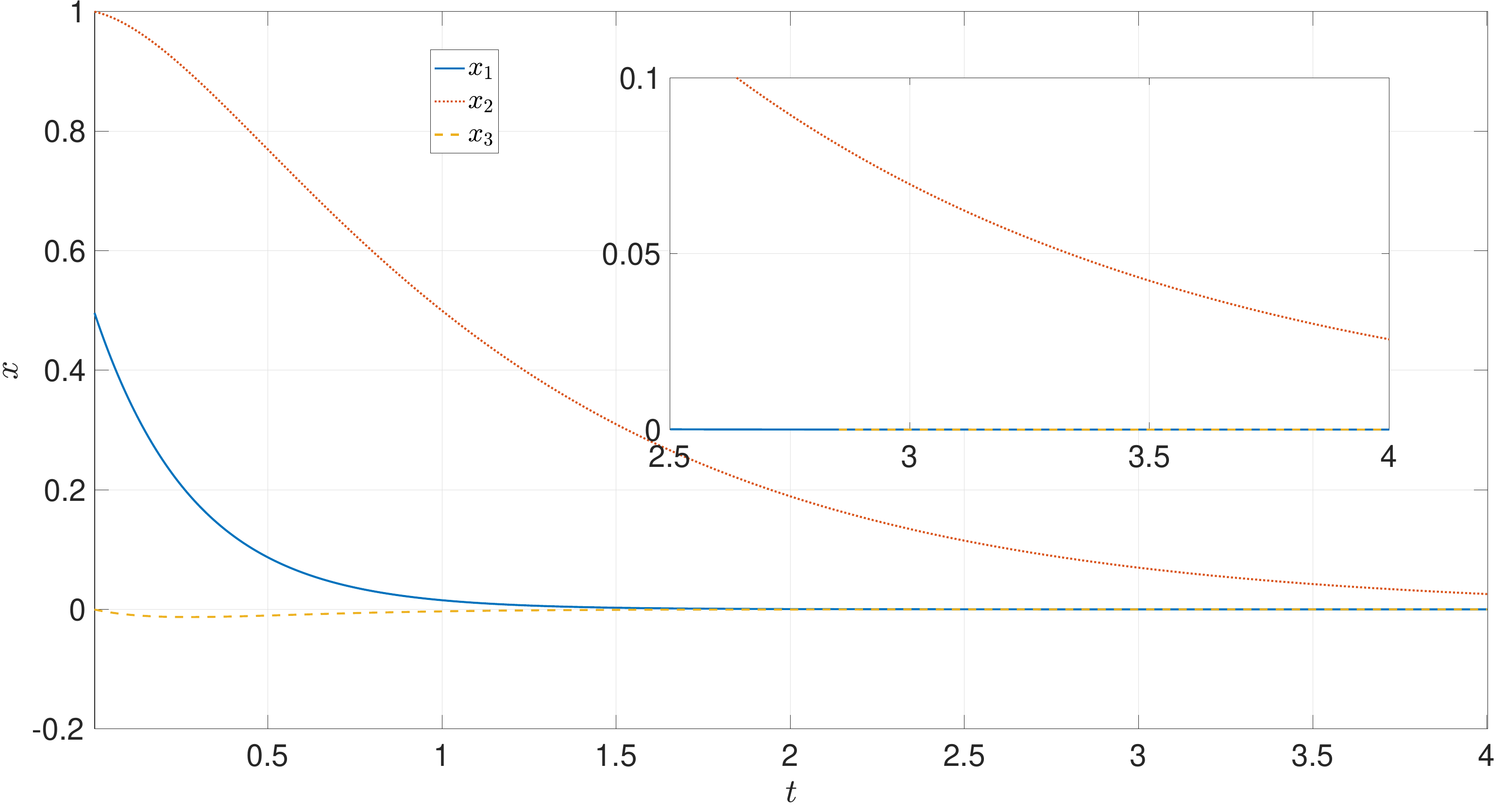}\quad \includegraphics[width=53mm]{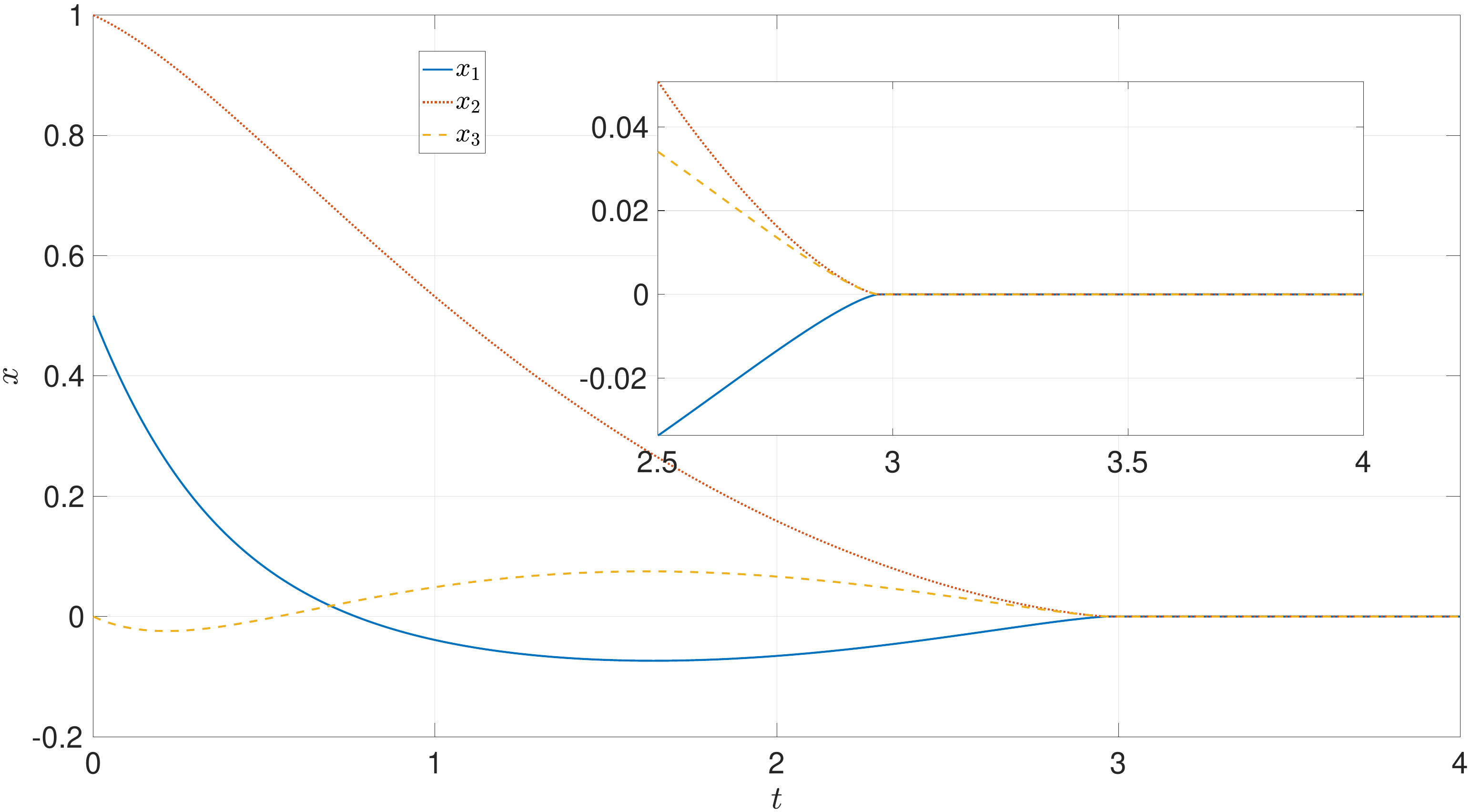}\\
		\includegraphics[width=53mm]{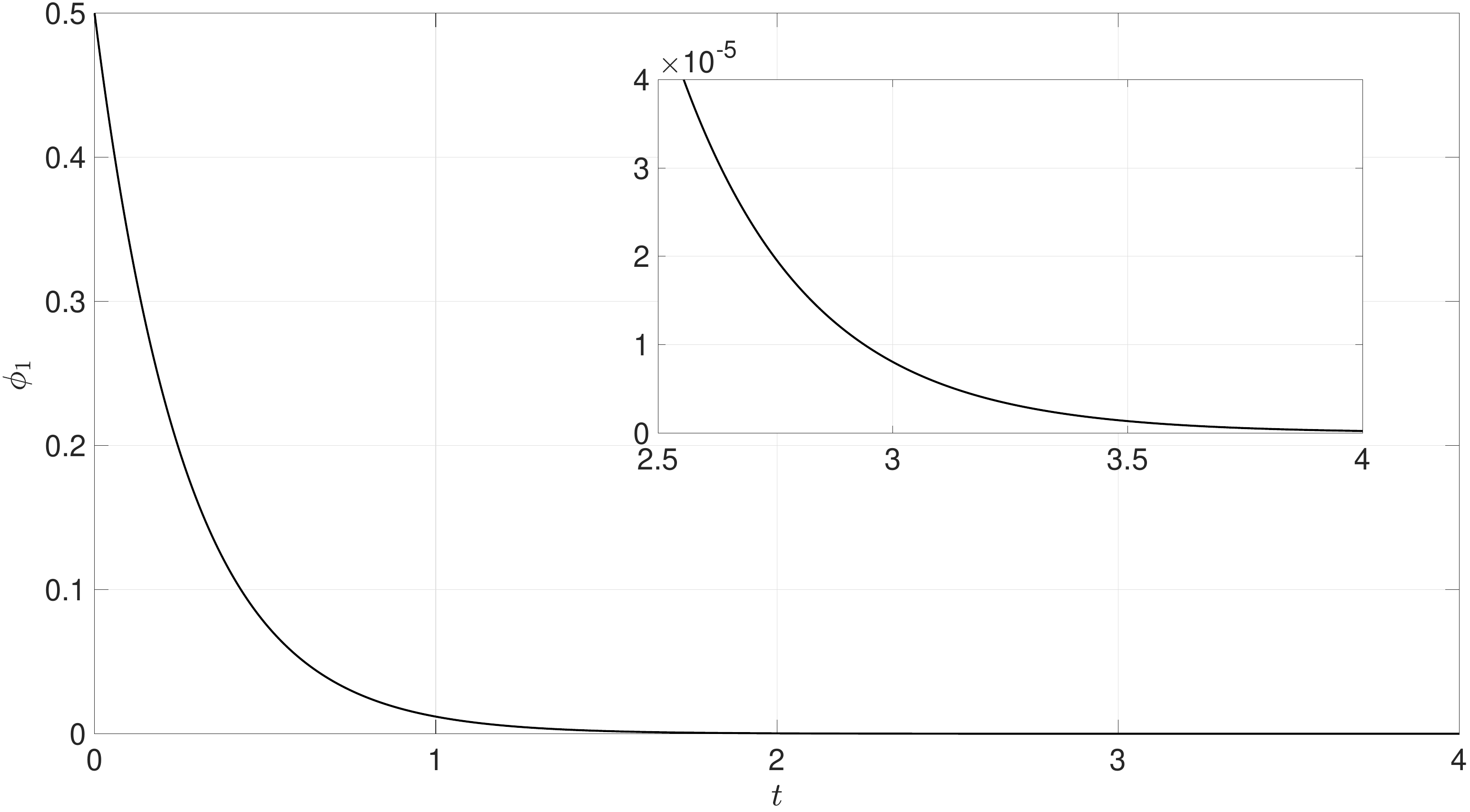}\quad \includegraphics[width=53mm]{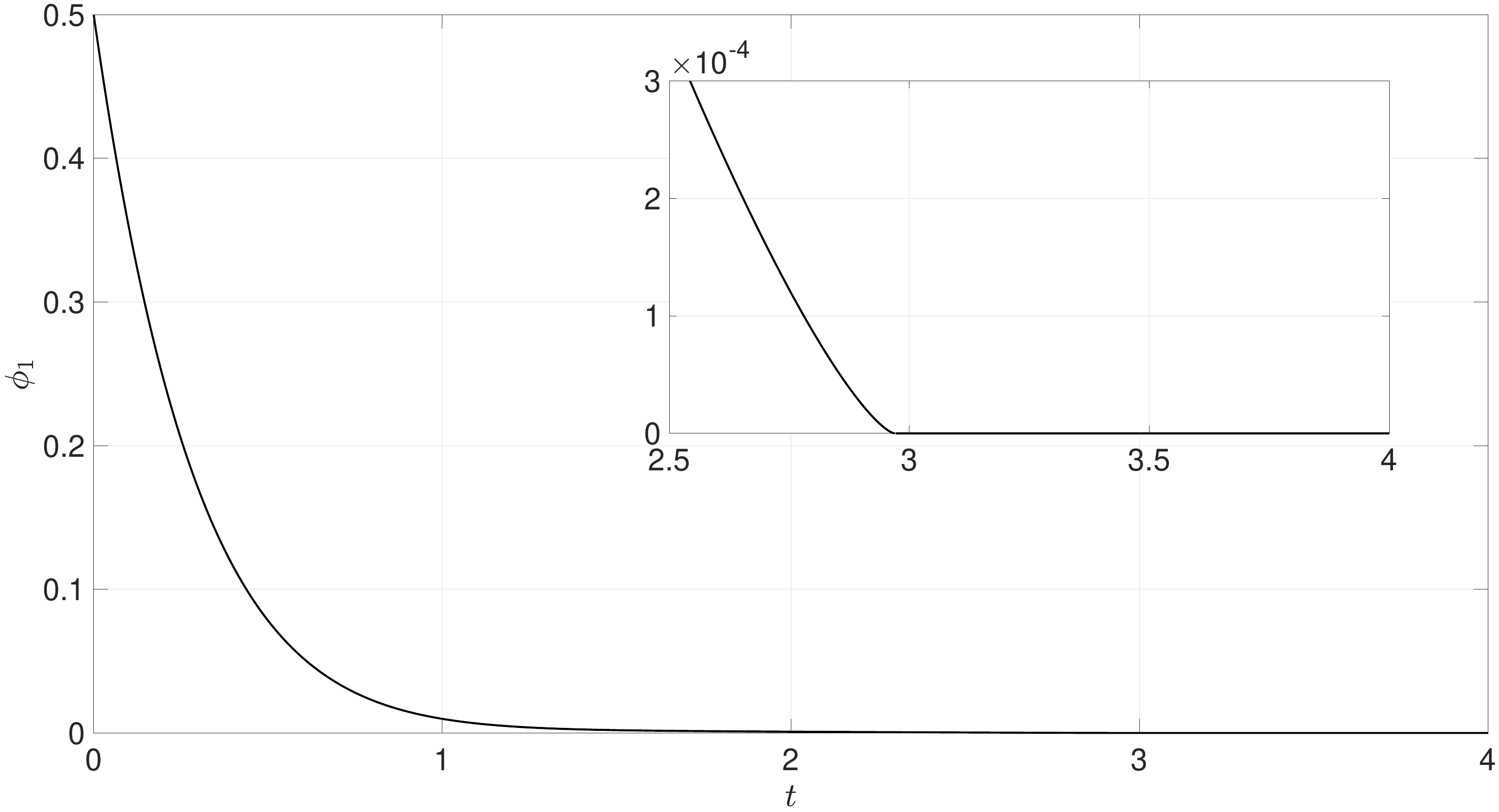}\\
		\includegraphics[width=53mm]{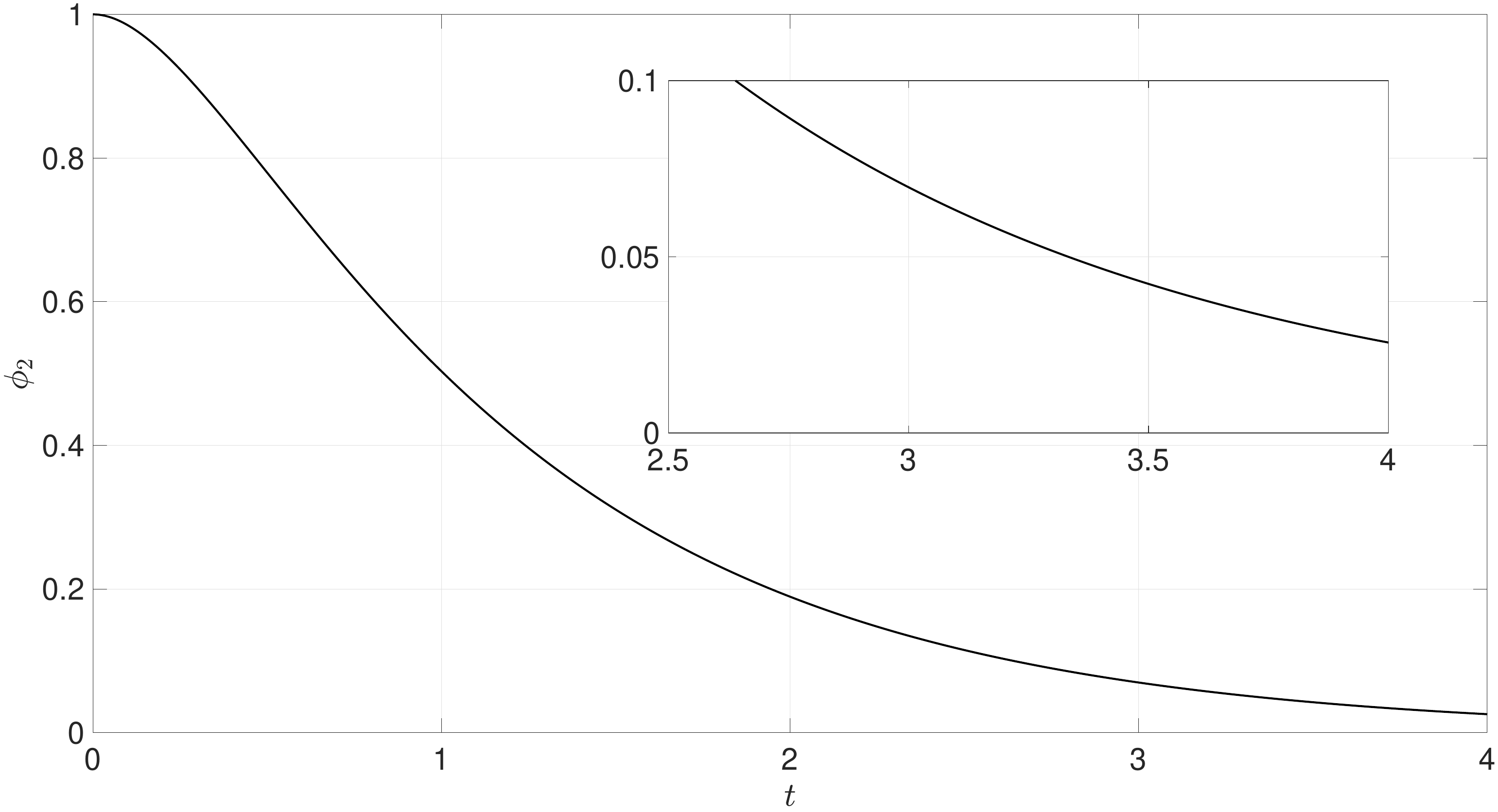}\quad \includegraphics[width=53mm]{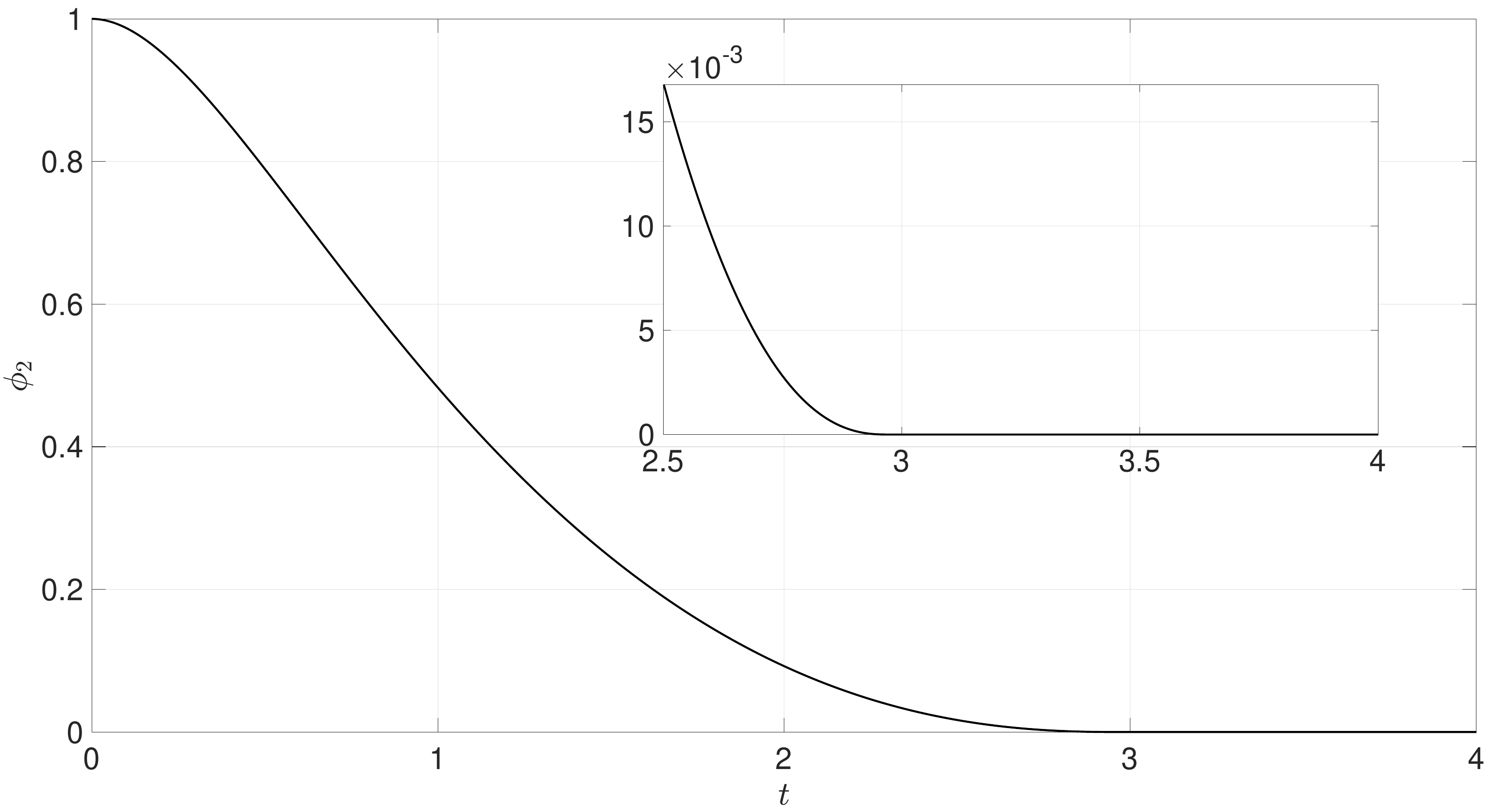}\\
			\vspace{-1mm}
		\caption{Trajectories of the nominal system \eqref{eq:LTI_system} with linear (left) and homogeneous (right) nonovershooting stabilizers }
		\label{fig:non_pert}
		\vspace{-1mm}
	\end{figure} 
	
	\subsection{ISS  on $\Sigma$}
	Let us consider the perturbed system
	\begin{equation}\label{eq:f(x,q)_example}
		\dot x=f(x,q):=Ax+Bu(x)+D|qx_1|^{\nu}, \quad D=\left(
		\begin{smallmatrix}
			1\\	0\\0
		\end{smallmatrix} \right) \quad 0<\nu<1+\mu, 
	\end{equation}
	where $\mu>-1$, $x=(x_1,x_2,x_3)^{\top}$, $A$, $B$ are as before,  $u$ is either a linear or a homogeneous nonovershooting stabilizer and  $q\in L^{\infty}(\R,\R)$ is an uncertain parameter. 
	Since 
	$\dn(s)D=e^s\sum_{i=0}^{\infty} \tfrac{s^i\mu^i G_0^i}{i!}D=e^{s}D$
	then
	taking $\tilde \dn(s) =e^{\left(\frac{1+\mu}{\nu}-1\right)s}$,  for $u$ given by \eqref{eq:hom_control}, we derive 
	$
	f(\dn(s)x,\tilde q)=e^{\mu s}\dn(s)f(x,q), \forall s\in \R,x\in\R^n,q\in \R.
	$
	The group $\tilde \dn(s)$ is a dilation in $\R$ if $0<\nu<1+\mu$.
	Hence, taking into account global asymptotic stability of the system \eqref{eq:f(x,q)_example} for $q=\zero$ we conclude ISS with respect to $q$ on $\R^n$. To prove invariance of $\Sigma$ in the perturbed case  we check the condition \eqref{eq:p_hom_inv_lin} for $p=2$. Since 
	$
	h_1^{\top}G_{\dn}x=h_1 x, \quad h_2^{\top}G_{\dn}x=(1-\mu)h_2^{\top}x
	$
	then \eqref{eq:p_hom_inv_lin} is fulfilled for $p=2$ and the closed-loop system is ISS on $\Sigma$ provided that Assumption \ref{as:3} is fulfilled. 
	
	The simulations results for the perturbed case ($q=\sin(5t),\nu=1/8$) are shown on Figure \ref{fig:pert}.
	The homogeneous stabilizer demonstrates a better suppression of perturbations (at least for the selected initial condition  and the selected perturbation).  
	
	\begin{figure}[t]
		\centering
		\includegraphics[width=53mm]{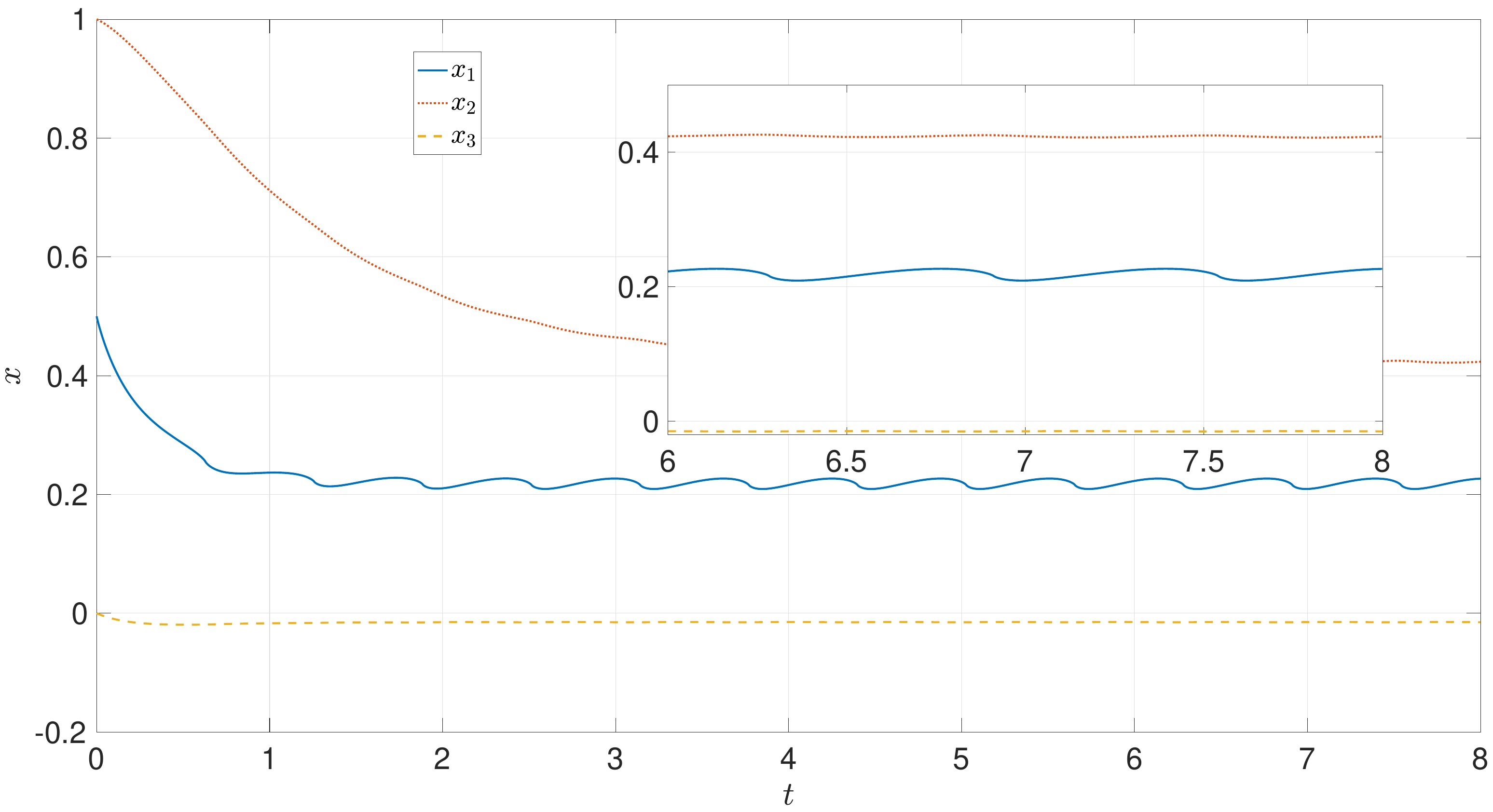}\quad 
		\includegraphics[width=53mm]{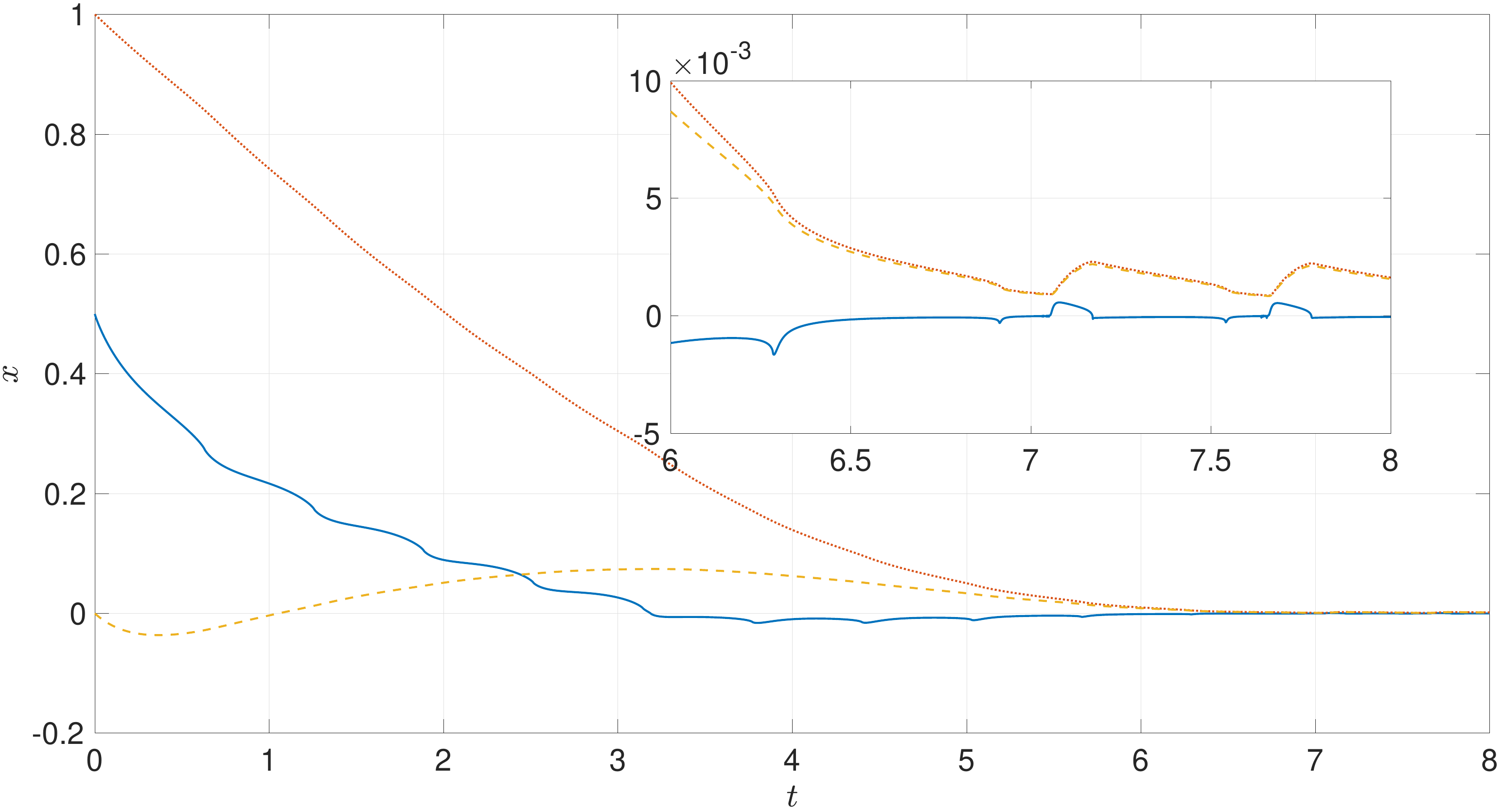}\\
		\includegraphics[width=53mm]{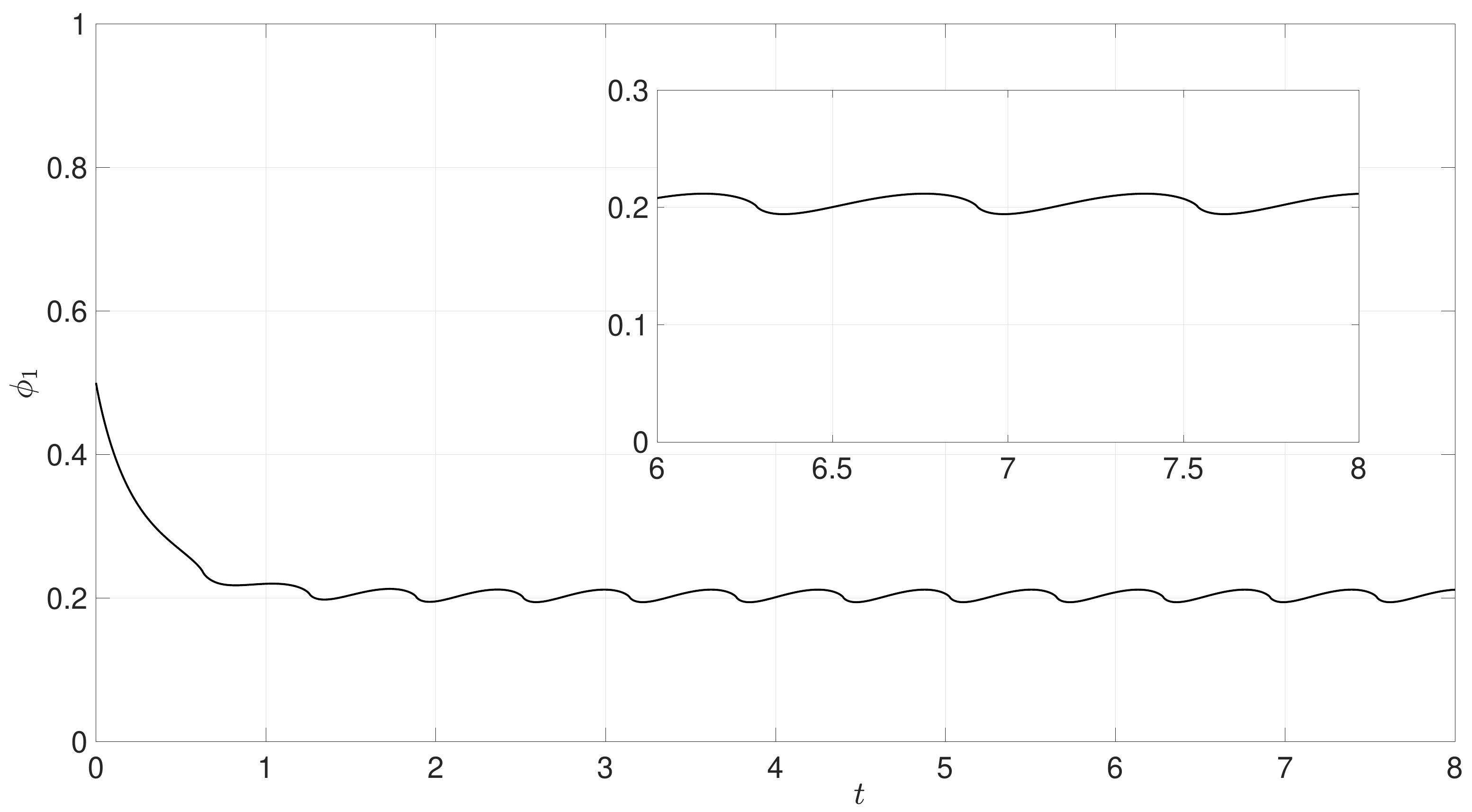}\quad \includegraphics[width=53mm]{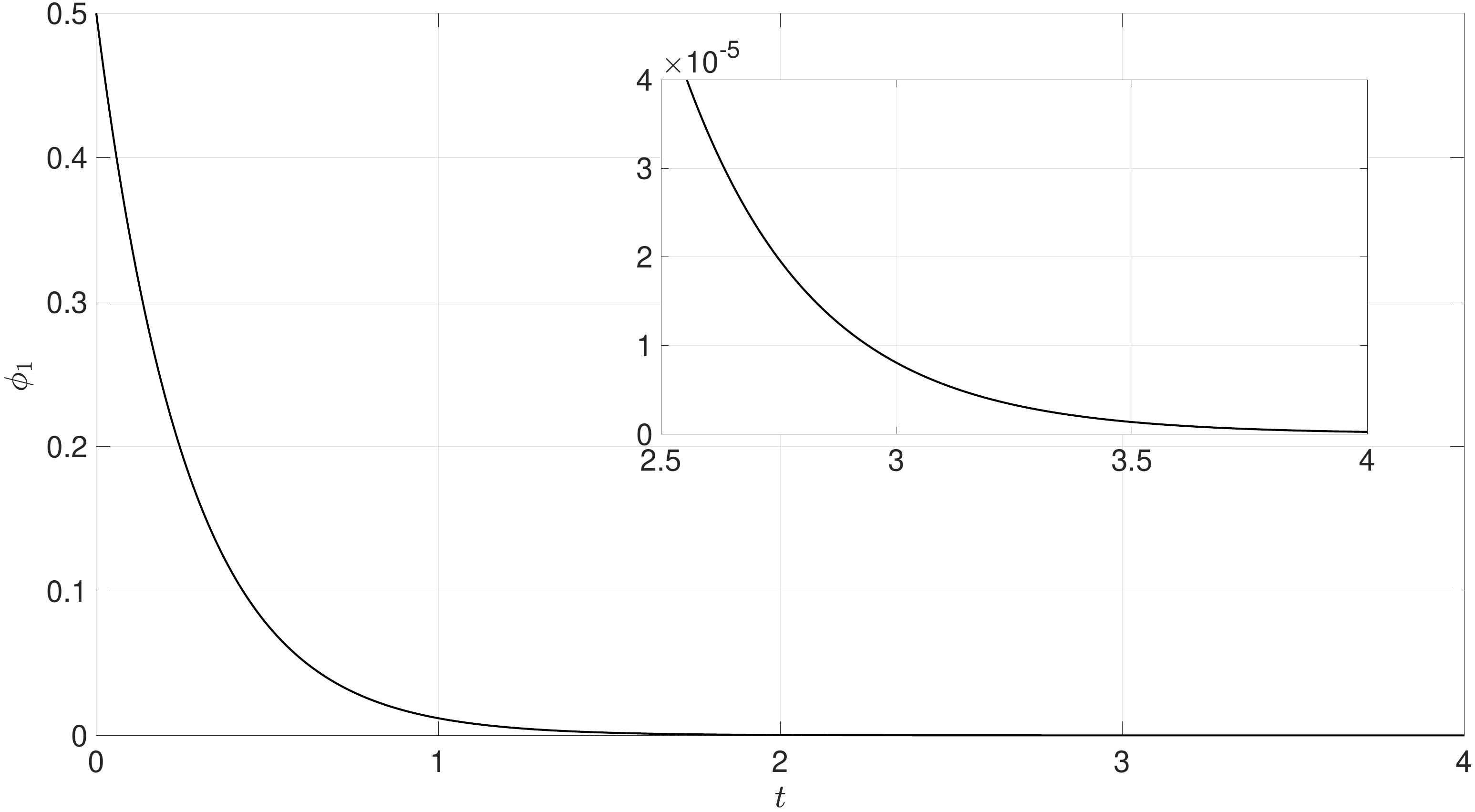}\\
		\includegraphics[width=53mm]{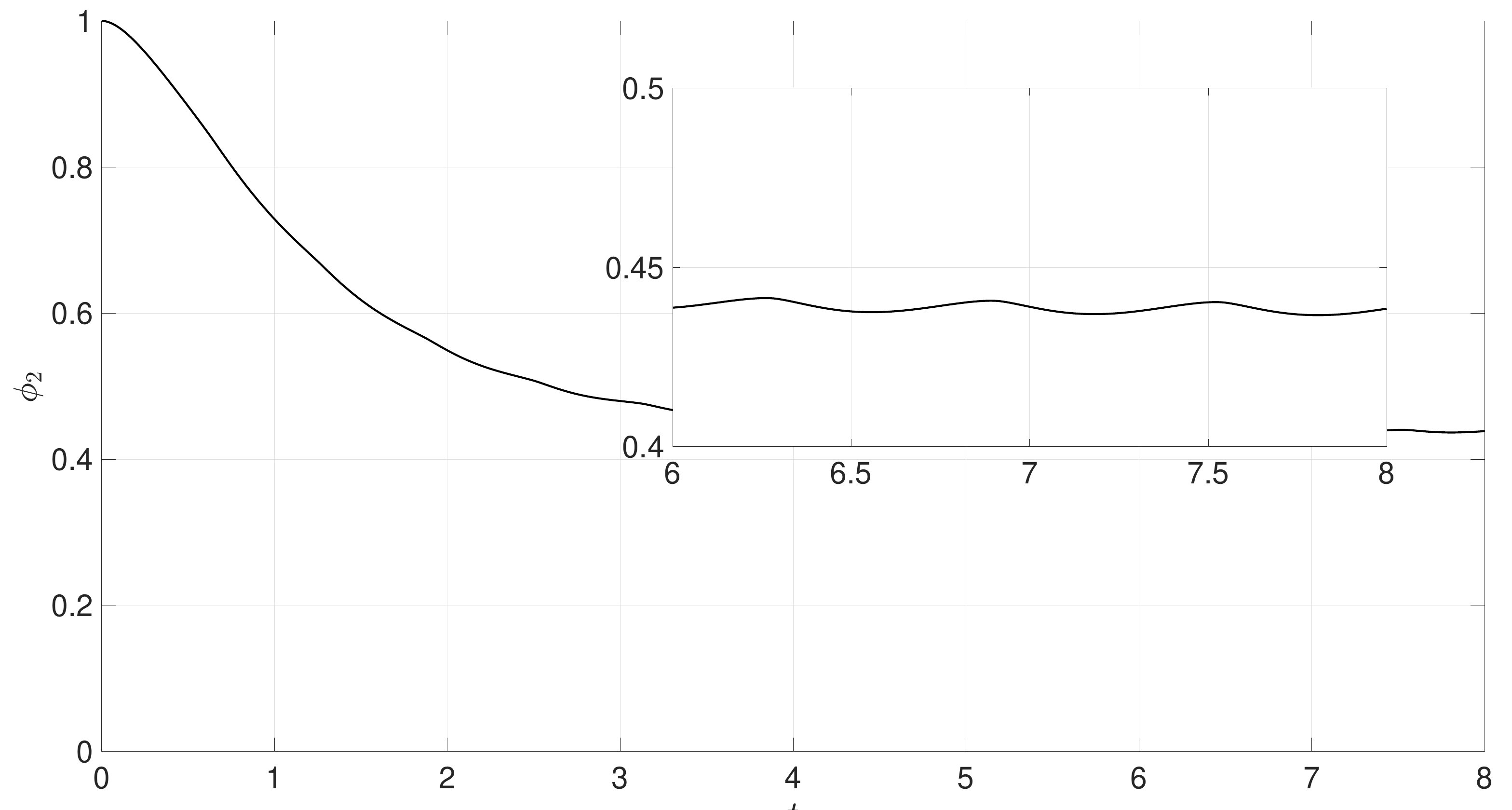}\quad \includegraphics[width=53mm]{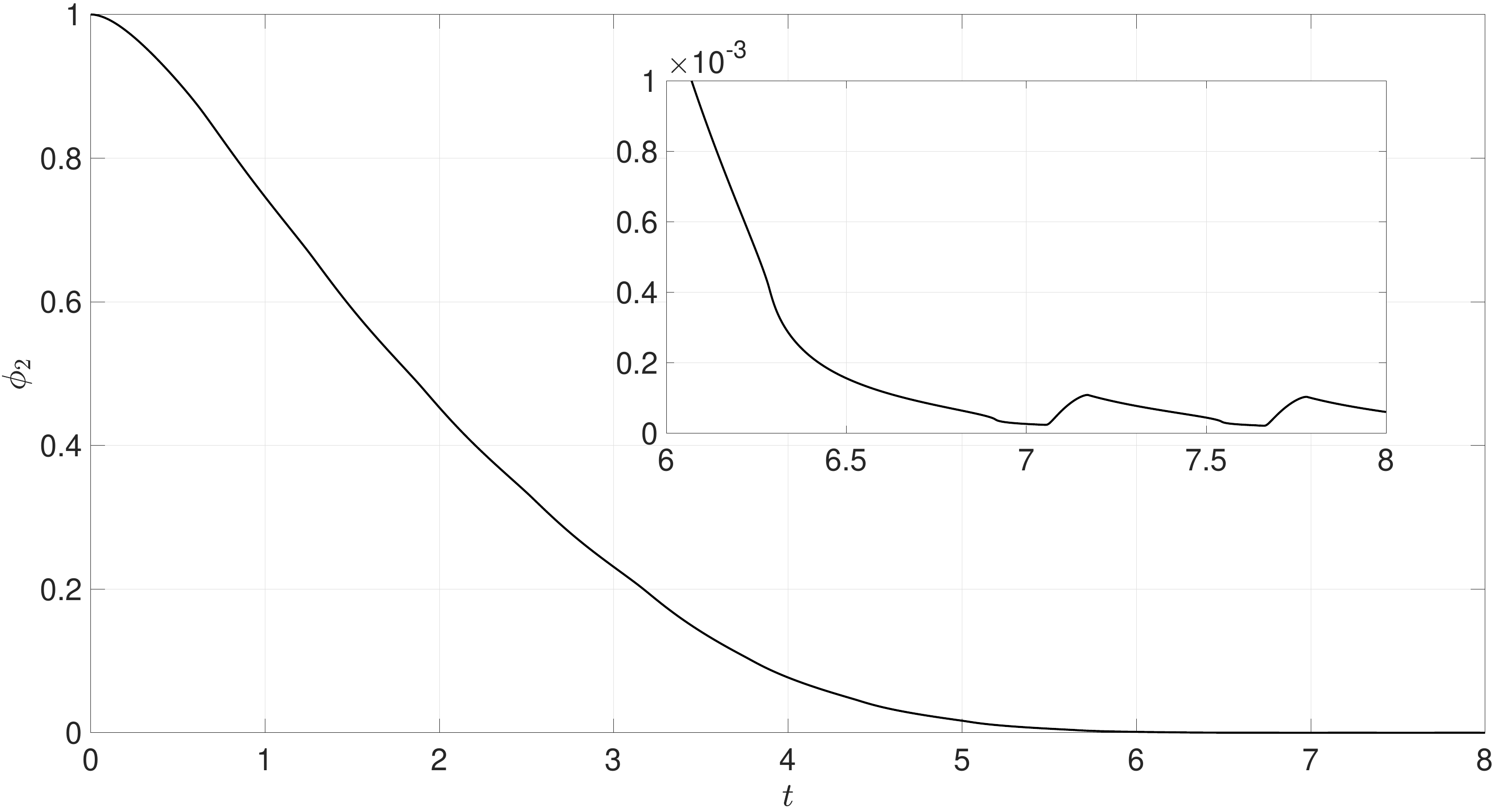}\\
			\vspace{-1mm}
		\caption{Trajectories of the perturbed system \eqref{eq:f(x,q)_example} with linear (left) and homogeneous (right) nonovershooting stabilizers}
		\label{fig:pert}
		\vspace{-1mm}
	\end{figure}

		\subsection{ISSf on $\Omega$}
			Let us consider the perturbed system
		\begin{equation}\label{eq:f(x,q)_example2}
			\dot x=f(x,q):=Ax+Bu(x+q_1)+q_2,  
		\end{equation}
		where $\mu>-1$, $x=(x_1,x_2,x_3)^{\top}$, $A$, $B$ are as before,  $u$ is either a linear or a homogeneous nonovershooting stabilizer designed above and  $q=(q_1^{\top},q_2^{\top})^{\top}\in L^{\infty}(\R,\R^6)$ defines both measurement noise and additive perturbation. 
		
		%Assumption  \ref{as:2} is fulfilled  for  
		%$\tilde \dn(s)=\left(
		%\begin{smallmatrix}
		% \dn(s) & \zero\\
		% \zero &e^{\mu s}\dn(s)
		%\end{smallmatrix}\right),s\in \R$ 
	%	(or with $\dn(s)=e^{s}I_6$ and $\mu=0$ in the case of linear stabilizer).
	%	 If the closed-loop system has unique solutions in the forward time then, in the view of Theorem \ref{thm:ISSfS}, to guarantee ISSf  of the considered system on $\Omega$ it is sufficient to check the condition \eqref{eq:ISSf_Omega_lin_0}:
	%	 \[
	%	 h_i^{\top} \!(A+BK)z\!>\! \tfrac{\mu h_i^{\top} G_{0}zz^{\top}P(A+BK) z}{z^{\top}PG_{\dn}z}
	%	 \]
	%	 for all $z\in \Xi_i(r) $. Making the change of variable, the latter inequality can be rewritten as 
	%	 \begin{equation}\label{eq:ISSf_example}
	%	 \inf_{y\in \tilde \Xi_i(r)}e_i^{\top}M(y)y>0
	%	 \end{equation}
	%	 where  $y\in \tilde \Xi_i(r)=\{y\in \R^3: y_i+r_i=0,y_j+r_j\geq 0, i\neq j\}, i=1,2,3$ and 
	%	 \[
	%	 M(y)=H(A+BK)H^{-1}+\gamma(y)
	%	 H (-G_{0})H^{-1}, \quad \gamma(y)=\tfrac{\mu y^{\top}H^{-\top} P(A+BK)H^{-1} y}{y^{\top}H^{-\top}PG_{\dn}H^{-1}y}.
	%	 \]
		 Using the representation \eqref{eq:temp2} we conclude that  
		 $-H(A+BK)H^{-1}r\in \R^n_+$ for $r=(0.2,1,0.2)^{\top}$. So, in the view of Corollary 
		 \ref{cor:ISSf_cont_system}, the perturbed system \eqref{eq:f(x,q)_example2} is ISSf and ISSfS, at least, for $\mu$ being sufficiently close to zero. The results of numerical simulations for the non-overshooting linear and homogeneous stabilizers designed above are given on Figure \ref{fig:ISSf}.
		 The measurement noise $q_1$ has been simulated as a uniformly distributed random variable with the magnitude $0.01$. 
		 The exogenous perturbation has been defined as follows $q_2(t)=\left(1,1,1 \right)^{\top}\sin(5t)$.
		 
		 The simulation results demonstrate faster convergence, better robustness and smaller overshoots of the homogeneous control system comparing with the linear one.
	\begin{figure}[t]
		\centering
		\includegraphics[width=53mm]{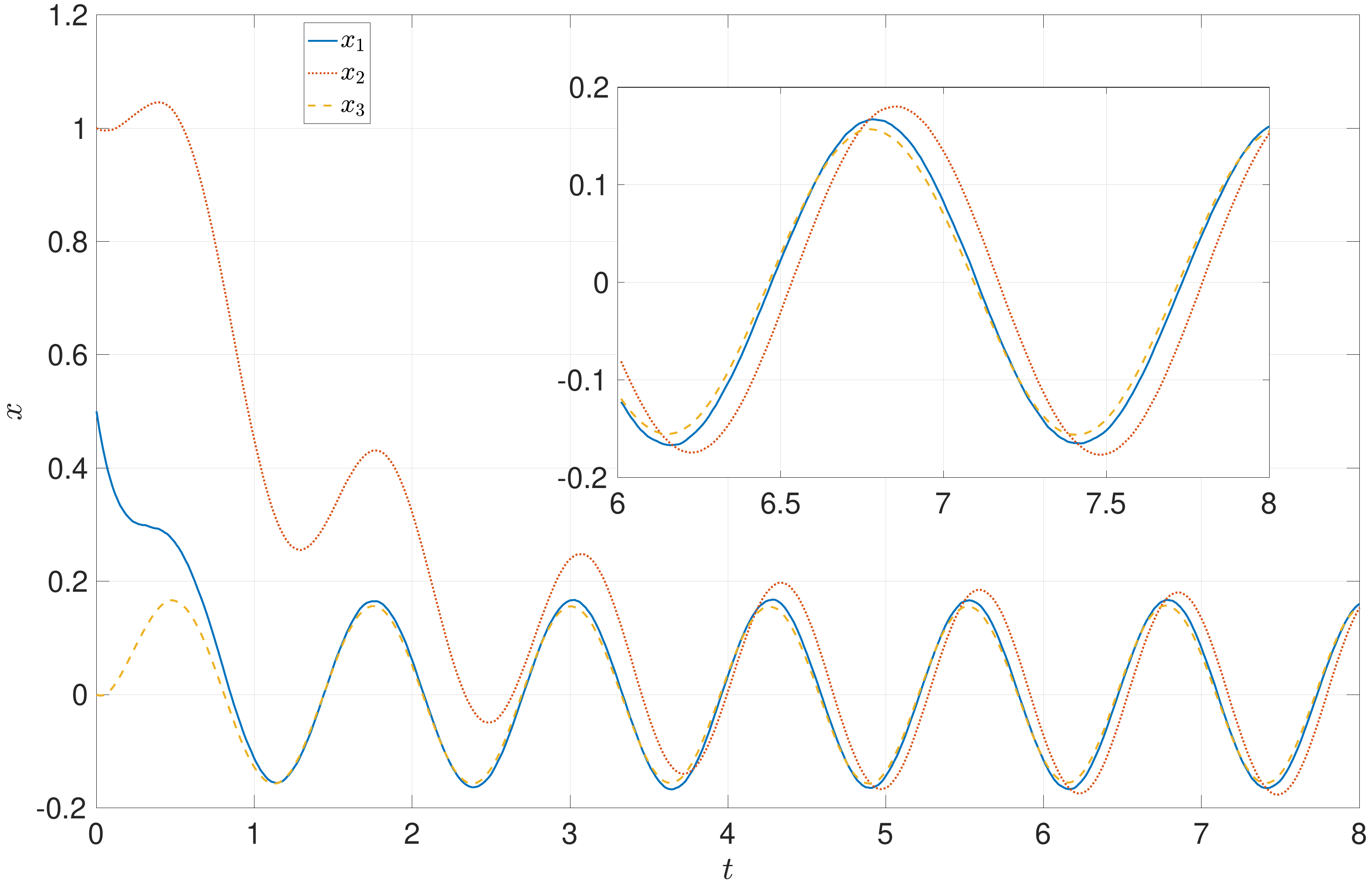}\quad \includegraphics[width=55mm]{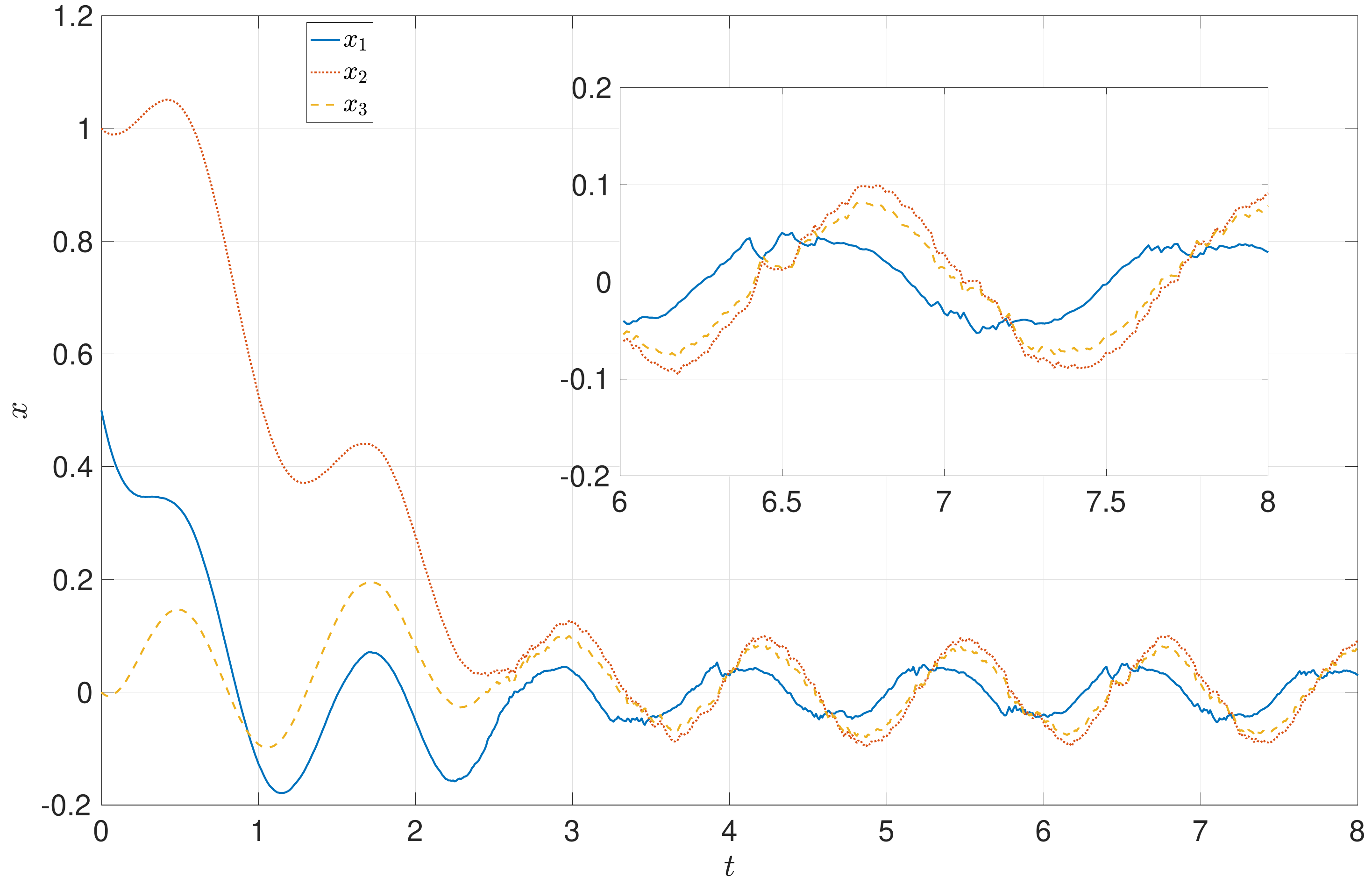}\\
		\includegraphics[width=53mm]{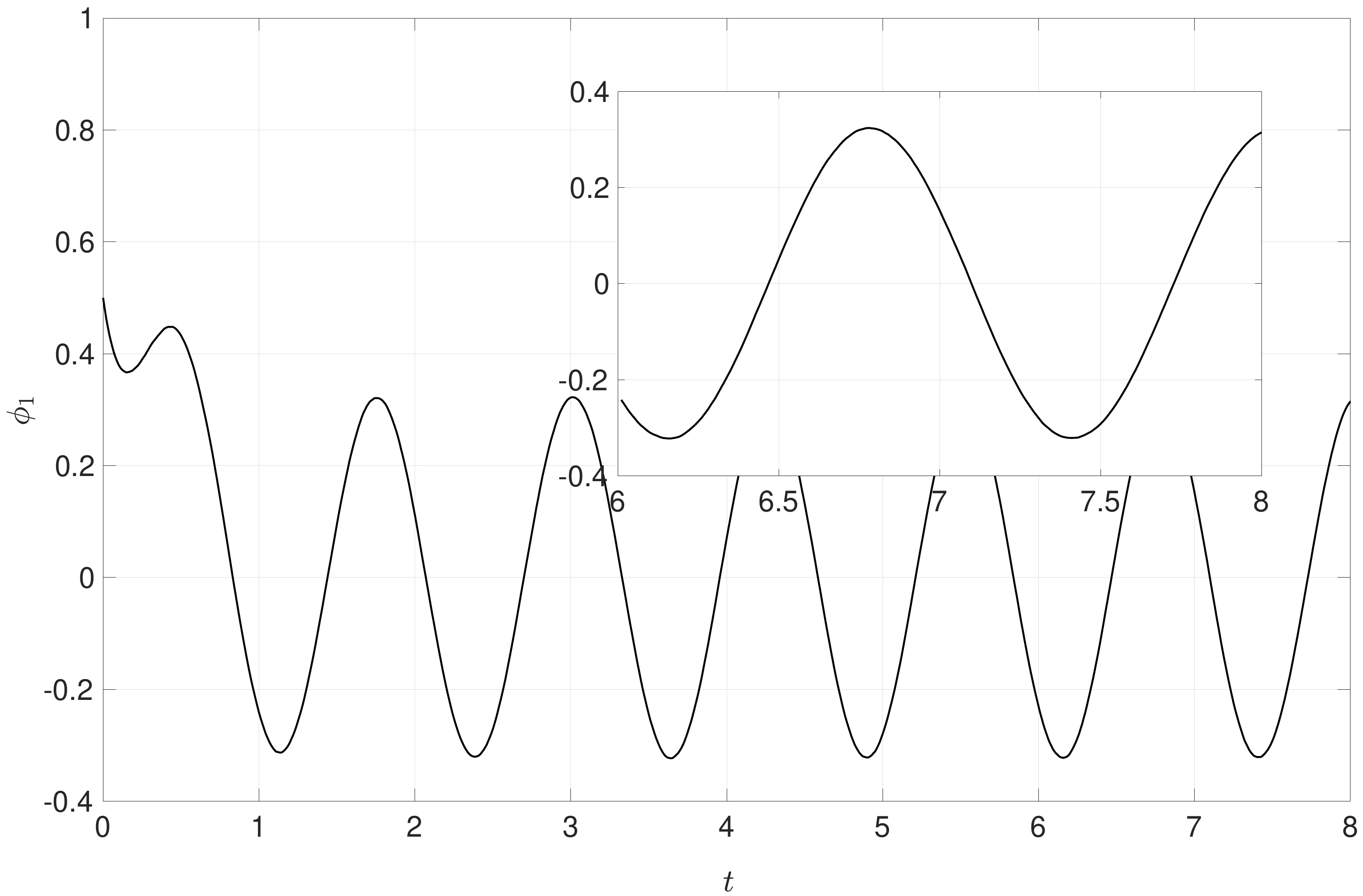}\quad \includegraphics[width=53mm]{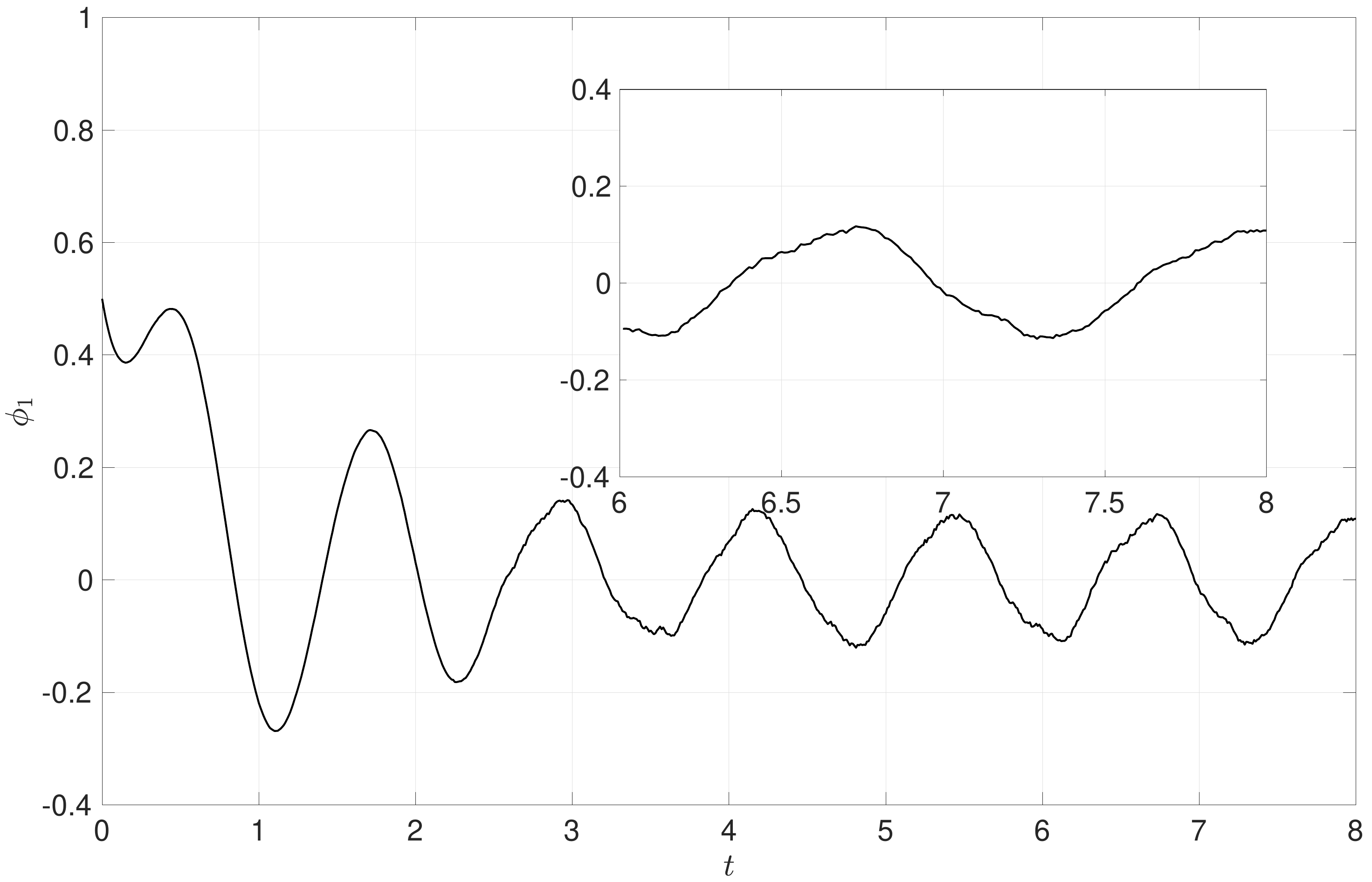}\\
		\includegraphics[width=53mm]{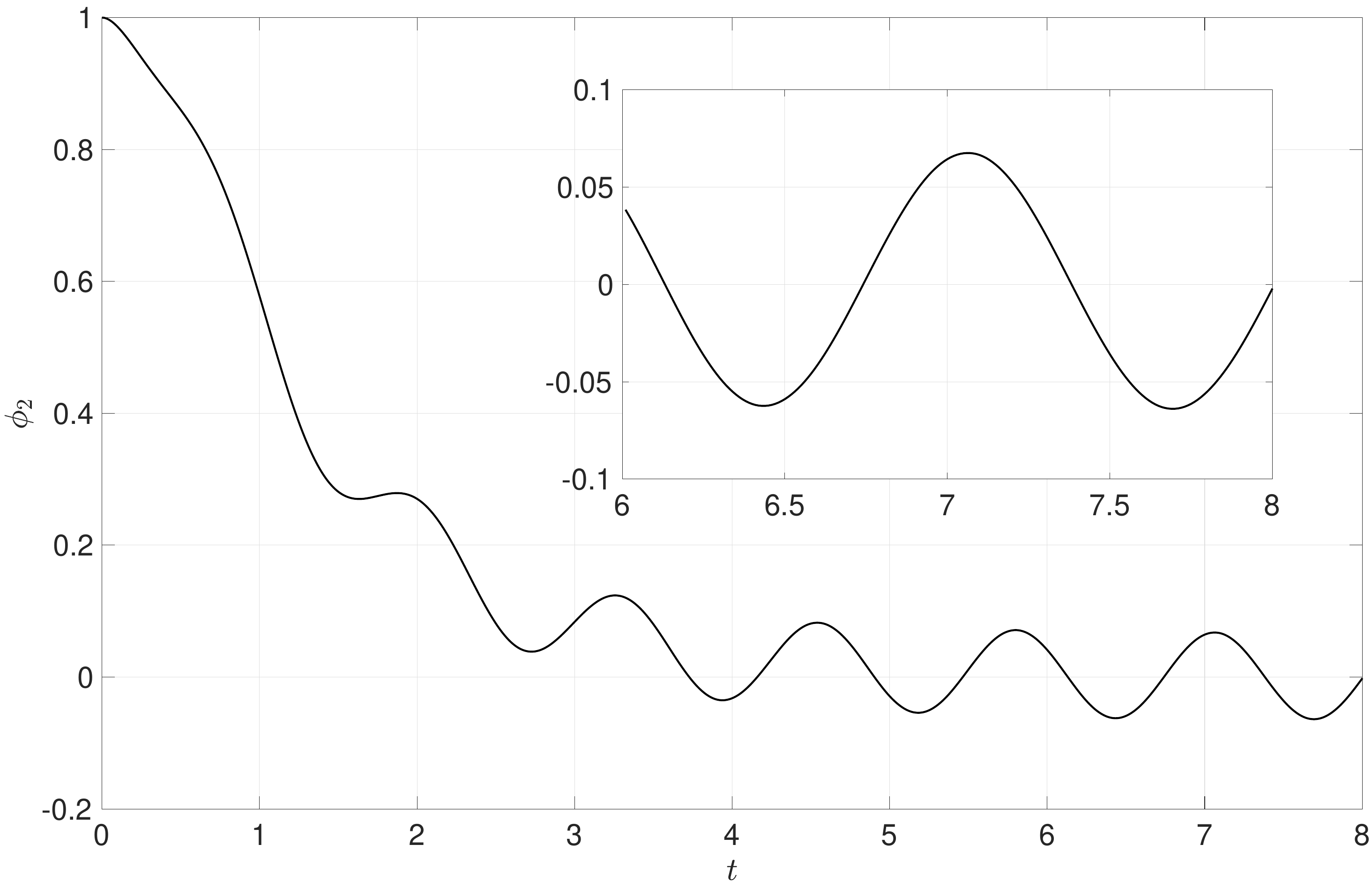}\quad \includegraphics[width=53mm]{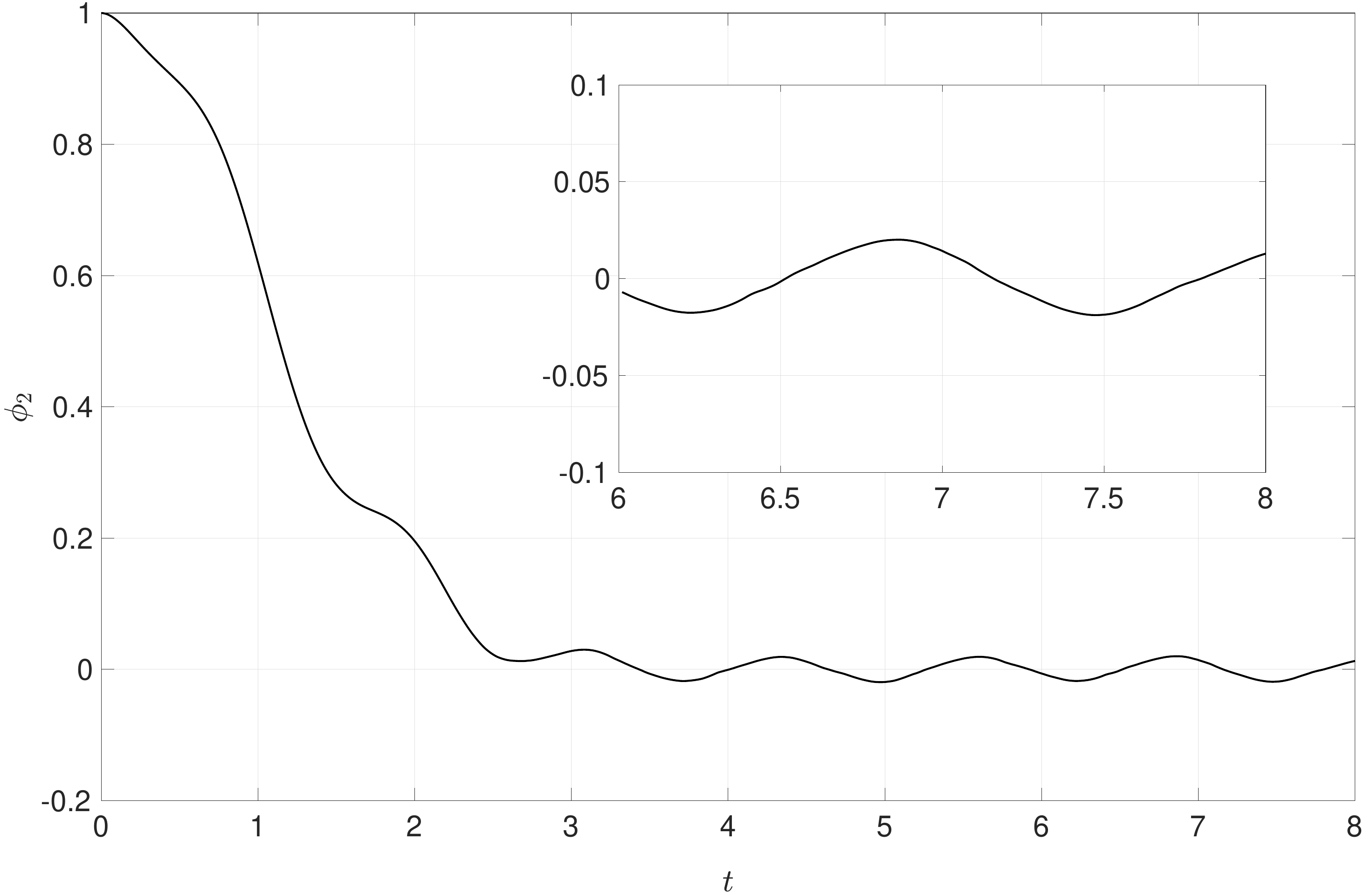}\\
			\vspace{-1mm}
		\caption{Trajectories of the pertubed system \eqref{eq:f(x,q)_example2} with linear (left) and homogeneous (right) nonovershooting stabilizers}
		\label{fig:ISSf}
		\vspace{-3mm}
	\end{figure}

	%%%%%%%%%%%%%%%%%%%%%%%%%%%%%%%%%%%%%%%%%%%%%%%%%%%%%%%%%%%%
	%%%%%%%%%%%%%%%%%%%%%%%%%%%%%%%%%%%%%%%%%%%%%%%%%%%%%%%%%%%%
	%%%%%%%%%%%%%%%%%%%%%%%%%%%%%%%%%%%%%%%%%%%%%%%%%%%%%%%%%%%%
	\section{Conclusions}
	In the paper, a scheme for nonovershooting finite-time stabilizer design for linear multi-input system is presented. The procedure is based on an upgrade of a linear nonovershooting stabilizer to  a  homogeneous one with negative degree. Robustness of the safety and finite-time stability  properties is analyzed using the concept of Input-to-State Safety and Input-to-State Stability. For this purpose, some known results about ISS analysis of homogeneous systems in $\R^n$ are expanded to homogeneous systems on homogeneous cones.  The presented example illustrates the simplicity of the proposed scheme of control design and robustness analysis. 

\bibliographystyle{plain}

\begin{thebibliography}{10}

\bibitem{Abel_etal2022:ACC}
I.~Abel, D.~Steeves, M.~Krstic, and M.~Jankovic.
\newblock Prescribed-time safety design for a chain of integrators.
\newblock In {\em American Control Conference}, 2022.

\bibitem{Ames_etal2015:CDC}
A.~D. Ames, J.~W. Grizzle, and P.~Tabuada.
\newblock Control barrier function based quadratic programs with application to
  adaptive cruise contro.
\newblock In {\em Conference on Decision and Control}, pages 6271--6278, 2014.

\bibitem{Ames_etal2017:TAC}
A.~D. Ames, J.~W. Grizzle, and P.~Tabuada.
\newblock Ccontrol barrier function based quadratic programs for safety
  critical systems.
\newblock {\em IEEE Transactions on Automatic Control}, 62:3861--3876, 2017.

\bibitem{Andrieu_etal2008:SIAM_JCO}
V.~Andrieu, L.~Praly, and A.~Astolfi.
\newblock {Homogeneous Approximation, Recursive Observer Design, and Output
  Feedback}.
\newblock {\em SIAM Journal of Control and Optimization}, 47(4):1814--1850,
  2008.

\bibitem{Aubin1991:Book}
J.-P. Aubin.
\newblock {\em Viability theory}.
\newblock Birkhauser, 1991.

\bibitem{Bernuau_etal2013:SCL}
E.~Bernuau, A.~Polyakov, D.~Efimov, and W.~Perruquetti.
\newblock Verification of {ISS}, {iISS} and {IOSS} properties applying weighted
  homogeneity.
\newblock {\em System \& Control Letters}, 62(12):1159--1167, 2013.

\bibitem{BhatBernstein2005:MCSS}
S.~P. Bhat and D.~S. Bernstein.
\newblock Geometric homogeneity with applications to finite-time stability.
\newblock {\em Mathematics of Control, Signals and Systems}, 17:101--127, 2005.

\bibitem{Blanchini1999:Aut}
F.~Blanchini.
\newblock Set invariance in control.
\newblock {\em Automatica}, 35:1747--1767, 1999.

\bibitem{BlanchiniMiani2016:Book}
F.~Blanchini and S.~Miani.
\newblock {\em Set-Theoretic Methods in Control}.
\newblock Birkhauser, 2016.

\bibitem{Boyd_etal1994:Book}
S.~Boyd, E.~Ghaoui, E.~Feron, and V.~Balakrishnan.
\newblock {\em Linear Matrix Inequalities in System and Control Theory}.
\newblock Philadelphia: SIAM, 1994.

\bibitem{CoddingtonLevinson1955:Book}
E.~A. Coddington and N.~Levinson.
\newblock {\em Theory of Ordinary Differential Equations}.
\newblock New York, McGraw-Hill, 1955.

\bibitem{DarbhaBhattacharrya2003:TAC}
S.~Darbha and S.~P. Bhattacharrya.
\newblock On the synthesis of controllers for a nonovershooting step response.
\newblock {\em IEEE Transactions on Automatic Control}, 48(5):797--799, 2003.

\bibitem{LeenheerAeyels2001:SCL}
P.~De~Leenheer and D.~Aeyels.
\newblock Stabilization of positive linear systems.
\newblock {\em Systems \& Control Letters}, 44:259--271, 2001.

\bibitem{ElKhoury_etal1993:Aut}
M.~El-Khoury, O.~D. Crisall, and R.~Longchamp.
\newblock Influence of zero locations on the number of step-response extre.
\newblock {\em Automatica}, 29:1571--1574, 1993.

\bibitem{FarinaRinaldi2000:Book}
L.~Farina and S.~Rinaldi.
\newblock {\em Positive Linear Systems; Theory and Applications}.
\newblock John Wiley, 2000.

\bibitem{Filippov1988:Book}
A.~F. Filippov.
\newblock {\em Differential Equations with Discontinuous Right-hand Sides}.
\newblock Kluwer Academic Publishers, 1988.

\bibitem{FischerRuzhansky2016:Book}
V.~Fischer and M.~Ruzhansky.
\newblock {\em {Quantization on Nilpotent Lie Groups}}.
\newblock Springer, 2016.

\bibitem{Garg_etal2022:IEEE_CSL}
K.~Garg, R.~K. Cosner, U.~Rosolia, A.~D. Ames, and D.~Panagou.
\newblock Multi-rate control design under input constraints via fixed-time
  barrier functions.
\newblock {\em IEEE Control Systems Letters}, 6:608--613, 2022.

\bibitem{Grune2000:SIAM_JCO}
L.~Gr{\"u}ne.
\newblock Homogeneous state feedback stabilization of homogeneous systems.
\newblock {\em SIAM Journal of Control and Optimization}, 38(4):1288--1308,
  2000.

\bibitem{Hong2001:Aut}
Y.~Hong.
\newblock H$_{\infty}$ control, stabilization, and input-output stability of
  nonlinear systems with homogeneous properties.
\newblock {\em Automatica}, 37(7):819--829, 2001.

\bibitem{Hong_etal2010:SIAM_JCO}
Y.~Hong, Z.~Jiang, and G.~Feng.
\newblock Finite-time input-to-state stability and applications to finite-time
  control design.
\newblock {\em SIAM Journal on Control and Optimization}, 48(7):4395--4418,
  2010.

\bibitem{Husch1970:Math_Ann}
L.S. Husch.
\newblock {Topological Characterization of The Dilation and The Translation in
  Frechet Spaces}.
\newblock {\em Mathematical Annals}, 190:1--5, 1970.

\bibitem{Jankovic2018:Aut}
M.~Jankovic.
\newblock Robust control barrier functions for constrained stabiliza- tion of
  nonlinear systems.
\newblock {\em Automatica}, 96:359--367, 2018.

\bibitem{Kawski1991:ACDS}
M.~Kawski.
\newblock Families of dilations and asymptotic stability.
\newblock {\em Analysis of Controlled Dynamical Systems}, pages 285--294, 1991.

\bibitem{Khomenuk1961:IVM}
V.~V. Khomenuk.
\newblock On systems of ordinary differential equations with generalized
  homogenous right-hand sides.
\newblock {\em Izvestia vuzov. Mathematica (in Russian)}, 3(22):157--164, 1961.

\bibitem{KolathayaAmes2019:CSL}
S.~Kolathaya and A.~D. Ames.
\newblock Input-to-state safety with control barrier functions.
\newblock {\em IEEE Control Systems Letters}, 3(1):108 -- 113, 2019.

\bibitem{KrsticBement2006:TAC}
M.~Krstic and M.~Bement.
\newblock Nonovershooting control of strict-feedback nonlinear systems.
\newblock {\em IEEE Transactions on Automatic Control}, 51(12):1938--1943,
  2006.

\bibitem{LindemannDimarogonas2019:IEEECSL}
L.~Lindemann and D.~V. Dimarogonas.
\newblock Control barrier functions for multi-agent systems under conflicting
  local signal temporal logic tasks.
\newblock {\em IEEE Control Systems Letters}, 3(3), 2019.

\bibitem{Magni_etal2006:TAC}
L.~Magni, D.M. Raimondo, and R.~Scattolini.
\newblock Regional input-to-state stability for nonlinear model predictive
  control.
\newblock {\em IEEE Transactions on Automatic Control}, 51(9):1548--1553, 2006.

\bibitem{Massera1949:AM}
J.~L. Massera.
\newblock On lyapunovff's conditions of stability.
\newblock {\em Annals of Mathematics}, 50:705--721, 1949.

\bibitem{Nagumno1942:PPMSJ}
M.~Nagumo.
\newblock Uber die lage der integralkurven gewohnlicher
  differentialgleichungen.
\newblock {\em Proceedings of the Physico-Mathematical Society of Japan},
  24:551--559, 1942.

\bibitem{Nakamura_etal2002:SICE}
H.~Nakamura, Y.~Yamashita, and H.~Nishitani.
\newblock Smooth {Lyapunov} functions for homogeneous differential inclusions.
\newblock In {\em Proceedings of the 41st SICE Annual Conference}, pages
  1974--1979, 2002.

\bibitem{Nekhoroshikh_etal2021:CDC}
A.~Nekhoroshikh, D.~Efimov, A.~Polyakov, W.~Perruquetti, and I.~Furtat.
\newblock Finite-time stabilization under state constraints.
\newblock In {\em Conference on Decision and Control}, 2021.

\bibitem{Pazy1983:Book}
A.~Pazy.
\newblock {\em Semigroups of Linear Operators and Applications to Partial
  Differential Equations}.
\newblock Springer, 1983.

\bibitem{PhillipsSeborg1988:IJC}
S.~F. Phillips and D.~E. Seborg.
\newblock Conditions that guarantee no overshoot for linear system.
\newblock {\em International Journal of Control}, 47(4):1043--1059, 1988.

\bibitem{Polyakov2018:RNC}
A.~Polyakov.
\newblock Sliding mode control design using canonical homogeneous norm.
\newblock {\em International Journal of Robust and Nonlinear Control},
  29(3):682--701, 2019.

\bibitem{Polyakov2020:Book}
A.~Polyakov.
\newblock {\em Generalized Homogeneity in Systems and Control}.
\newblock Springer, 2020.

\bibitem{Polyakov2021:Aut}
A.~Polyakov.
\newblock {Input-to-State Stability of Homogeneous Infinite Dimensional Systems
  with Locally Lipschitz Nonlinearities}.
\newblock {\em Automatica}, 129:109615, 2021.

\bibitem{PolyakovKrstic2023:TAC}
A.~Polyakov and M.~Krstic.
\newblock Finite-and fixed-time nonovershooting stabilizers and safety filters
  by homogeneous feedback.
\newblock {\em IEEE Transaction on Automatic Control}, 2023.

\bibitem{Poznyak_etal2014:Book}
A.~Poznyak, A.~Polyakov, and V.~Azhmyakov.
\newblock {\em Attractive Ellipsoids in Robust Control}.
\newblock Birkhauser, 2014.

\bibitem{Rahman_etal2021:ACC}
Y.~Rahman, M.~Jankovic, and M.~Santill.
\newblock Driver intent prediction with barrier functions.
\newblock In {\em American Control Conference}, pages 224--230, 2021.

\bibitem{Rantzer2015:EJC}
A.~Rantzer.
\newblock Scalable control of positive systems.
\newblock {\em European Journal of Control}, 24(7):72--80, 2015.

\bibitem{Rosier1992:SCL}
L.~Rosier.
\newblock Homogeneous {L}yapunov function for homogeneous continuous vector
  field.
\newblock {\em Systems \& Control Letters}, 19:467--473, 1992.

\bibitem{Ryan1995:SCL}
E.P. Ryan.
\newblock Universal stabilization of a class of nonlinear systems with
  homogeneous vector fields.
\newblock {\em Systems \& Control Letters}, 26:177--184, 1995.

\bibitem{SantilloJankovic2021:ACC}
M.~Santillo and M.~Jankovic.
\newblock Collision free navigation with interacting, non-communicating
  obstacles.
\newblock In {\em American Control Conference}, pages 1637--1643, 2021.

\bibitem{Sontag1989:TAC}
E.D. Sontag.
\newblock Smooth stabilization implies coprime factorization.
\newblock {\em IEEE Transactions on Automatic Control}, 34:435--443, 1989.

\bibitem{SontagWang1996:SCL}
E.D. Sontag and Y.~Wang.
\newblock On characterizations of the input-to-state stability property.
\newblock {\em Systems \& Control Letters}, 24(5):351--359, 1996.

\bibitem{Sontag1997:ECC}
E.~Soontag and Y.~Wang.
\newblock A notion of input to output stability.
\newblock In {\em European Control Conference}, pages 3862--3867, 1997.

\bibitem{Wang_etal2017:TR}
L.~Wang, A.D. Ames, and M.~Egerstedt.
\newblock Safety barrierc ertificates for collisions-free multirobot systems.
\newblock {\em IEEE Transactions on Robotics}, 33(3):661--674, 2017.

\bibitem{Wang_2020:RNC}
S.~Wang, A.~Polyakov, and G.~Zheng.
\newblock Generalized homogenization of linear controllers: Theory and
  experiment.
\newblock {\em International Journal of Robust and Nonlinear Control},
  31(9):3455--3479, 2021.

\bibitem{WielandAllgower2007:NOLCOS}
P.~Wieland and F.~Allg\"ower.
\newblock Constructive safety using control barrier functions.
\newblock In {\em IFAC Proceedings Volumes}, volume~40, pages 462--467,, 2007.

\bibitem{Zimenko_etal2020:TAC}
K.~Zimenko, A.~Polyakov, D.~Efimov, and W.~Perruquetti.
\newblock Robust feedback stabilization of linear mimo systems using
  generalized homogenization.
\newblock {\em IEEE Transactions on Automatic Control}, 2020.

\bibitem{Zubov1958:IVM}
V.I. Zubov.
\newblock On systems of ordinary differential equations with generalized
  homogeneous right-hand sides.
\newblock {\em Izvestia vuzov. Mathematica (in Russian)}, 1:80--88, 1958.

\end{thebibliography}

\end{document}